\documentclass[a4paper,12pt, reqno]{amsart}
\usepackage{latexsym}
\usepackage{color,dsfont}
\usepackage{xcolor}
\usepackage{amssymb,amsfonts,amsmath,mathrsfs}
\newtheorem{theorem}{Theorem}[section]
\newtheorem{proposition}{Proposition}[section]
\newtheorem{lemma}{Lemma}[section]
\newtheorem{corollary}{Corollary}[section]
\newtheorem{conjecture}{Conjecture}[section]
\newtheorem{remark}{Remark}[section]

\numberwithin{equation}{section}

\newcommand{\ord}{\text{ord}}

\newcommand{\hcom}[1]{{\color{cyan}{Hung: #1}} }
\newcommand{\kommentar}[1]{}
\newcommand{\acom}[1]{{\color{blue}{Alexandra: #1}} }

\begin{document}

\title[The Ratios Conjecture and upper bounds for negative moments]{The Ratios Conjecture and upper bounds for negative moments of $L$--functions over function fields}

\author{Hung M. Bui, Alexandra Florea and Jonathan P. Keating}
\address{Department of Mathematics, University of Manchester, Manchester M13 9PL, UK}
\email{hung.bui@manchester.ac.uk}
\address{Department of Mathematics, UC Irvine, Irvine CA 92617, USA}
\email{floreaa@uci.edu}
\address{Mathematical Institute, University of Oxford, Oxford OX2 6GG, UK}
\email{keating@maths.ox.ac.uk}

\begin{abstract}
We prove special cases of the Ratios Conjecture for the family of quadratic Dirichlet $L$--functions over function fields. More specifically, we study the average of $L(1/2+\alpha,\chi_D)/L(1/2+\beta,\chi_D)$, when $D$ varies over monic, square-free polynomials of degree $2g+1$ over $\mathbb{F}_q[x]$, as $g \to \infty$, and we obtain an asymptotic formula when $\Re \beta \gg g^{-1/2+\varepsilon}$. We also study averages of products of $2$ over $2$ and $3$ over $3$ $L$--functions, and obtain asymptotic formulas when the shifts in the denominator have real part bigger than $g^{-1/4+\varepsilon}$ and $g^{-1/6+\varepsilon}$ respectively. The main ingredient in the proof is obtaining upper bounds for negative moments of $L$--functions. The upper bounds we obtain are expected to be almost sharp in the ranges described above. As an application, we recover the asymptotic formula for the one-level density of zeros in the family with the support of the Fourier transform in $(-2,2)$.
\end{abstract}

\allowdisplaybreaks

\maketitle

\section{Introduction}

The Ratios Conjecture, formulated by Conrey, Farmer and Zirnbauer \cite{cfz}, is a wide-reaching conjecture with applications to a number of notoriously difficult questions in number theory. The conjectures in \cite{cfz}, which apply to different families of $L$--functions, generalize an earlier conjecture of Farmer \cite{farmer} for ratios of the Riemann zeta-function. Namely, for complex numbers $\alpha, \beta, \gamma, \delta$ with positive real parts of size $c/\log T$, Farmer conjectured that
$$ \frac{1}{T} \int_0^T \frac{\zeta(s+\alpha) \zeta(1-s+\beta)}{ \zeta(s+\gamma) \zeta(1-s+\delta)} \, dt \sim \frac{ (\alpha+\delta)(\beta+\gamma)}{(\alpha+\beta)(\gamma+\delta)}-T^{-(\alpha+\beta)} \frac{ (\alpha-\gamma)(\beta-\delta)}{(\alpha+\beta)(\gamma+\delta)} .$$
The conjecture above would imply several results on zeros of the Riemann zeta-function, such as Montgomery's pair correlation conjecture \cite{montgomery}.

The conjectures of Conrey, Farmer and Zirnbauer generalize Farmer's conjecture to quotients of products
of an arbitrary number of $L$--functions averaged over a family, and include precise lower-order terms down to a power-saving error term. 
To obtain these conjectures, the authors extend the ``recipe'' of Conrey, Farmer, Keating, Rubinstein and Snaith \cite{CFKRS} which was used to predict asymptotic formulas for moments in families of $L$--functions. In \cite{cfz}, a comparison is made to the analogous quantities for the
characteristic polynomials of matrices averaged over classical compact groups. Since it is believed that families of $L$--functions can be modeled by the characteristic polynomials of matrices from one of the classical compact groups, it is of interest to consider the analogous questions of computing the ratios of characteristic polynomials in random matrix theory. These questions have been settled for random matrices, as in \cite{zirn, forester, borodin}.

The Ratios Conjecture has wide applicability to many questions in number theory related to zeros of $L$--functions, as detailed in \cite{cs}. For example, one can use the Ratios Conjecture to compute the one-level density of zeros in families of $L$--functions, including lower-order terms, or to compute lower order terms in the pair correlation of zeros of the Riemann zeta-function, which were originally conjectured by Bogomolny and Keating \cite{bk}. 

While in the case of moments of $L$--functions, one can usually compute a few moments in each family (see, for example, \cite{H-B, sound23, kmv, young} for results in various families), there are no rigorous results in the literature on the Ratios Conjecture, as far as the authors are aware (there is some forthcoming work of Cech \cite{cech} using multiple Dirichlet series which addresses the Ratios Conjecture in certain ranges of the parameters). 
One difficulty when computing averages of quotients of $L$--functions is the fact that for small shifts in the denominator, one is very close to possible zeros of the $L$--function. The closer we are to the critical line, the more difficult the problem gets.

The question of obtaining asymptotic formulas for quotients of $L$--functions is closely related to that of obtaining upper bounds for negative moments of $L$--functions. 

In the case of the Riemann zeta-function $\zeta(s)$, a conjecture due to Gonek \cite{gonek} states the following. 
\begin{conjecture}[Gonek] \label{gonek_conj}
Let $k>0$ be fixed. Uniformly for $1 \leq \delta \leq \log T$,
$$ \frac{1}{T}\int_1^T \Big| \zeta \Big(\frac{1}{2}+ \frac{\delta}{\log T}+it \Big) \Big|^{-2k} \, dt \asymp  \Big(\frac{\log T}{\delta} \Big)^{k^2}, $$
and uniformly for $0 < \delta \leq 1$,
$$
\frac{1}{T}\int_1^T \Big| \zeta \Big(\frac{1}{2}+ \frac{\delta}{\log T}+it \Big) \Big|^{-2k} \, dt \asymp 
\begin{cases}
(\log T)^{k^2} & \mbox{ if } k<1/2, \\
(\log\frac{e}{\delta}) (\log T)^{k^2} & \mbox{ if } k=1/2, \\
\delta^{1-2k} (\log T)^{k^2} & \mbox{ if } k>1/2.
\end{cases}
$$
\end{conjecture}
However, random matrix theory computations due to Berry and Keating \cite{berryk} and Forrester and Keating \cite{fk} suggest extra transition regimes in the conjecture above when $0<\delta \leq 1$ and $k=(2n+1)/2$ for $n$ a positive integer.
While obtaining lower bounds is a more tractable problem (Gonek \cite{gonek} proved lower bounds of the right order of magnitude in certain ranges for $k$ and $\delta$ on the Riemann Hypothesis), there has not been any progress on the corresponding upper bounds. We will prove partial results towards the analogue of this conjecture in the function field setting. 


Over function fields, Andrade and Keating \cite{ak} adapted the ``recipe''  to make conjectures for ratios of products of $L$--functions associated to quadratic characters over $\mathbb{F}_q[x]$. We will explicitly write down the conjecture in the case of one $L$--function over one $L$--function. Let $\mathcal{H}_{2g+1}$ denote the ensemble of monic, square-free polynomials of degree $2g+1$ over $\mathbb{F}_q[x]$, and let $\chi_D$ denote the quadratic character of modulus $D$. 
For 
\begin{equation} 
| \Re \alpha| < \frac{1}{4}  , \, \, \, \frac{1}{g} \ll \Re \beta< \frac{1}{4},
\label{condrc}
\end{equation} we expect
\begin{align*}
&\frac{1}{|\mathcal{H}_{2g+1}|}  \sum_{D \in \mathcal{H}_{2g+1}} \frac{ L(1/2+\alpha,\chi_D)}{L(1/2+\beta,\chi_D)}\sim A(\alpha,\beta) \frac{\zeta_q(1+2\alpha)}{\zeta_q(1+\alpha+\beta)} + q^{-2g \alpha}A(-\alpha,\beta)  \frac{\zeta_q(1-2\alpha)}{\zeta_q(1-\alpha+\beta)}
\end{align*}
with some power saving error term, where $\zeta_q$ denotes the zeta-function over $\mathbb{F}_q[x]$ and 
$$A(\alpha,\beta) = \prod_{P\in\mathcal{P}} \bigg(1-\frac{1}{|P|^{1+\alpha+\beta}} \bigg)^{-1} \bigg( 1-\frac{1}{|P|^{\alpha+\beta}(|P|+1)}-\frac{1}{|P|^{1+2\alpha}(|P|+1)} \bigg).$$
Here and throughout the paper, we denote $\mathcal{P}$ to be the ensemble of monic, irreducible polynomials.

More generally, let ${\bf A}=\{\alpha_1,\ldots,\alpha_k\}$ and ${\bf B}=\{\beta_1,\ldots,\beta_k\}$. We are interested in
\begin{equation}
\frac{1}{|\mathcal{H}_{2g+1}|}\sum_{D\in \mathcal{H}_{2g+1}}\frac{\prod_{j=1}^kL(1/2+\alpha_j,\chi_D)}{\prod_{j=1}^kL(1/2+\beta_j,\chi_D)},
\label{to_study}
\end{equation}
where
\begin{equation} |\Re \alpha_j| < \frac{1}{4} , \, \, \frac{1}{g} \ll \Re \beta_j < \frac{1}{4}, \, \, \text{ for } 1\leq j\leq k.
\label{cond_param}
\end{equation}
We note that the conditions on the real parts of the parameters ensure convergence of the Euler products expected in the asymptotic formulas, but it is possible to formulate the conjecture in a wider range, as long as the Euler products are convergent.

If ${\bf C}=\{\gamma_1,\ldots,\gamma_k\}$, then we denote
\[
{\bf C}^-=\{-\alpha:\alpha\in {\bf C}\},\qquad q^{-2g\bf C}=q^{-2g\sum_{j=1}^k\gamma_j},
\]
\begin{equation}
\mu_{\bf C}(f)=\prod_{f=f_1\ldots f_k}\frac{\mu(f_1)\ldots\mu(f_k)}{|f_1|^{\gamma_1}\ldots |f_k|^{\gamma_k}}\qquad\text{and}\qquad\tau_{\bf C}(f)=\prod_{f=f_1\ldots f_k}\frac{1}{|f_1|^{\gamma_1}\ldots |f_k|^{\gamma_k}}, \label{notation}
\end{equation}
where $\mu(f)$ is the M\"{o}bius function over $\mathbb{F}_q[x]$. We rewrite the Ratios Conjecture in the following form.
\begin{conjecture}[Ratios Conjecture]\label{ratios}
Under the constraints \eqref{cond_param}, we have
\begin{align*}
&\frac{1}{|\mathcal{H}_{2g+1}|}\sum_{D\in \mathcal{H}_{2g+1}}\frac{\prod_{j=1}^kL(1/2+\alpha_j,\chi_D)}{\prod_{j=1}^kL(1/2+\beta_j,\chi_D)}\sim\sum_{\bf R\subset\bf A}q^{-2g\bf R}\mathcal{S}_{\bf (A\backslash\bf R)\cup\bf R^-}
\end{align*}
with some power saving error term, where  if ${\bf C}=\{\gamma_1,\ldots,\gamma_k\}$, then
\begin{align*}
&\mathcal{S}_{\bf C}=\frac{\prod_{1\leq i\leq j\leq k}\zeta_q(1+\gamma_i+\gamma_j)\prod_{1\leq i< j\leq k}\zeta_q(1+\beta_i+\beta_j)}{\prod_{1\leq i, j\leq k}\zeta_q(1+\beta_i+\gamma_j)}\\
&\quad\times\prod_{P\in\mathcal{P}}\prod_{1\leq i\leq j\leq k}\bigg(1-\frac{1}{|P|^{1+\gamma_i+\gamma_j}}\bigg)\prod_{1\leq i< j\leq k}\bigg(1-\frac{1}{|P|^{1+\beta_i+\beta_j}}\bigg)\prod_{1\leq i, j\leq k}\bigg(1-\frac{1}{|P|^{1+\beta_i+\gamma_j}}\bigg)^{-1}\\
&\quad\times\prod_{P\in\mathcal{P}}\bigg(1+\bigg(1+\frac{1}{|P|}\bigg)^{-1}\sum_{i+j\geq 2\ even}\frac{\mu_{\mathbf{B}}(P^i)\tau_{\mathbf{C}}(P^j)}{|P|^{(i+j)/2}}\bigg).
\end{align*}
\end{conjecture}

We will prove the conjecture above for $|\bf{A}| =|\bf{B}|$$\ \leq 3$, in certain ranges of the parameters. The challenge in evaluating \eqref{to_study} lies in being able to choose the shifts in the denominator as small as possible. The smaller the shifts in the denominator are, the closer we are to possible zeros of the $L$--functions in the denominator, and hence the harder the problem gets. More precisely, we will prove the following result.

\begin{theorem}\label{theorem1}
Let $0<\Re \beta_j<1/2 $ for $1\leq j\leq k$. We denote $\alpha= \max \{|\Re \alpha_1|,\ldots, |\Re \alpha_k| \}$ and $\beta=\min \{\Re \beta_1,\ldots, \Re \beta_k\}$. Then Conjecture \ref{ratios} holds
for $1\leq k\leq 3$ with the error term $E_k$, where
\begin{align}\label{first_part}
E_1\ll_\varepsilon
\begin{cases}
q^{-g \beta( 3+2  \alpha)+\varepsilon g\beta} & \mbox{ if } 0 \leq \Re\alpha_1 <1/2  \mbox{ and }\ \beta\gg g^{-1/2+\varepsilon}, \\
q^{-g  \beta( 3-4  \alpha)+\varepsilon g\beta} & \mbox{ if } -1/2<\Re\alpha_1<0\ \mbox{and}\ \beta\gg g^{-1/2+\varepsilon},
\end{cases} 
\end{align}
and
\begin{align}
E_2&\ll_\varepsilon q^{-g \beta \min \{ \frac{1-4\alpha}{1+\beta} , \frac{1-2\alpha}{2+\beta} \}+\varepsilon g\beta} & \mbox{ if } \alpha< 1/4 \text{ and } \beta\gg g^{-1/4+\varepsilon}, \label{second_part}\\
E_3&\ll_\varepsilon q^{-g \beta \min \{ \frac{1/4-4\alpha}{\beta} , \frac{1/2-4\alpha}{3+\beta} \}+\varepsilon g\beta} & \mbox{ if } \alpha< 1/16\ \mbox{and}\ \beta\gg g^{-1/6+\varepsilon}. \label{third_part}
\end{align}
\end{theorem}

\kommentar{\begin{theorem}\label{theorem1}
Let $|\Re \alpha_j| < 1$ and $0<\Re \beta_j < 1$ for $1\leq j\leq k$. We denote $\alpha= \max \{|\Re \alpha_1|,\ldots, |\Re \alpha_k| \}$ and $\beta=\min \{\Re \beta_1,\ldots, \Re \beta_k\}$. Then Conjecture \ref{ratios} holds
for $1\leq k\leq 3$ with the error term $E_k$, where
\begin{align}\label{first_part}
E_1\ll_\varepsilon
\begin{cases}
q^{-g \beta( 3+2  \alpha)+\varepsilon g\beta} & \mbox{ if } \alpha_1 \geq 0\ \mbox{and}\ \beta\gg g^{-1/2+\varepsilon}, \\
q^{-g  \beta( 3-4  \alpha)+\varepsilon g\beta} & \mbox{ if } \alpha_1<0\ \mbox{and}\ \beta\gg g^{-1/2+\varepsilon},
\end{cases} 
\end{align}
and
\begin{align}
E_2&\ll_\varepsilon q^{-g \beta \min \{ \frac{1-4\alpha}{1+\beta} , \frac{1-2\alpha}{2+\beta} \}+\varepsilon g\beta}\quad \mbox{ if } \alpha< 1/4\ \mbox{and}\ \beta\gg g^{-1/4+\varepsilon}, \label{second_part}\\
E_3&\ll_\varepsilon q^{-g \beta \min \{ \frac{1/4-4\alpha}{\beta} , \frac{1/2-4\alpha}{3+\beta} \}+\varepsilon g\beta}\quad \mbox{ if } \alpha< 1/16\ \mbox{and}\ \beta\gg g^{-1/6+\varepsilon}. \label{third_part}
\end{align}
\end{theorem}
}
To evaluate \eqref{to_study}, we write the $L$--functions in the denominator as Dirichlet series involving the M\"{o}bius function, and then truncate the Dirichlet series thus obtained. We will write the first piece of the series as 
 \begin{equation}
\sum_{d(h_1), \ldots, d(h_k) \leq X } \frac{\prod_{j=1}^k \mu(h_j)}{ \prod_{j=1}^k |h_j|^{1/2+\beta_j}}  \sum_{D \in \mathcal{H}_{2g+1}} \bigg(\prod_{j=1}^kL(\tfrac12+\alpha_j,\chi_D)\bigg)   \chi_D\bigg(\prod_{j=1}^k h_j\bigg) ,
 \label{initial1}
 \end{equation}
for some parameter $X$, where $d(h)$ denotes the degree of the polynomial $h$, and we will prove asymptotic formulas for twisted, shifted moments of $L$--functions, generalizing the work in \cite{BF}. This will be the content of Theorem \ref{theorem2}. We note that we could improve the range of the parameters $\alpha_j$ in the statement of Theorem \ref{theorem1} for $k=2,3$ by keeping the error terms in Theorem \ref{theorem2} explicit and then making use of the cancellation provided by the M\"{o}bius function in the formula above. We have decided not to do that to keep the paper at a reasonable length, and to focus instead on the more difficult task of evaluating the  second piece in the Dirichlet series, corresponding to at least one of the polynomials $h_j$ in \eqref{initial1} having degree bigger than $X$.  

 We express the second piece in terms of an integrated ratio of $L$--functions, and then prove upper bounds for negative moments of $L$--functions. Our proof builds on work of Soundararajan \cite{sound_ub} who proved almost sharp upper bounds for the positive moments of the Riemann zeta-function conditional on the Riemann Hypothesis, and of Harper \cite{harper} who refined the method, obtaining sharp upper bounds for the positive moments. However, no upper bounds are known for the negative moments.

In our work, in certain ranges of the parameters, the upper bounds obtained are expected to be sharp (up to a log factor), partially proving the analogue of Gonek's Conjecture \ref{gonek_conj} in this setting. Obtaining upper bounds for negative moments of $L$--functions relies on using some sieve-theoretic inspired ideas which have been recently successfully used in a number of settings \cite{sound_maks, lester_radziwill, DFL, heap_sound, BELP}. As far as we are aware, our work provides the first upper bounds for negative moments of $L$--functions, when the shift in the $L$--function goes to zero. In forthcoming work \cite{flub}, upper bounds for negative moments of $L$--functions with shifts smaller than those considered in this paper are obtained.

Proving Theorem \ref{theorem1} relies on proving the following two results, which might be of independent interest.

\begin{theorem}\label{theorem2}
 Let  $h=h_1h_2^2$ with $d(h)\ll g$ and $h_1$ a square-free monic polynomial. For $\alpha=\max \{|\Re\alpha_1|,\ldots,|\Re\alpha_k|\}<1/2$  we have
\begin{align*}
&\frac{1}{|\mathcal{H}_{2g+1}|}\sum_{D\in \mathcal{H}_{2g+1}}\bigg(\prod_{j=1}^kL(\tfrac 12+\alpha_j,\chi_D)\bigg)\chi_D(h)=\frac{1}{\sqrt{|h_1|}}\sum_{\bf R\subset\bf A}q^{-2g\bf R}\widetilde{\mathcal{S}}_{\bf (A\backslash\bf R)\cup\bf R^-}(h)+\widetilde{E}_k.
\end{align*}
Here if ${\bf C}=\{\gamma_1,\ldots,\gamma_k\}$, then
\[
\widetilde{\mathcal{S}}_{\bf C}(h)=\mathcal{A}_{{\bf C}}(1)\mathcal{B}_{{\bf C}}(h;1)\prod_{1\leq i\leq j\leq k}\zeta_q(1+\gamma_i+\gamma_j),
\]
where
\begin{align}\label{formulaA}
\mathcal{A}_{{\bf C}}(u)&=\prod_{P\in\mathcal{P}}\prod_{1\leq i\leq j\leq k}\bigg(1-\frac{u^{2d(P)}}{|P|^{1+\gamma_i+\gamma_j}}\bigg)\prod_{P\in\mathcal{P}}\bigg(1+\bigg(1+\frac{1}{|P|}\bigg)^{-1}\sum_{j=1}^{\infty}\frac{\tau_{\bf C}(P^{2j})}{|P|^j}u^{2jd(P)}\bigg)
\end{align}
and
\begin{align}\label{formulaB}
\mathcal{B}_{{\bf C}}(h;u)&=\prod_{P|h}\bigg(1+\frac{1}{|P|}+\sum_{j=1}^{\infty}\frac{\tau_{\bf C}(P^{2j})}{|P|^j}u^{2jd(P)}\bigg)^{-1}\nonumber\\
&\qquad\quad\times\prod_{P|h_1}\bigg(\sum_{j=0}^{\infty}\frac{\tau_{\bf C}(P^{2j+1})}{|P|^j}u^{2jd(P)}\bigg)\prod_{\substack{P\nmid h_1\\P|h_2}}\bigg(\sum_{j=0}^{\infty}\frac{\tau_{\bf C}(P^{2j})}{|P|^j}u^{2jd(P)}\bigg).
\end{align}
Also,
\begin{align}\label{formulaE1}
\widetilde{E}_1&=E_{\alpha_1}(h;g)+q^{-2g\alpha_1}E_{-\alpha_1}(h;g-1)\nonumber\\
&\qquad\qquad+O_\varepsilon\big( |h|^{1/2}q^{-(3/2-\alpha)g+\varepsilon g}\big)+O_\varepsilon\big(|h_1|^{1/4}q^{-(3/2-2\alpha) g+\varepsilon g}\big),
\end{align}
where $E_{\gamma_1}(h;N)$ is given explicitly in \eqref{3002} and satisfies
\[
E_{\gamma_1}(h;N)\asymp |h_1|^{1/6}q^{-4g/3-g\Re\gamma_1+\varepsilon}+|h_1|^{1/6+\Re\gamma_1/3}q^{-4g/3-2g\Re\gamma_1/3+\varepsilon}
\]
in particular, and
\begin{align}\label{formulaE2}
\begin{cases}
\widetilde{E}_2\ll_\varepsilon |h|^{1/2}q^{-(1-2\alpha) g+\varepsilon g}+q^{-(1-4\alpha) g+\varepsilon g},\\
\widetilde{E}_3\ll_\varepsilon |h|^{1/2}q^{-(1/2-4\alpha) g+\varepsilon g}+q^{-(1-6\alpha) g+\varepsilon g}+|h_{1}|^{-3/4}q^{-(1/4-4\alpha) g+\varepsilon g}.
\end{cases}
\end{align}
\end{theorem}

\begin{theorem}\label{theorem3}
Let $k$ be a positive integer and $m>0$ such that $2km >1$. Let $0<\Re \beta_j <1/2$ for $1\leq j\leq k$. For $\beta=\min \{\beta_1,\ldots,\beta_k\} \gg    g^{ -\frac{1}{2 km}+\varepsilon} $, we have
\begin{align*}
\frac{1}{|\mathcal{H}_{2g+1}|} & \sum_{D\in \mathcal{H}_{2g+1}}  \prod_{j=1}^k \frac{1}{ | L(1/2+\beta_j+it_j,\chi_D)|^m}\\
& \ll 
  \Big( \frac{1}{\beta} \Big)^{k^2m^2/2} \prod_{j=1}^k \min \Big\{ \frac{1}{\beta_j}, \frac{1}{\overline{t_j}}\Big\}^{-m/2}   (\log g)^{km(km+1)/2},
\end{align*}
where $\overline{t} = \min \{ t \bmod {2 \pi}, 2 \pi - (t \bmod {2 \pi})\} .$ 
\end{theorem}

We expect that the upper bound we obtain in the theorem above is almost sharp, up to the log factor.

As a particular case of Theorem \ref{theorem3}, we obtain the following corollary. 

\begin{corollary}
Let $m>1/2$. Let $0<\beta <1/2$ such that $\beta \gg    g^{ -\frac{1}{2 m}+\varepsilon} $. Then we have
\begin{align*}
\frac{1}{|\mathcal{H}_{2g+1}|}  \sum_{D\in \mathcal{H}_{2g+1}} &  \frac{1}{ | L(1/2+\beta,\chi_D)|^m} \ll 
\Big( \frac{1}{\beta} \Big)^{m(m-1)/2}  (\log g)^{m(m+1)/2}.
\end{align*}
\end{corollary} 

As an application to the Ratios Conjecture, we also compute the one-level density of zeros in the family. Let $\phi(\theta) = \sum_{|n| \leq N} \widehat{\phi}(n) e(n \theta)$ (with $e(x) = e^{2 \pi i x}$) be a real, even trigonometric polynomial, and let $\Phi(2g \theta) = \phi(\theta)$. The one-level density of zeros is defined to be 
\begin{equation}
\Sigma (\Phi,g) := \frac{1}{|\mathcal{H}_{2g+1}|} \sum_{D \in \mathcal{H}_{2g+1}} \sum_{j=1}^{2g} \Phi(2g \theta_{j,D}),
\label{density}
\end{equation}
where $\theta_{j,D}$ are defined in \eqref{zeros}.
 Rudnick \cite{R} obtained an asymptotic formula for $\Sigma(\Phi,g)$ when $N<4g$, and several lower order terms were identified in \cite{BF1} for restricted ranges of $N$. As an application to the computation of the one-level density of zeros, Bui and Florea \cite{BF1} also proved that more than $94\%$ of $L(1/2,\chi_D) \neq 0$. From the Katz and Sarnak density conjectures \cite{katzsarnak, katzsarnak2}, it is expected that $L(1/2,\chi_D) \neq 0$ for $100 \%$ of the discriminants. Here, we recover and get agreement with the asymptotic formula obtained in \cite{R} for $N<4g$, using different techniques.
 
 \kommentar{Here, we extend the support of the Fourier transform to $N<5g$, and hence improve the proportion of non-vanishing obtained in \cite{BF}. 
 \begin{theorem}
 We have
\begin{align} 
\Sigma(\Phi,g) &= \widehat{\Phi}(0) - \frac{1}{g} \sum_{n=1}^g \widehat{\Phi} \Big(  \frac{n}{g} \Big)- \frac{\widehat{\Phi}(1)}{g(q-1)} + \frac{1}{g} \sum_{n=1}^{N/2} \widehat{\Phi} \Big( \frac{n}{g}  \Big) \frac{1}{q^n} \sum_{d(P)|n} \frac{d(P)}{|P|+1}  \nonumber \\
&+ O(q^{N -5g+\varepsilon g}).
\end{align}
 \label{theorem_density}
 \end{theorem}

As an application to the theorem above, we also obtain the following corollary.

\begin{corollary}
\label{nonvanish_cor}
We have
$$ \frac{1}{|\mathcal{H}_{2g+1}|} \Big| \Big\{ D \in \mathcal{H}_{2g+1} \, \Big| L(1/2,\chi_D) \neq 0  \Big\} \Big| \geq 0.96+o(1),$$ as $g \to \infty$.
\end{corollary}}

The paper is organized as follows. In Section \ref{background} we provide the necessary background. In Section \ref{mainprop} we prove Theorem \ref{theorem2}. We prove Theorem \ref{theorem1} in Section \ref{mainthm}, and refine the result obtained in Section \ref{mainprop} in the case of a quotient of two $L$--functions. We prove the upper bounds for negative moments of $L$--functions in Section \ref{sectiontheorem3}, and compute the $1$-level density of zeros in Section \ref{density_comp}. 
Finally, in the Appendix, we prove an asymptotic formula for a certain trigonometric sum.

\section{Background in function fields}
\label{background}
Let $\mathcal{M}$ denote the set of monic polynomials over $\mathbb{F}_q[x]$, $\mathcal{M}_n$ the set of monic polynomials of degree $n$ and $\mathcal{M}_{\leq n}$ the set of monic polynomials of degree at most $n$. Let $\mathcal{H}_n$ denote the set of monic, square-free polynomials of degree $n$. We will denote the degree of a polynomial $f$ by $d(f)$. The norm of a polynomial is defined by $|f|=q^{d(f)}$.
\subsection{Quadratic Dirichlet $L$-functions over function fields}

For $\Re s>1$, the zeta function of $\mathbb{F}_q[x]$ is defined by
\[
\zeta_q(s):=\sum_{f\in\mathcal{M}}\frac{1}{|f|^s}=\prod_{P\in \mathcal{P}}\bigg(1-\frac{1}{|P|^s}\bigg)^{-1}.
\]
Since there are $q^n$ monic polynomials of degree $n$, we see that
\[
\zeta_q(s)=\frac{1}{1-q^{1-s}}.
\]
With the change of variable $u=q^{-s}$, we then write $\mathcal{Z}(u)=\zeta_q(s)$, so that $$\mathcal{Z}(u)=\frac{1}{1-qu}.$$

For $P$ a monic irreducible polynomial, the quadratic residue symbol $\big(\frac{f}{P}\big)\in\{0,\pm1\}$ is defined by
\[
\Big(\frac{f}{P}\Big)\equiv f^{(|P|-1)/2}(\textrm{mod}\ P).
\]
If $Q=P_{1}^{\alpha_1}P_{2}^{\alpha_2}\ldots P_{r}^{\alpha_r}$, then the Jacobi symbol is defined by
\[
\Big(\frac{f}{Q}\Big)=\prod_{j=1}^{r}\Big(\frac{f}{P_j}\Big)^{\alpha_j}.
\]
The Jacobi symbol satisfies the quadratic reciprocity law. That is to say if $A,B\in \mathbb{F}_q[x]$ are relatively prime, monic polynomials, then
\[
\Big(\frac{A}{B}\Big)=(-1)^{(q-1)d(A)d(B)/2}\Big(\frac{B}{A}\Big).
\]
For simplicity we assume that $q\equiv 1(\textrm{mod}\ 4)$, and hence the quadratic reciprocity law gives $\big(\frac{A}{B}\big)=\big(\frac{B}{A}\big)$, a fact we will use throughout the paper.

For $D$ monic, we define the character 
\[
\chi_D(g)=\Big(\frac{D}{g}\Big),
\]
and consider the $L$-function attached to $\chi_D$,
\[
L(s,\chi_D):=\sum_{f\in\mathcal{M}}\frac{\chi_D(f)}{|f|^s}.
\]
With the change of variable $u=q^{-s}$ we have
\begin{equation}\label{formulaL}
\mathcal{L}(u,\chi_D):=L(s,\chi_D)=\sum_{f\in\mathcal{M}}\chi_D(f)u^{d(f)}=\prod_{P\in \mathcal{P}}\big(1-\chi_D(P)u^{d(P)}\big)^{-1}.
\end{equation}
For $D\in\mathcal{H}_{2g+1}$, $\mathcal{L}(u,\chi_D)$ is a polynomial in $u$ of degree $2g$ and it satisfies a functional equation
\begin{equation}\label{tfe}
\mathcal{L}(u,\chi_D)=(qu^2)^g\mathcal{L}\Big(\frac{1}{qu},\chi_D\Big).
\end{equation}

There is a connection between $\mathcal{L}(u,\chi_D)$ and zeta function of curves. For $D\in\mathcal{H}_{2g+1}$, the affine equation $y^2=D(x)$ defines a projective and connected hyperelliptic curve $C_D$ of genus $g$ over $\mathbb{F}_q$. The zeta function of the curve $C_D$ is defined by
\[
Z_{C_D}(u)=\exp\bigg(\sum_{j=1}^{\infty}N_j(C_D)\frac{u^j}{j}\bigg),
\]
where $N_j(C_D)$ is the number of points on $C_D$ over $\mathbb{F}_q$, including the point at infinity. Weil \cite{W} showed that
\[
Z_{C_D}(u)=\frac{P_{C_D}(u)}{(1-u)(1-qu)},
\]
where $P_{C_D}(u)$ is a polynomial of degree $2g$. It is known that $P_{C_D}(u)=\mathcal{L}(u,\chi_D)$ (this was proved in Artin's thesis). The Riemann Hypothesis for curves over function fields was proven by Weil \cite{W}, so all the zeros of $\mathcal{L}(u,\chi_D)$ are on the circle $|u|=q^{-1/2}$. We express $\mathcal{L}(u,\chi_D)$ in terms of its zeros as
\begin{equation}
\mathcal{L}(u,\chi_D)=\prod_{j=1}^{2g}\big(1-uq^{1/2}e^{-2\pi i\theta_{j,D}}\big). \label{zeros}
\end{equation} 

\subsection{Preliminary lemmas}

We start with the analogue of the approximate functional equation in the number field setting, which gives the following exact formula for $\prod_{j=1}^kL(1/2+\alpha_j,\chi_D)$. 

\begin{lemma}\label{fe} 
We have
$$\prod_{j=1}^kL(\tfrac12+\alpha_j,\chi_D)=\sum_{f\in\mathcal{M}_{\leq kg}}\frac{\tau_{\bf A}(f)\chi_D(f)}{\sqrt{|f|}}+q^{-2g{\bf A}}\sum_{f\in\mathcal{M}_{\leq kg-1}}\frac{\tau_{{\bf A}^-}(f)\chi_D(f)}{\sqrt{|f|}},$$ with $\tau_{\bf A}$ defined in \eqref{notation}.
\end{lemma}
\begin{proof}
From \eqref{formulaL} we have
\[
\mathcal{L}(q^{-\alpha_j}u,\chi_D)=\sum_{n= 0}^\infty\bigg(\sum_{f\in\mathcal{M}_n}\chi_D(f)|f|^{-\alpha_j}\bigg)u^n.
\]
So if we let 
\begin{equation}\label{sumovern}
\prod_{j=1}^k\mathcal{L}(q^{-\alpha_j}u,\chi_D)=\sum_{n= 0}^\infty c_nu^n,
\end{equation}
then
\[
c_n=\sum_{f\in\mathcal{M}_n}\tau_{\bf A}(f)\chi_D(f).
\]
Note that as $\mathcal{L}(u,\chi_D)$ is a polynomial in $u$ of degree $2g$, the sum in \eqref{sumovern} can be truncated at $n\leq 2kg$.

From the functional equation \eqref{tfe} we get
\begin{align*}
\sum_{n\leq 2kg}c_nu^n&=\prod_{j=1}^k(q^{1-2\alpha_j}u^2)^{g}\mathcal{L}\Big(\frac{q^{\alpha}}{qu},\chi_D\Big)\\
&=q^{kg-2g{\bf A}}u^{2kg}\sum_{n\leq 2kg}\bigg(\sum_{f\in\mathcal{M}_n}\tau_{{\bf A}^-}(f)\chi_D(f)\bigg)\Big(\frac{1}{qu}\Big)^n.
\end{align*}
Equating the coefficients of $u^{2kg-n}$ we obtain
\[
c_{2kg-n}=q^{kg-n-2g{\bf A}}\bigg(\sum_{f\in\mathcal{M}_n}\tau_{{\bf A}^-}(f)\chi_D(f)\bigg).
\]
Hence
\begin{align*}
\prod_{j=1}^k\mathcal{L}(q^{-\alpha_j}u,\chi_D)&=\sum_{n\leq kg}c_nu^n+\sum_{n\leq kg-1}c_{2kg-n}u^{2kg-n}\\
&=\sum_{n\leq kg}c_nu^n+q^{-2g{\bf A}}\sum_{n\leq kg-1}(qu^2)^{kg-n}\bigg(\sum_{f\in\mathcal{M}_n}\tau_{{\bf A}^-}(f)\chi_D(f)\bigg)u^{n}.
\end{align*}
The lemma follows by letting $u=q^{-1/2}$.
\end{proof}

The following lemmas are in \cite{F1} (see  Lemma 2.2, Proposition 3.1 and Lemma 3.2)
.

\begin{lemma}\label{L1}
For $f\in\mathcal{M}$ we have
\[
\sum_{D\in\mathcal{H}_{2g+1}}\chi_D(f)=\sum_{C|f^\infty}\sum_{r\in\mathcal{M}_{2g+1-2d(C)}}\chi_f(r)-q\sum_{C|f^\infty}\sum_{r\in\mathcal{M}_{2g-1-2d(C)}}\chi_f(r),
\]
where the summations over $C$ are over monic polynomials $C$ whose prime factors are among the prime factors of $f$.
\end{lemma}

We define the generalized Gauss sum as
\[
G(V,\chi):= \sum_{u (\textrm{mod}\ f)}\chi(u)e\Big(\frac{uV}{f}\Big),
\]
where the exponential was defined in \cite{Hayes} as follows. For $a \in  \mathbb{F}_q\big((\frac 1x)\big) $, 
$$ e(a) = e^{ 2 \pi i \text{Tr}_{\mathbb{F}_q / \mathbb{F}_p} (a_1)/p},$$ where $a_1$ is the coefficient of $1/x$ in the Laurent expansion of $a$. 

\begin{lemma}\label{L3}
Let $f\in\mathcal{M}_n$. If $n$ is even then
\[
\sum_{r\in\mathcal{M}_m}\chi_f(r)=\frac{q^m}{|f|}\bigg(G(0,\chi_f)+q\sum_{V\in\mathcal{M}_{\leq n-m-2}}G(V,\chi_f)-\sum_{V\in\mathcal{M}_{\leq n-m-1}}G(V,\chi_f)\bigg),
\]
otherwise
\[
\sum_{r\in\mathcal{M}_m}\chi_f(r)= \frac{q^{m+1/2}} {|f|}\sum_{V\in\mathcal{M}_{n-m-1}}G(V,\chi_f).
\]
\end{lemma}

\begin{lemma}\label{L2}

\begin{enumerate}
\item If $(f,h)=1$, then $G(V, \chi_{fh})= G(V, \chi_f) G(V,\chi_h)$.
\item Write $V= V_1 P^{\alpha}$ where $P \nmid V_1$.
Then 
 $$G(V , \chi_{P^j})= 
\begin{cases}
0 & \mbox{if } j \leq \alpha \text{ and } j \text{ odd,} \\
\varphi(P^j) & \mbox{if }  j \leq \alpha \text{ and } j \text{ even,} \\
-|P|^{j-1} & \mbox{if }  j= \alpha+1 \text{ and } j \text{ even,} \\
\chi_P(V_1) |P|^{j-1/2} & \mbox{if } j = \alpha+1 \text{ and } j \text{ odd, } \\
0 & \mbox{if } j \geq \alpha+2 .
\end{cases}$$ 
\end{enumerate}
\end{lemma}

\begin{lemma}\label{L5}
For $f\in\mathcal{M}$ we have
\[
\frac{1}{| \mathcal{H}_{2g+1}|}\sum_{D \in \mathcal{H}_{2g+1}} \chi_D(f^2)=\prod_{P|f}\bigg(1+\frac{1}{|P|}\bigg)^{-1}+O_\varepsilon(q^{-2g}|f|^{\varepsilon}).
\]
\end{lemma}
\begin{proof}
See, for example, Lemma 3.4 in \cite{BF}.
\end{proof}

We also have the following explicit formula (see, for example, \cite{FR}). 
\begin{lemma}\label{explicit}
Let $N$ be a positive integer and $h(\theta) = \sum_{|n| \leq N} \widehat{h}(n) e(n \theta)$ be a real valued even trigonometric polynomial. Then
$$ \sum_{j=1}^{2g} h(\theta_{D,j}) = 2g \int_0^1 h(\theta) d \theta - 2 \sum_{f\in\mathcal{M}} \widehat{h}(d(f)) \frac{ \Lambda(f)\chi_D(f)}{\sqrt{|f|}}.$$
\end{lemma}

Now let $t \in \mathbb{R}$ and $\ell$ be an even integer. Let
$$E_{\ell}(t) = \sum_{s \leq \ell} \frac{t^s}{s!}.$$
Note that $E_{\ell}(t) \geq 1$ if $t\geq 0$ and $E_{\ell}(t) >0$ for any $t$ since $\ell$ is even. We will frequently use the fact that  for $t \leq \ell/e^2$, we have that
\begin{equation}
e^t \leq (1+e^{-\ell/2}) E_{\ell}(t).
\label{taylor}
\end{equation}

Let $\nu(f)$ be the multiplicative function given by $$\nu(P^a) = \frac{1}{a!}.$$ We shall need the following result, which is part of Lemma $3.2$ in \cite{DFL}.

\begin{lemma}\label{power}
Let $a(f)$ be a completely multiplicative function
. Then for any interval $I$ and any $s\in\mathbb{N}$ we have that
\begin{equation*}
\bigg( \sum_{d(P) \in I} a(P)  \bigg)^s = s! \sum_{\substack{P | f \Rightarrow d(P) \in I \\ \Omega(f) = s}} a(f) \nu(f).
\end{equation*}
\end{lemma}

\section{Proof of Theorem \ref{theorem2}}\label{mainprop}

In this section we shall prove Theorem \ref{theorem2}. Similar results without the shifts were obtained 
in \cite{BF}.

\subsection{Initial manipulations}

By the functional equation \eqref{tfe} we have
\begin{equation}\label{switchingalpha}
\prod_{j=1}^{k}L(\tfrac 12+\alpha_j,\chi_D)=q^{-2g\sum_{j=1}^{k}\mathfrak{a}_j\alpha_j}\prod_{j=1}^{k}L(\tfrac 12+\varepsilon_j\alpha_j,\chi_D),
\end{equation}
where $\mathfrak{a}_j=0$ and $\varepsilon_j=1$ if $\Re \alpha_j\geq0$, and $\mathfrak{a}_j=1$ and $\varepsilon_j=-1$ if $\Re \alpha_j<0$.
So we can assume that $\Re \alpha_j\geq0$ for every $1\leq j\leq k$.

In view of Lemma \ref{fe} we have
\[
\frac{1}{|\mathcal{H}_{2g+1}|}\sum_{D\in \mathcal{H}_{2g+1}}\bigg(\prod_{j=1}^kL(\tfrac12+\alpha_j,\chi_D)\bigg)\chi_D(h)=S_{\bf A}(h;kg)+q^{-2g\bf A}S_{\bf A^-}(h;kg-1),
\]
where if ${\bf C}=\{\gamma_1,\ldots,\gamma_k\}$ with  $\Re \gamma_j\geq0$ for every $1\leq j\leq k$, or $\Re \gamma_j\leq0$ for every $1\leq j\leq k$, then
\[
S_{\bf C}(h;N)=\frac{1}{|\mathcal{H}_{2g+1}|}\sum_{D\in \mathcal{H}_{2g+1}}\sum_{d(f)\leq N}\frac{\tau_{\bf C}(f)\chi_D(fh)}{\sqrt{|f|}}
\]
for $N\in\{kg,kg-1\}$. 

From Lemma \ref{L1} we obtain that
\[
S_{\bf C}(h;N)=S_{{\bf C};1}(h;N)-qS_{{\bf C};2}(h;N),
\]where
\begin{align*}
S_{{\bf C};1}(h;N)=\frac{1}{|\mathcal{H}_{2g+1}|}\sum_{\substack{d(f)\leq N}}\frac{\tau_{\bf C}(f)}{\sqrt{|f|}}\sum_{C|(fh)^\infty}\sum_{r\in\mathcal{M}_{2g+1-2d(C)}}\chi_{fh}(r)
\end{align*}
and
\[
S_{{\bf C};2}(h;N)=\frac{1}{|\mathcal{H}_{2g+1}|}\sum_{\substack{d(f)\leq N}}\frac{\tau_{\bf C}(f)}{\sqrt{|f|}}\sum_{C|(fh)^\infty}\sum_{r\in\mathcal{M}_{2g-1-2d(C)}}\chi_{fh}(r).
\]
We further write 
\[
S_{{\bf C};1}(h;N)=S_{{\bf C};1}^{\textrm{e}}(h;N)+S_{{\bf C};1}^{\textrm{o}}(h;N)
\]
according to whether the degree of the product $fh$ is even or odd, respectively. Lemma \ref{L3} and Lemma \ref{L2} then lead to
\[
S_{{\bf C};1}^{\textrm{e}}(h;N)=M_{{\bf C};1}(h;N)+S_{{\bf C};1}^{\textrm{e}}(h;N;V\ne0),
\]
where
\begin{align}\label{M}
M_{{\bf C};1}(h;N)=\frac{q}{(q-1)|h|}\sum_{\substack{d(f)\leq N\\fh=\square}}\frac{\tau_{\bf C}(f)\varphi(fh)}{|f|^{3/2}}\sum_{\substack{C|(fh)^\infty\\d(C)\leq g}}\frac{1}{|C|^2},
\end{align}
\begin{align}\label{Ske}
&S_{{\bf C};1}^{\textrm{e}}(h;N;V\ne0)=\frac{q}{(q-1)|h|}\sum_{\substack{d(f)\leq N\\d(fh)\ \textrm{even}}}\frac{\tau_{\bf C}(f)}{|f|^{3/2}}\sum_{\substack{C|(fh)^\infty\\d(C)\leq g}}\frac{1}{|C|^2}\nonumber\\
&\qquad\qquad\times\bigg(q\sum_{V\in\mathcal{M}_{\leq d(fh)-2g-3+2d(C)}}G(V,\chi_{fh})-\sum_{V\in\mathcal{M}_{\leq d(fh)-2g-2+2d(C)}}G(V,\chi_{fh})\bigg),
\end{align}
and
\begin{align}
S_{{\bf C};1}^{\textrm{o}}(h;N)=\frac{q^{3/2}}{(q-1)|h|}\sum_{\substack{d(f)\leq N\\d(fh)\ \textrm{odd}}}\frac{ \tau_{\bf C}(f)}{|f|^{3/2}}\sum_{\substack{C|(fh)^\infty\\d(C)\leq g}}\frac{1}{|C|^2}\sum_{V\in\mathcal{M}_{d(fh)-2g-2+2d(C)}}G(V,\chi_{fh}).
\label{s1odd}
\end{align}
We also decompose
\[
S_{{\bf C};1}^{\textrm{e}}(h;N;V\ne0)=S_{{\bf C};1}^{\textrm{e}}(h;N;V=\square)+S_{{\bf C};1}^{\textrm{e}}(h;N;V\ne\square)
\]
correspondingly to whether $V$ is a square or not.

We treat $S_{{\bf C};2}(h;N)$ similarly and define the functions $M_{{\bf C};2}(h;N)$, $S_{{\bf C};2}^{\textrm{o}}(h;N)$ and $S_{{\bf C};2}^{\textrm{e}}(h;N;V=\square)$, $S_{{\bf C};2}^{\textrm{e}}(h;N;V\ne\square)$ in the same way. Further denote 
\[
M_{{\bf C}}(h;N)=M_{{\bf C};1}(h;N)-qM_{{\bf C};2}(h;N)
\] 
and
$$S_{{\bf C}}^{\textrm{e}}(h;N;V=\square)=S_{{\bf C};1}^{\textrm{e}}(h;N;V=\square)-qS_{{\bf C};2}^{\textrm{e}}(h;N;V=\square).$$ 

For the terms $S_{{\bf C};1}^{\textrm{o}}(h;N)$ and $S_{{\bf C};2}^{\textrm{o}}(h;N)$, we note that the summations over $V$ are over odd degree polynomials, so $V\ne\square$ in these cases. Let 
\begin{align}\label{Sknonsquare}
S_{{\bf C}}(h;N;V \neq \square) &=\big(S_{{\bf C};1}^{\textrm{o}}(h;N)-qS_{{\bf C};2}^{\textrm{o}}(h;N)\big)\nonumber\\
&\qquad\qquad+\big(S_{{\bf C};1}^{\textrm{e}}(h;N;V\ne\square)-qS_{{\bf C};2}^{\textrm{e}}(h;N;V\ne\square)\big) 
\end{align} be the total contribution from $V \neq \square$. In Section \ref{sectionnonsquare} we will prove the bound
\begin{equation}\label{boundfornonsquare}
 S_{{\bf C}}(h;N;V \neq \square) \ll_\varepsilon |h|^{1/2} q^{N/2-N\min\{\Re \gamma_j\} -2g+ \varepsilon g}.
 \end{equation} 

We shall next consider $M_{{\bf C}}(h;N)$. The term $S_{{\bf C}}^{\textrm{e}}(h;N;V=\square)$ also contributes to the main term for $k \geq 2$ and will be evaluated in Section \ref{Vsquare}.

\subsection{Evaluate $M_{{\bf C}}(h;N)$}
\label{main}
Note that we can remove the condition $d(C)\leq g$ in the  sum over $C$ in \eqref{M} at the expense of an error of size $O_\varepsilon(q^{N/2-N\min\{\Re \gamma_j\}-2g+\varepsilon g})$ using the Rankin trick,
\begin{equation*}
\sum_{\substack{C|(fh)^\infty\\d(C)> g}}\frac{1}{|C|^2}< \sum_{C|(fh)^\infty}\frac{1}{|C|^2}\Big(\frac{|C|}{q^{g}}\Big)^{2-\varepsilon}=q^{-2g+\varepsilon g}\prod_{P|fh}\bigg(1-\frac{1}{|P|^\varepsilon}\bigg)^{-1},
\end{equation*}
and the fact that $|\tau_{\bf C}(f)|\leq |f|^{-\min\{\Re \gamma_j\}}\tau_k(f)$. 
So
\begin{align*}
M_{{\bf C};1}(h;N)&=\frac{q}{(q-1)|h|}\sum_{\substack{d(f)\leq N\\fh=\square}}\frac{ \tau_{\bf C}(f)\varphi(fh)}{|f|^{3/2}}\prod_{P|fh}\bigg(1-\frac{1}{|P|^2}\bigg)^{-1}\\
&\qquad\qquad+O_\varepsilon\big(q^{N/2-N\min\{\Re \gamma_j\}-2g+\varepsilon g}\big)\\
&=\frac{q}{(q-1)\sqrt{|h_1|}}\sum_{\substack{2d(f)\leq N-d(h_1)}}\frac{a(fh) \tau_{\bf C}(f^2h_1)}{|f|}+O_\varepsilon\big(q^{N/2-N\min\{\Re \gamma_j\}-2g+\varepsilon g}\big),
\end{align*}
where
\[
a(f)=\prod_{P|f}\bigg(1+\frac{1}{|P|}\bigg)^{-1},
\]
by  a change of variables $f\rightarrow f^2h_1$. A similar argument holds for $M_{{\bf C};2}(h;N)$ and we obtain
\begin{align}
M_{{\bf C}}(h;N)&=\frac{1}{\sqrt{|h_1|}}\sum_{2d(f)\leq N-d(h_1)}\frac{a(fh)\tau_{\bf C}(f^2h_1)}{|f|}+O_\varepsilon\big(q^{N/2-N\min\{\Re \gamma_j\}-2g+\varepsilon g}\big).\label{mc}
\end{align}

Using an analogue of Perron's formula in the form
\[
\sum_{2n\leq N} g(n)=\frac{1}{2\pi i}\int_{|u|=r}\Big(\sum_{n=0}^{\infty}g(n)u^{2n}\Big)\frac{du}{u^{N+1}(1-u)},
\]
where $r$ is such that $\sum_{n=0}^{\infty} g(n) u^n$ converges absolutely in $|u| \leq r<1$, 
leads to
\begin{align*}
M_{{\bf C}}(h;N)=\frac{1}{\sqrt{|h_1|}}\frac{1}{2\pi i}\oint_{|u|=r}\frac{\mathcal{F}_{\bf C}(u)du}{u^{N-d(h_1)+1}(1-u)}+O_\varepsilon\big(q^{N/2-N\min\{\Re \gamma_j\}-2g+\varepsilon g}\big),
\end{align*}
where
\begin{align*}
\mathcal{F}_{\bf C}(u)&=\sum_{f\in\mathcal{M}}\frac{a(fh)\tau_{\bf C}(f^2h_1)}{|f|}u^{2d(f)}.
\end{align*}
Now by multiplicativity we have
\[
\mathcal{F}_{\bf C}(u)=\mathcal{A}_{{\bf C}}(u)\mathcal{B}_{{\bf C}}(h;u)\prod_{1\leq i\leq j\leq k}\mathcal{Z}\Big(\frac{u^2}{q^{1+\gamma_i+\gamma_j}}\Big),
\]
where $\mathcal{A}_{{\bf C}}(u)$ and $
\mathcal{B}_{{\bf C}}(h;u)$ are defined in \eqref{formulaA} and \eqref{formulaB}. 
Thus,
\begin{align}\label{2001}
M_{\bf C}(h;N)&=\frac{1}{\sqrt{|h_1|}}\frac{1}{2\pi i}\oint_{|u|=r}\frac{\mathcal{A}_{{\bf C}}(u)\mathcal{B}_{{\bf C}}(h;u)du}{u^{N-d(h_1)+1}(1-u)\prod_{1\leq i\leq j\leq k}(1-q^{-(\gamma_i+\gamma_j)}u^2)}\\
&\qquad\qquad+O_\varepsilon\big(q^{N/2-N\min\{\Re \gamma_j\}-2g+\varepsilon g}\big).\nonumber
\end{align}

We have $\mathcal{A}_{{\bf C}}(u)$ converges absolutely for $|u|< q^{1/2+\min\{\Re \gamma_j\}}$. We move the contour of integration to $|u|=q^{1/2+\min\{\Re \gamma_j\}-\varepsilon}$, encountering simple poles at $u=1$ and $u=\pm q^{(\gamma_i+\gamma_j)/2}$ for every $1\leq i\leq j\leq k$. 
The integral over the new contour is trivially bounded by $O_\varepsilon(|h_1|^{\min\{\Re \gamma_j\}} q^{-N/2-N\min\{\Re \gamma_j\}+\varepsilon g})$ so 
\begin{align*}
M_{\bf C}(h;N)&=-\textrm{Res}(u=1) -\sum_{1\leq i\leq j\leq k}\textrm{Res}\big(u=\pm q^{(\gamma_i+\gamma_j)/2}\big)\\
&\qquad +O_\varepsilon\big(q^{N/2-N\min\{\Re \gamma_j\}-2g+\varepsilon g}\big)+O_\varepsilon\big(|h_1|^{\min\{\Re \gamma_j\}} q^{-N/2-N\min\{\Re \gamma_j\}+\varepsilon g}\big).
\end{align*}

Standard calculations give
\begin{equation}\label{2002}
\textrm{Res}(u=1)=-\frac{\mathcal{A}_{{\bf C}}(1)\mathcal{B}_{{\bf C}}(h;1)}{\sqrt{|h_1|}}\prod_{1\leq i\leq j\leq k}\zeta_q(1+\gamma_i+\gamma_j)=-\frac{\widetilde{\mathcal{S}}_{\bf C}(h)}{\sqrt{|h_1|}}
\end{equation}
and
\begin{align}\label{2003}
&\textrm{Res}\big(u=q^{(\gamma_i+\gamma_j)/2}\big)+\textrm{Res}\big(u=- q^{(\gamma_i+\gamma_j)/2}\big)=-\frac{\mathcal{A}_{{\bf C}}\big(q^{(\gamma_i+\gamma_j)/2}\big)\mathcal{B}_{{\bf C}}\big(h;q^{(\gamma_i+\gamma_j)/2}\big)}{\sqrt{|h_1|}}\nonumber\\
&\qquad\qquad\times q^{-\left[\frac{N-d(h_1)}{2}\right](\gamma_i+\gamma_j)}\zeta_q(1-\gamma_i-\gamma_j)\prod_{\substack{1\leq i'\leq j'\leq k\\(i',j')\ne (i,j)}}\zeta_q(1+\gamma_i'+\gamma_j'-\gamma_i-\gamma_j)
\end{align}
 for every $1\leq i\leq j\leq k$. Thus,
\begin{align}\label{MCformula}
M_{\bf C}(h;N)&=\frac{\widetilde{\mathcal{S}}_{\bf C}(h)}{\sqrt{|h_1|}}+\sum_{1\leq i\leq j\leq k}M_{\bf C}^{i,j}(h;N)+O_\varepsilon\big(q^{N/2-N\min\{\Re \gamma_j\}-2g+\varepsilon g}\big)\\
&\qquad\qquad+O_\varepsilon\big(|h_1|^{\min\{\Re \gamma_j\}} q^{-N/2-N\min\{\Re \gamma_j\}+\varepsilon g}\big),\nonumber
\end{align}
where
\begin{align}\label{Mịformula}
M_{\bf C}^{i,j}(h;N)&=\frac{\mathcal{A}_{{\bf C}}\big(q^{(\gamma_i+\gamma_j)/2}\big)\mathcal{B}_{{\bf C}}\big(h;q^{(\gamma_i+\gamma_j)/2}\big)}{\sqrt{|h_1|}} q^{-\left[\frac{N-d(h_1)}{2}\right](\gamma_i+\gamma_j)}\zeta_q(1-\gamma_i-\gamma_j)\nonumber\\
&\qquad\qquad\times\prod_{\substack{1\leq i'\leq j'\leq k\\(i',j')\ne (i,j)}}\zeta_q(1+\gamma_i'+\gamma_j'-\gamma_i-\gamma_j).
\end{align}

\subsection{Evaluate $S_{\bf C}^{\textrm{e}}(h;N;V=\square)$}\label{Vsquare}

First we note that as in the previous subsection we can extend the sum over $C$ in \eqref{Ske} to infinity, at the expense of an error of size $O_\varepsilon(q^{N/2-N\min\{\Re \gamma_j\}-2g+\varepsilon g})$. 
 So
\begin{align}
&S_{\bf C}^{\textrm{e}}(h;N;V=\square)=\frac{q}{(q-1)|h|}\sum_{\substack{d(f)\leq N\\d(fh)\ \textrm{even}}}\frac{\tau_{\bf C}(f)}{|f|^{3/2}}\sum_{C|(fh)^\infty}\frac{1}{|C|^2} \nonumber \\
&\qquad\quad\times\bigg(q\sum_{V\in\mathcal{M}_{\leq d(fh)/2-g-2+d(C)}}G(V^2,\chi_{fh})-2\sum_{V\in\mathcal{M}_{\leq d(fh)/2-g-1+d(C)}}G(V^2,\chi_{fh}) \nonumber \\
&\qquad\qquad\qquad\qquad+\frac1q\sum_{V\in\mathcal{M}_{\leq d(fh)/2-g+d(C)}}G(V^2,\chi_{fh})\bigg)+O_\varepsilon\big(q^{N/2-N\min\{\Re \gamma_j\}-2g+\varepsilon g}\big). \nonumber
\end{align}
 Applying the Perron formula in the form
\[
\sum_{n\leq N}g(n)=\frac{1}{2\pi i}\int_{|u|=r_1}\Big(\sum_{n=0}^{\infty}g(n)u^{n}\Big)\frac{du}{u^{N+1}(1-u)}
\] to the sums over $V$, where $r_1<1$ is such that $\sum_{n=0}^{\infty} g(n) u^n$ converges absolutely for $|u| \leq r_1$, we get
\begin{align*}
&S_{\bf C}^{\textrm{e}}(h;N;V=\square) = \frac{1}{(q-1)|h|} \sum_{\substack{d(f)\leq N\\d(fh)\ \textrm{even}}}\frac{\tau_{\bf C}(f)}{|f|^{3/2}}\sum_{C|(fh)^\infty}\frac{1}{|C|^2} \frac{1}{2 \pi i} \oint_{|u|=r_1}u^{-d(f)/2-d(C)}\\
&\qquad\qquad\times \Big( \sum_{V \in \mathcal{M}} G(V^2,\chi_{f h})\, u^{d(V)}\Big) \frac{(1-qu)^2du}{ u^{d(h)/2-g+1}(1-u) }+O_\varepsilon\big(q^{N/2-N\min\{\Re \gamma_j\}-2g+\varepsilon g}\big).
\end{align*}
Another application of the Perron formula, this time in the form
\[
\sum_{\substack{n\leq N\\n+m\ \textrm{even}}}g(n)=\frac{1}{2\pi i}\int_{|w|=r_2}\Big(\sum_{n=0}^{\infty}g(n)w^{n}\Big)\delta(m,N;w)\frac{dw}{w^{N+1}},
\] 
where
\begin{equation}\label{identity1}
\delta(m,N;w)=\frac12\bigg(\frac{1}{1-w}+\frac{(-1)^{N-m}}{1+w}\bigg)=\frac{w^{N-m-2\left[\frac{N-m}{2}\right]}}{1-w^2},
\end{equation}to the sum over $f$ yields
\begin{align}
&S_{\bf C}^{\textrm{e}}(h;N;V=\square) = \frac{1}{(q-1)|h|} \frac{1}{(2 \pi i)^2} \oint_{|u|=r_1}  \oint_{|w|=r_2} \nonumber \\
&\qquad \frac{\mathcal{N}_{\bf C} (h;u,w) (1-qu)^2 dwdu }{u^{\left[\frac{N+d(h)}{2}\right]-g+1}w^{2\left[\frac{N-d(h_1)}{2}\right]+d(h_1)+1}(1-u) (1-uw^2)} +O_\varepsilon\big(q^{N/2-N\min\{\Re \gamma_j\}-2g+\varepsilon g}\big), \label{second_error}
\end{align} where $r_2<1$ and
$$ \mathcal{N}_{\bf C} (h;u,w) = \sum_{f,V \in \mathcal{M}} \frac{\tau_{\bf C}(f) G(V^2,\chi_{f h})}{|f|^{3/2}} \prod_{P | f h} \bigg(1- \frac{1}{|P|^2u^{d(P)}} \bigg)^{-1} u^{d(V)} w^{d(f)}. $$

Our next step is to write $ \mathcal{N}_{\bf C} (h;u,w)$ as an Euler product. From Lemma \ref{L2} we have
\begin{align*}
&\sum_{f \in \mathcal{M}}   \frac{\tau_{\bf C}(f) G(V^2,\chi_{fh})}{|f|^{3/2}} \prod_{P|f h} \bigg(1-\frac{1}{|P|^2 u^{d(P)}} \bigg)^{-1}w^{d(f)} \\
&\quad = \prod_{P |h}  \bigg(1-\frac{1}{|P|^2 u^{d(P)}} \bigg)^{-1} \prod_{P \nmid  h V} \bigg( 1+\sum_{\gamma\in\bf C}\frac{ w^{d(P)}}{|P|^{1+\gamma}}\bigg(1-\frac{1}{|P|^2 u^{d(P)}} \bigg)^{-1}\bigg) \\
& \, \, \, \quad\ \ \prod_{\substack{P\nmid h \\P|V}} \bigg( 1+ \sum_{j=1}^{\infty} \frac{ \tau_{\bf C}(P^j) G(V^2,\chi_{P^j}) w^{j d(P)}}{|P|^{3j/2} }\bigg(1-\frac{1}{|P|^2 u^{d(P)}} \bigg)^{-1} \bigg) \\
& \, \, \, \,\,  \quad\ \prod_{\substack{ P | h\\P \nmid V }} G(V^2,\chi_{P^{\ord_P(h)}})\prod_{\substack{ P | h\\P | V }} \bigg( G(V^2,\chi_{P^{\ord_P(h)}}) + \sum_{j=1}^{\infty} \frac{\tau_{\bf C}(P^j) G(V^2,\chi_{P^{j+\ord_P(h)}}) w^{jd(P)}}{|P|^{3j/2}} \bigg) .
\end{align*} Note that if $P| h_2$ and $P \nmid V$, then the above expression is $0$. Hence we must have that $ (\prod_{P | h_2} P )| V$. Moreover, from the last Euler factor above, note that we must have $h_2 | V$, so we denote $V= h_2 V_1$. 
Using Lemma \ref{L2}, we then rewrite 
$$ \prod_{\substack{P | h \\ P \nmid V}} G(V^2,\chi_{P^{\ord_P(h)}}) =\prod_{\substack{P | h_1 \\ P \nmid h_2}} |P|^{1/2} \prod_{\substack{P | h_1 \\ P \nmid h_2 \\ P |V_1}} |P|^{-1/2}.$$
By multiplicativity we then obtain
\begin{align}\label{theEulerproduct}
& \mathcal{N}_{\bf C}  (h;u,w)  =  u^{d(h_2)}\prod_{P\in\mathcal{P}} \bigg( 1- \frac{1}{|P|^2 u^{d(P)}} \bigg)^{-1}  \nonumber\\
&\ \prod_{P \nmid h} \bigg(1-\frac{1}{|P|^2u^{d(P)}}+\sum_{\gamma\in\bf C}\frac{ w^{d(P)}}{|P|^{1+\gamma}}\nonumber\\
&\qquad\qquad\qquad\qquad\qquad+ \sum_{i=1}^{\infty} u^{i d(P)} \bigg( 1-\frac{1}{|P|^2u^{d(P)}}+ \sum_{j=1}^{\infty} \frac{\tau_{\bf C}(P^j) G(P^{2i},\chi_{P^j}) w^{j d(P)}}{|P|^{3j/2}}\bigg) \bigg)\nonumber \\
&\ \ \prod_{P | h_1} \Bigg(|P|^{1/2+2 \ord_P(h_2)}+\sum_{i=1}^{\infty}u^{i d(P)} \sum_{j=0}^{\infty} \frac{\tau_{\bf C}(P^j) G( P^{2i+ 2 \ord_P(h_2)} , \chi_{P^{j+1+2 \ord_P(h_2)}}) w^{jd(P)}}{|P|^{3j/2}} \Bigg)\nonumber \\
&\ \ \ \prod_{\substack{P \nmid h_1   \\ P | h_2 }} \bigg( \varphi\big(P^{2 \ord_P(h_2)}\big) + \sum_{\gamma\in\bf C}\frac{  |P|^{2 \ord_P(h_2)} w^{d(P)}}{|P|^{1+\gamma}}\\
&\qquad\qquad\qquad\qquad\qquad\qquad+ \sum_{i=1}^{\infty} u^{i d(P)} \sum_{j=0}^{\infty} \frac{ \tau_{\bf C}(P^j) G( P^{2i+ 2 \ord_P(h_2)} , \chi_{P^{j+2 \ord_P(h_2)}}) w^{jd(P)}}{|P|^{3j/2}}  \bigg).\nonumber
\end{align}

\subsubsection{The case $k=1$}
We have 
\begin{align*}
\mathcal{N}_{\gamma_1}  (h;u,w) &= \frac{|h|u^{d(h_2)}}{\sqrt{|h_1 |}}\mathcal{C}_{\gamma_1}(u,w)\mathcal{D}_{\gamma_1}(h;u,w)\\
&\qquad\qquad\times\mathcal{Z}(u) \mathcal{Z}\Big(\frac {w}{q^{1+\gamma_1}}\Big)\mathcal{Z}\Big(\frac {w^2}{q^{2+2\gamma_1}}\Big)^{-1}  \mathcal{Z}\Big(\frac {uw^2}{q^{1+2\gamma_1}}\Big)\mathcal{Z}\Big(\frac {w}{q^{3+\gamma_1}u}\Big) ,
\end{align*} where 
\begin{align*}
\mathcal{C}_{\gamma_1}(u,w)&=\prod_P\bigg(1-\frac{1}{|P|^2u^{d(P)}}\bigg)^{-1}\bigg(1+\frac{w^{d(P)}}{|P|^{1+\gamma_1}}\bigg)^{-1}\bigg(1-\frac{w^{d(P)}}{|P|^{3+\gamma_1}u^{d(P)}}\bigg)\\
&\qquad\qquad\times \bigg(1+\frac{w^{d(P)}(1-u^{d(P)})}{|P|^{1+\gamma_1}}-\frac{1}{|P|^2u^{d(P)}}-\frac{(uw^2)^{d(P)}}{|P|^{2+2\gamma_1}}+\frac{w^{2d(P)}}{|P|^{3+2\gamma_1}}\bigg)
\end{align*}
and
\begin{align*}
&\mathcal{D}_{\gamma_1}(h;u,w)=\prod_{P|h}\bigg(1+\frac{w^{d(P)}(1-u^{d(P)})}{|P|^{1+\gamma_1}}-\frac{1}{|P|^2u^{d(P)}}-\frac{(uw^2)^{d(P)}}{|P|^{2+2\gamma_1}}+\frac{w^{2d(P)}}{|P|^{3+2\gamma_1}}\bigg)^{-1}\\
&\qquad\qquad\times \prod_{P | h_1} \bigg(1-u^{d(P)}+\frac{(uw)^{d(P)}}{|P|^{\gamma_1}}-\frac{(uw)^{d(P)}}{|P|^{1+\gamma_1}}\bigg) \prod_{\substack{P \nmid h_1 \\ P | h_2}}\bigg(1-\frac{1}{|P|}+\frac{w^{d(P)}-(uw)^{d(P)}}{|P|^{1+\gamma_1}}\bigg).
\end{align*}
Note that
$$
\mathcal{C}_{\gamma_1}(u,w)=1-\frac{(uw)^{d(P)}}{|P|^{1+\gamma_1}}+\frac{w^{2d(P)}}{|P|^{3+2\gamma_1}}-\frac{w^{2d(P)}}{|P|^{4+2\gamma_1}u^{d(P)}}+\frac{w^{d(P)}}{|P|^{5+\gamma_1}u^{2d(P)}}-\frac{2w^{2d(P)}}{|P|^{6+2\gamma_1}u^{2d(P)}}+\ldots,
$$
so $\mathcal{C}_{\gamma_1}(u,w)$ converges absolutely  for $|u|>1/q^2$, $|uw|<q^{\Re \gamma_1}$, $|w|<q^{1+\Re \gamma_1}$, $|w|^2<q^{3+2\Re \gamma_1}|u|$ and $|w|<q^{4+\Re \gamma_1}|u|^2$. 
So moving the $u$-contour to $|u|=r_1=q^{-3/2+\varepsilon}$ we get 
\begin{align}
S_{\gamma_1}^{\textrm{e}}(h;N;V=\square)& = \frac{1}{(q-1)\sqrt{|h_1|}}  \frac{1}{(2 \pi i)^2}\oint_{|u|=r_1}  \oint_{|w|=r_2} \nonumber  \\
&\qquad \frac{\mathcal{C}_{\gamma_1}(u,w)\mathcal{D}_{\gamma_1}(h;u,w) (1-qu)(1-q^{-(1+2\gamma_1)}w^2)  }{(1-u) (1-q^{-\gamma_1}w)(1-uw^2)(1-q^{-2\gamma_1}uw^2)(u-q^{-(2+\gamma_1)}w)}  \nonumber \\
&\qquad\qquad\times\frac{dwdu}{u^{\left[\frac{N+d(h_1)}{2}\right]-g}w^{2\left[\frac{N-d(h_1)}{2}\right]+d(h_1)+1}}+O_\varepsilon\big(q^{-3g/2-g\Re \gamma_1+\varepsilon g}\big). \label{second_explicit}
\end{align}
We enlarge the contour of integration over $w$ to $|w| = q^{3/4+\min\{0,\Re \gamma_1\}-2\varepsilon}$, encountering a simple pole at $w=q^{\gamma_1}$ 
and another simple pole at $w=q^{2+\gamma_1}u$, as $\Re \gamma_1<1/4$. 
The new integral is $$\ll_\varepsilon |h_1|^{1/4}q^{-3g/2-g\min\{0,\Re \gamma_1\}+\varepsilon g }.$$  Hence
\begin{align}\label{6003}
S_{\gamma_1}^{\textrm{e}}(h;N;V=\square) &= S_{\gamma_1}^{\textrm{e};1}(h;N) +S_{\gamma_1}^{\textrm{e};2}(h;N)+O_\varepsilon\big(|h_1|^{1/4}q^{-3g/2-g\min\{0,\Re \gamma_1\}+\varepsilon g }\big),
\end{align}
where
\begin{align}\label{secondmaintermk=1}
S_{\gamma_1}^{\textrm{e};1}(h;N)& = \frac{q^{-2\gamma_1\left[\frac{N-d(h_1)}{2}\right]-1}}{|h_1|^{1/2+\gamma_1}}  \frac{1}{2 \pi i}\oint_{|u|=r_1} \frac{\mathcal{C}_{\gamma_1}(u,q^{\gamma_1})\mathcal{D}_{\gamma_1}(h;u,q^{\gamma_1}) (1-qu) du }{u^{\left[\frac{N+d(h_1)}{2}\right]-g}(1-u)^2(1-q^{2\gamma_1}u)(u-q^{-2})}
\end{align} 
and
\begin{align}
S_{\gamma_1}^{\textrm{e};2}(h;N)& = \frac{q^{-2(2+\gamma_1)\left[\frac{N-d(h_1)}{2}\right]}}{(q-1)|h_1|^{5/2+\gamma_1}}  \frac{1}{2 \pi i}\oint_{|u|=r_1} \frac{\mathcal{C}_{\gamma_1}(u,q^{2+\gamma_1}u)\mathcal{D}_{\gamma_1}(h;u,q^{2+\gamma_1}u) (1-qu)(1-q^{3}u^2)  }{(1-u) (1-q^{2}u)(1-q^{4+2\gamma_1}u^3)(1-q^{4}u^3)}\nonumber \\
&\qquad\qquad\times\frac{du}{u^{\left[\frac{N+d(h_1)}{2}\right]+2\left[\frac{N-d(h_1)}{2}\right]-g+d(h_1)+1}}.\label{s2}
\end{align}

We first evaluate $S_{\gamma_1}^{\textrm{e};2}(h;N)$. We shift the contour of integration to $|u| =q^{-1-\varepsilon}$, encountering a simple pole at $u=q^{-4/3}$ and another simple pole at $u=q^{-(4+2\gamma_1)/3}$, again as $|\Re \gamma_1|<1/4$. The integral over the new contour is bounded by $O_\varepsilon(q^{-3g/2-g\Re \gamma_1+\varepsilon g})$. So
\begin{align}\label{3002}
S_{\gamma_1}^{\textrm{e};2}(h;N)&=-\text{Res}(u=q^{-4/3})-\text{Res}(u=q^{-(4+2\gamma_1)/3})+O_\varepsilon\big(q^{-3g/2-g\Re \gamma_1+\varepsilon g}\big)\nonumber\\
&=E_{\gamma_1}(h;N)+O_\varepsilon\big(q^{-3g/2-g\Re \gamma_1+\varepsilon g}\big),
\end{align}
say. Here for $N\in\{g,g-1\}$ we have
\[
E_{\gamma_1}(h;N)\asymp |h_1|^{1/6}q^{-4g/3-g\Re \gamma_1+\varepsilon}+|h_1|^{1/6+\Re \gamma_1/3}q^{-4g/3-2g\Re \gamma_1/3+\varepsilon}.
\]

Next we evaluate $S_{\gamma_1}^{\textrm{e};1}(h;N)$. By a change of variables $u\rightarrow q^{2\alpha_1}/ u^2$ in \eqref{secondmaintermk=1} we get
\begin{align*}
&q^{-2g\alpha_1}S_{-\alpha_1}^{\textrm{e};1}(h;g-1)\\
&\qquad= \frac{q^{1+2\alpha_1}}{|h_1|^{1/2+\alpha_1}}  \frac{1}{2 \pi i}\oint_{|u|=r_1'} \frac{\mathcal{C}_{-\alpha_1}(q^{2\alpha_1}/ u^2,q^{-\alpha_1})\mathcal{D}_{-\alpha_1}(h;q^{2\alpha_1}/ u^2,q^{-\alpha_1}) (q^{1+2\alpha_1}-u^2) du }{u^{2g-2\left[\frac{g-1+d(h_1)}{2}\right]-3}(1-u^2)(q^{2\alpha_1}-u^2)^2(q^{2\alpha_1+2}-u^2)},
\end{align*}
where $r_1'=q^{3/4+\Re \alpha_1-\varepsilon/2}$. It is standard to verify that
\[
\mathcal{C}_{-\alpha_1}\Big(\frac{q^{2\alpha_1}}{u^2},q^{-\alpha_1}\Big)=\mathcal{Z}\Big(\frac{1}{q^{1-2\alpha_1}u^2}\Big)^{-1}\mathcal{Z}\Big(\frac{u^2}{q^{3+2\alpha_1}}\Big)^{-1}\mathcal{Z}\Big(\frac{u^2}{q^{2+2\alpha_1}}\Big)\mathcal{A}_{\alpha_1}(u)
\]
and
\[
\mathcal{D}_{-\alpha_1}\Big(h;\frac{q^{2\alpha_1}}{u^2},q^{-\alpha_1}\Big)=|h_1|^{\alpha_1}\mathcal{B}_{\alpha_1}(h;u).
\]
So we obtain that
\begin{align}\label{2000}
&q^{-2g\alpha_1}S_{-\alpha_1}^{\textrm{e};1}(h;g-1)=- \frac{1}{\sqrt{|h_1|}}  \frac{1}{2 \pi i}\oint_{|u|=r_1'} \frac{\mathcal{A}_{\alpha_1}(u) \mathcal{B}_{\alpha_1}(h;u) du }{u^{2g-2\left[\frac{g-1+d(h_1)}{2}\right]-1}(1-u^2)(1-q^{-2\alpha_1}u^2)}\nonumber\\
&\qquad\qquad=- \frac{1}{2\sqrt{|h_1|}} \bigg( \frac{1}{2 \pi i}\oint_{|u|=r_1'} \frac{\mathcal{A}_{\alpha_1}(u) \mathcal{B}_{\alpha_1}(h;u) du }{u^{g-d(h_1)+1}(1-u)(1-q^{-2\alpha_1}u^2)}\nonumber\\
&\qquad\qquad\qquad\qquad\qquad\qquad\qquad+\frac{1}{2 \pi i}\oint_{|u|=r_1'} \frac{(-1)^{g-d(h_1)}\mathcal{A}_{\alpha_1}(u) \mathcal{B}_{\alpha_1}(h;u) du }{u^{g-d(h_1)+1}(1+u)(1-q^{-2\alpha_1}u^2)}\bigg)\\
&\qquad\qquad=- \frac{1}{\sqrt{|h_1|}} \frac{1}{2 \pi i}\oint_{|u|=r_1'} \frac{\mathcal{A}_{\alpha_1}(u) \mathcal{B}_{\alpha_1}(h;u) du }{u^{g-d(h_1)+1}(1-u)(1-q^{-2\alpha_1}u^2)}\nonumber,
\end{align}
by making the change of variables $u \rightarrow -u$ in the integral in \eqref{2000} and using the fact that $\mathcal{A}_{\alpha_1}(u)$, $\mathcal{B}_{\alpha_1}(h;u)$ are even functions.

Combining with \eqref{2001} we have
\begin{align}\label{3000}
&M_{\alpha_1}(h;g)+q^{-2g\alpha_1}S_{-\alpha_1}^{\textrm{e};1}(h;g-1)\nonumber\\
&\quad=\frac{1}{\sqrt{|h_1|}} \frac{1}{2 \pi i}\bigg(\oint_{|u|=r}-\oint_{|u|=r_1'}\bigg) \frac{\mathcal{A}_{\alpha_1}(u) \mathcal{B}_{\alpha_1}(h;u) du }{u^{g-d(h_1)+1}(1-u)(1-q^{-2\alpha_1}u^2)}+O_\varepsilon\big(q^{-3g/2-g\Re \alpha_1+\varepsilon g}\big)\nonumber\\
&\quad=-\text{Res}(u=1)-\text{Res}(u=\pm q^{\alpha_1})+O_\varepsilon\big(q^{-3g/2-g\Re \alpha_1+\varepsilon g}\big)\nonumber\\
&\quad=\frac{\widetilde{\mathcal{S}}_{\alpha_1}(h)}{\sqrt{|h_1|}}+\frac{\mathcal{A}_{\alpha_1}(q^{\alpha_1})\mathcal{B}_{\alpha_1}(h;q^{\alpha_1})}{\sqrt{|h_1|}}q^{-2\alpha_1\left[\frac{g-d(h_1)}{2}\right]}\zeta_q(1-2\alpha_1)+O_\varepsilon\big(q^{-3g/2-g\Re \alpha_1+\varepsilon g}\big),
\end{align}
like in \eqref{2002} and \eqref{2003}. Similarly,
\begin{align}\label{3001}
&q^{-2g\alpha_1}M_{-\alpha_1}(h;g-1)+S_{\alpha_1}^{\textrm{e};1}(h;g)=\frac{q^{-2g\alpha_1}}{\sqrt{|h_1|}}\widetilde{\mathcal{S}}_{-\alpha_1}(h)\\
&\qquad+\frac{\mathcal{A}_{-\alpha_1}(q^{-\alpha_1})\mathcal{B}_{-\alpha_1}(h;q^{-\alpha_1})}{\sqrt{|h_1|}}q^{-2g\alpha_1+2\alpha_1\left[\frac{g-d(h_1)}{2}\right]}\zeta_q(1+2\alpha_1)+O_\varepsilon\big(q^{-3g/2-g\Re \alpha_1+\varepsilon g}\big).\nonumber
\end{align}

By \eqref{boundfornonsquare}, \eqref{6003}, \eqref{3002}, \eqref{3000} and \eqref{3001} and \eqref{switchingalpha} the total error is
\begin{align*}
&\ll_\varepsilon q^{-2g\mathfrak{a}_1\Re\alpha_1+\varepsilon g}\Big(|h|^{1/2}q^{-3g/2-g|\Re\alpha_1|}+|h_1|^{1/4}q^{-3g/2}\\
&\qquad\qquad\qquad\qquad\qquad\qquad\qquad+q^{-2g|\Re\alpha_1|}\big(|h|^{1/2}q^{-3g/2+g|\Re\alpha_1|}+|h_1|^{1/4}q^{-3g/2+g|\Re\alpha_1|}\big)\Big)\\
&\ll_\varepsilon |h|^{1/2}q^{-3g/2+g|\Re\alpha_1|+\varepsilon g}+|h_1|^{1/4}q^{-3g/2+2g|\Re\alpha_1|+\varepsilon g},
\end{align*}
and we obtain the error term in \eqref{formulaE1}.

To prove Theorem \ref{theorem2} for $k=1$ we are left to verify that
\begin{align*}
&\frac{\mathcal{A}_{\alpha_1}(q^{\alpha_1})\mathcal{B}_{\alpha_1}(h;q^{\alpha_1})}{\sqrt{|h_1|}} q^{-2\alpha_1\left[\frac{g-d(h_1)}{2}\right]}\zeta_q(1-2\alpha_1)\\
&\qquad\qquad+\frac{\mathcal{A}_{-\alpha_1}(q^{-\alpha_1})\mathcal{B}_{-\alpha_1}(h;q^{-\alpha_1})}{\sqrt{|h_1|}} q^{-2g\alpha_1+2\alpha_1\left[\frac{g-1-d(h_1)}{2}\right]}\zeta_q(1+2\alpha_1)=0
\end{align*}
This is easily seen to be the case by noticing that
\[
\mathcal{A}_{\alpha_1}(q^{\alpha_1})=\mathcal{A}_{-\alpha_1}(q^{-\alpha_1}),\qquad \mathcal{B}_{\alpha_1}(h;q^{\alpha_1})=|h_1|^{-2\alpha_1}\mathcal{B}_{-\alpha_1}(h;q^{-\alpha_1}),
\]
\[
\zeta_q(1-2\alpha_1)=-q^{-2\alpha_1}\zeta_q(1+2\alpha_1)\qquad\text{and}\qquad \Big[\frac{n}{2}\Big]+\Big[\frac{n+1}{2}\Big]=n.
\]

\subsubsection{The case $k=2$}
An exercise with the Euler product \eqref{theEulerproduct}  shows that
\begin{align*}
\mathcal{N}_{\gamma_1,\gamma_2}  (h;u,w) &= \frac{|h|u^{d(h_2)}}{\sqrt{|h_1 |}}\mathcal{C}_{\gamma_1,\gamma_2}(u,w)\mathcal{D}_{\gamma_1,\gamma_2}(h;u,w) \\
&\qquad\qquad \times \mathcal{Z} \Big( \frac{1}{q^2u} \Big) \mathcal{Z}(u)\mathcal{Z}\Big(\frac {w}{q^{1+\gamma_1}}\Big)\mathcal{Z}\Big(\frac {w}{q^{1+\gamma_2}}\Big) \mathcal{Z}\Big(\frac{uw^2}{q^{1+2\gamma_1}}\Big) \mathcal{Z}\Big(\frac{uw^2}{q^{1+2\gamma_2}}\Big),
\end{align*}
 where 
\begin{align*}
&\mathcal{C}_{\gamma_1,\gamma_2}(u,w) =\prod_{P\in\mathcal{P}}\bigg( 1-\frac{ w^{d(P)}}{|P|^{1+\gamma_1}}\bigg)\bigg( 1-\frac{ w^{d(P)}}{|P|^{1+\gamma_2}}\bigg)\\
&\quad \times\bigg(1 +\frac{w^{d(P)}\big(1-u^{d(P)}\big)}{|P|^{1+\gamma_1}}+\frac{w^{d(P)}\big(1-u^{d(P)}\big)}{|P|^{1+\gamma_2}}+\frac{(uw^2)^{d(P)}}{|P|^{1+\gamma_1+\gamma_2}}-\frac{(uw^2)^{d(P)}}{|P|^{2+2\gamma_1}}-\frac{(uw^2)^{d(P)}}{|P|^{2+2\gamma_2}}\\
&\qquad\qquad-\frac{(uw^2)^{d(P)}}{|P|^{2+\gamma_1+\gamma_2}}-\frac{1}{|P|^2u^{d(P)}}+\frac{w^{ 2d(P)}}{|P|^{3+2\gamma_1}}+\frac{w^{ 2d(P)}}{|P|^{3+2\gamma_2}}+\frac{(uw^2)^{2d(P)}}{|P|^{3+2\gamma_1+2\gamma_2}}-\frac{(uw^4)^{d(P)}}{|P|^{4+2\gamma_1+2\gamma_2}}\bigg)
 \end{align*}
 and $\mathcal{D}_{\gamma_1,\gamma_2}(h;u,w)$ is given by
\begin{align*}
&\prod_{P|h}\bigg(1 +\frac{w^{d(P)}\big(1-u^{d(P)}\big)}{|P|^{1+\gamma_1}}+\frac{w^{d(P)}\big(1-u^{d(P)}\big)}{|P|^{1+\gamma_2}}+\frac{(uw^2)^{d(P)}}{|P|^{1+\gamma_1+\gamma_2}}-\frac{1}{|P|^2u^{d(P)}}-\frac{(uw^2)^{d(P)}}{|P|^{2+2\gamma_1}}\\
&\qquad\quad-\frac{(uw^2)^{d(P)}}{|P|^{2+2\gamma_2}}-\frac{(uw^2)^{d(P)}}{|P|^{2+\gamma_1+\gamma_2}}+\frac{w^{ 2d(P)}}{|P|^{3+2\gamma_1}}+\frac{w^{ 2d(P)}}{|P|^{3+2\gamma_2}}+\frac{(uw^2)^{2d(P)}}{|P|^{3+2\gamma_1+2\gamma_2}}-\frac{(uw^4)^{d(P)}}{|P|^{4+2\gamma_1+2\gamma_2}}\bigg)^{-1}\\
&\quad\times \prod_{P | h_1} \bigg(1-u^{d(P)}+\frac{(uw)^{d(P)}}{|P|^{\gamma_1}}+\frac{(uw)^{d(P)}}{|P|^{\gamma_2}}-\frac{(uw)^{d(P)}}{|P|^{1+\gamma_1}}-\frac{(uw)^{d(P)}}{|P|^{1+\gamma_2}}+\frac{(uw^2)^{d(P)}\big(1-u^{d(P)}\big)}{|P|^{1+\gamma_1+\gamma_2}}\bigg)\\
&\quad\times\prod_{\substack{P \nmid h_1 \\ P | h_2}} \Big(1-\frac{1}{|P|}  +\frac{w^{d(P)}\big(1-u^{d(P)}\big)}{|P|^{1+\gamma_1}}+\frac{w^{d(P)}\big(1-u^{d(P)}\big)}{|P|^{1+\gamma_2}}+\frac{(uw^2)^{d(P)}}{|P|^{1+\gamma_1+\gamma_2}}-\frac{(uw^2)^{d(P)}}{|P|^{2+\gamma_1+\gamma_2}}\Big).\end{align*}
Note that $\mathcal{C}_{\gamma_1,\gamma_2}(u,w)$ converges absolutely  for $|u| > 1/q$, $|w| < 
q^{1/2+\min\{\Re \gamma_1,\Re \gamma_2\}}$, $|uw| < q^{\min\{\Re \gamma_1,\Re \gamma_2\}}$ and $|uw^2| < q^{\Re(\gamma_1+\gamma_2)}$. Hence, moving the $u$-contour to $|u|=q^{-1+\varepsilon}$ we obtain
\begin{align*}
&S_{\alpha_1,\alpha_2}^{\textrm{e}}(h;2g;V=\square)=\ - \frac{q}{(q-1)\sqrt{|h_1|}} \frac{1}{(2 \pi i)^2} \oint_{|u|=q^{-1+\varepsilon}}  \oint_{|w|=r_2}\nonumber\\
&\qquad\qquad \frac{\mathcal{C}_{\alpha_1,\alpha_2}(u,w)\mathcal{D}_{\alpha_1,\alpha_2}(h;u,w) }{(1-u)(1-q^{-\alpha_1}w)(1-q^{-\alpha_2}w) (1-uw^2)(1-q^{-2\alpha_1}uw^2)(1-q^{-2\alpha_2}uw^2)}\nonumber\\
&\qquad\qquad\qquad\qquad \times\frac{ dwdu  }{u^{\left[\frac{d(h_1)}{2}\right]}w^{2g-2\left[\frac{d(h_1)+1}{2}\right]+d(h_1)+1}}+O_\varepsilon\big(q^{-g-2g\min\{\Re \alpha_j\}+\varepsilon g}\big).
\end{align*}

We shift the $w$-contour to $|w|=q^{1/2-\varepsilon}$ (for $S_{-\alpha_2,-\alpha_1}^{\textrm{e}}(h;2g-1;V=\square)$ we move the contour to $|w|=q^{1/2+\min\{-\Re \alpha_1,-\Re \alpha_2\}-\varepsilon}$ instead).  In doing so, we encounter two simple poles at $w=q^{\alpha_1}$ and $w=q^{\alpha_2}$. Moreover, the new integral is bounded by $O_\varepsilon(q^{-g+\varepsilon g})$. Hence

\begin{align*}
&S_{\alpha_1,\alpha_2}^{\textrm{e}}(h;2g;V=\square)\\
&\quad=- \frac{q^{-2g\alpha_1+2\alpha_1\left[\frac{d(h_1)+1}{2}\right]+1}\zeta_q(1-\alpha_1+\alpha_2)}{(q-1)|h_1|^{1/2+\alpha_1}} \\
&\qquad\qquad\qquad\frac{1}{2 \pi i} \oint_{|u|=q^{-1+\varepsilon}} \frac{\mathcal{C}_{\alpha_1,\alpha_2}(u,q^{\alpha_1})\mathcal{D}_{\alpha_1,\alpha_2}(h;u,q^{\alpha_1}) du}{u^{\left[\frac{d(h_1)}{2}\right]}(1-u)^2(1-q^{2\alpha_1}u)(1-q^{2(\alpha_1-\alpha_2)}u) }\\
&\quad\qquad- \frac{q^{-2g\alpha_2+2\alpha_2\left[\frac{d(h_1)+1}{2}\right]+1}\zeta_q(1+\alpha_1-\alpha_2)}{(q-1)|h_1|^{1/2+\alpha_2}} \\
&\qquad\qquad\qquad\frac{1}{2 \pi i} \oint_{|u|=q^{-1+\varepsilon}} \frac{\mathcal{C}_{\alpha_1,\alpha_2}(u,q^{\alpha_2})\mathcal{D}_{\alpha_1,\alpha_2}(h;u,q^{\alpha_2}) du}{u^{\left[\frac{d(h_1)}{2}\right]}(1-u)^2(1-q^{2\alpha_2}u)(1-q^{2(\alpha_2-\alpha_1)}u) }+O_\varepsilon(q^{-g+\varepsilon g})\\
&\quad=S_{\alpha_1,\alpha_2}^{\textrm{e};1}(h;2g)+S_{\alpha_1,\alpha_2}^{\textrm{e};2}(h;2g)+O_\varepsilon(q^{-g+\varepsilon g}),
\end{align*}
say. Similarly we have
\begin{align*}
S_{-\alpha_2,-\alpha_1}^{\textrm{e}}(h;2g-1;V=\square)&=S_{-\alpha_2,-\alpha_1}^{\textrm{e};1}(h;2g-1)+S_{-\alpha_2,-\alpha_1}^{\textrm{e};2}(h;2g-1)\\
&\qquad\qquad+O_\varepsilon\big(q^{-g+2g\max\{\Re \alpha_1,\Re \alpha_2\}+\varepsilon g}\big),
\end{align*}
where
\begin{align*}
S_{-\alpha_2,-\alpha_1}^{\textrm{e};1}(h;2g-1)&=- \frac{q^{2(g-1)\alpha_1-2\alpha_1\left[\frac{d(h_1)}{2}\right]+1}\zeta_q(1+\alpha_1-\alpha_2)}{(q-1)|h_1|^{1/2-\alpha_1}} \\
&\qquad\qquad\frac{1}{2 \pi i} \oint_{|u|=q^{-1+\varepsilon}} \frac{\mathcal{C}_{-\alpha_2,-\alpha_1}(u,q^{-\alpha_1})\mathcal{D}_{-\alpha_2,-\alpha_1}(h;u,q^{-\alpha_1}) du}{u^{\left[\frac{d(h_1)-1}{2}\right]}(1-u)^2(1-q^{-2\alpha_1}u)(1-q^{2(\alpha_2-\alpha_1)}u) }
\end{align*}
and
\begin{align}\label{1expression}
S_{-\alpha_2,-\alpha_1}^{\textrm{e};2}(h;2g-1)&=- \frac{q^{2(g-1)\alpha_2-2\alpha_2\left[\frac{d(h_1)}{2}\right]+1}\zeta_q(1-\alpha_1+\alpha_2)}{(q-1)|h_1|^{1/2-\alpha_2}} \\
&\qquad\qquad\frac{1}{2 \pi i} \oint_{|u|=q^{-1+\varepsilon}} \frac{\mathcal{C}_{-\alpha_2,-\alpha_1}(u,q^{-\alpha_2})\mathcal{D}_{-\alpha_2,-\alpha_1}(h;u,q^{-\alpha_2}) du}{u^{\left[\frac{d(h_1)-1}{2}\right]}(1-u)^2(1-q^{-2\alpha_2}u)(1-q^{2(\alpha_1-\alpha_2)}u) }.\nonumber
\end{align}

Combining with \eqref{boundfornonsquare}, \eqref{MCformula} and \eqref{switchingalpha} we see that the total error is
\begin{align*}
&\ll_\varepsilon q^{-2g(\mathfrak{a}_1\Re\alpha_1+\mathfrak{a}_2\Re\alpha_2)+\varepsilon g}\Big(|h|^{1/2}q^{-g-2g\min\{|\Re\alpha_1|,|\Re\alpha_2|\}}+q^{-g}\\
&\qquad\qquad\qquad\qquad\qquad\qquad\qquad+q^{-2g(|\Re\alpha_1|+|\Re\alpha_2|)}|h|^{1/2}q^{-g+2g\max\{|\Re\alpha_1|,|\Re\alpha_2|\}}\Big)\\
&\ll_\varepsilon |h|^{1/2}q^{-g+2g\max\{|\Re\alpha_1|,|\Re\alpha_2|\}+\varepsilon g}+q^{-g+4g\max\{|\Re\alpha_1|,|\Re\alpha_2|\}+\varepsilon g},
\end{align*}
and we obtain the expression for $\widetilde{E}_2$ in \eqref{formulaE2}.

For the main term, we have
\begin{align*}
\mathcal{A}_{\gamma_1,\gamma_2}(u)&=\prod_P\bigg(1-\frac{u^{2d(P)}}{|P|^{1+\gamma_1+\gamma_2}}\bigg)\bigg(1+\frac{1}{|P|}\bigg)^{-1}\\
&\qquad\qquad\times\bigg(1+\frac{u^{2d(P)}}{|P|^{1+\gamma_1+\gamma_2}}+\frac{1}{|P|}\bigg(1-\frac{u^{2d(P)}}{|P|^{1+2\gamma_1}}\bigg)\bigg(1-\frac{u^{2d(P)}}{|P|^{1+2\gamma_2}}\bigg)\bigg)
\end{align*}
and
\begin{align*}
\mathcal{B}_{\gamma_1,\gamma_2}(h;u)&=\prod_{P|h}\bigg(1+\frac{u^{2d(P)}}{|P|^{1+\gamma_1+\gamma_2}}+\frac{1}{|P|}\bigg(1-\frac{u^{2d(P)}}{|P|^{1+2\gamma_1}}\bigg)\bigg(1-\frac{u^{2d(P)}}{|P|^{1+2\gamma_2}}\bigg)\bigg)^{-1}\\
&\qquad\qquad\times\prod_{P|h_1}\bigg(\frac{1}{|P|^{\gamma_1}}+\frac{1}{|P|^{\gamma_2}}\bigg)\prod_{\substack{P\nmid h_1\\P|h_2}}\bigg(1+\frac{u^{2d(P)}}{|P|^{1+\gamma_1+\gamma_2}}\bigg).
\end{align*}
As
\begin{equation*}
\mathcal{A}_{\gamma_1,\gamma_2}(q^{(\gamma_1+\gamma)/2})=\mathcal{A}_{-\gamma_2,-\gamma_1}(q^{-(\gamma_1+\gamma_2)/2}),
\end{equation*}
\begin{equation}\label{relation1}
\mathcal{B}_{\gamma_1,\gamma_2}(h;q^{(\gamma_1+\gamma)/2})=|h_1|^{-(\gamma_1+\gamma_2)}\mathcal{B}_{-\gamma_2,-\gamma_1}(h;q^{-(\gamma_1+\gamma_2)/2})
\end{equation}
\[
\zeta_q(1-\gamma_1-\gamma_2)=-q^{-(\gamma_1+\gamma_2)}\zeta_q(1+\gamma_1+\gamma_2),
\]
it follows that
\begin{equation}\label{relation2}
M_{\alpha_1,\alpha_2}^{1,2}(h;2g)+q^{-2g(\alpha_1+\alpha_2)}M_{-\alpha_2,-\alpha_1}^{1,2}(h;2g-1)=0.
\end{equation}
So to prove Theorem \ref{theorem2} for $k=2$ it suffices to show that
\begin{align}\label{toverify}
\frac{q^{-2g\alpha_1}}{\sqrt{|h_1|}}\widetilde{\mathcal{S}}_{-\alpha_1,\alpha_2}(h)&=M_{\alpha_1,\alpha_2}^{1,1}(h;2g)+q^{-2g(\alpha_1+\alpha_2)}M_{-\alpha_2,-\alpha_1}^{1,1}(h;2g-1)\nonumber\\
&\qquad\qquad+S_{\alpha_1,\alpha_2}^{\textrm{e};1}(h;2g)+q^{-2g(\alpha_1+\alpha_2)}S_{-\alpha_2,-\alpha_1}^{\textrm{e};2}(h;2g-1)
\end{align}
and
\begin{align*}
\frac{q^{-2g\alpha_2}}{\sqrt{|h_1|}}\widetilde{\mathcal{S}}_{\alpha_1,-\alpha_2}(h)&=M_{\alpha_1,\alpha_2}^{2,2}(h;2g)+q^{-2g(\alpha_1+\alpha_2)}M_{-\alpha_2,-\alpha_1}^{2,2}(h;2g-1)\\
&\qquad\qquad+S_{\alpha_1,\alpha_2}^{\textrm{e};2}(h;2g)+q^{-2g(\alpha_1+\alpha_2)}S_{-\alpha_2,-\alpha_1}^{\textrm{e};1}(h;2g-1).
\end{align*}
These two identities are similar so we only need to verify \eqref{toverify}.

We next focus on $S_{\alpha_1,\alpha_2}^{\textrm{e};1}(h;2g)$. We have
\begin{equation}\label{2expression}
\mathcal{C}_{\gamma_1,\gamma_2}(u,q^{\gamma_1})=\bigg(1-\frac{u^{d(P)}}{|P|}\bigg)\widetilde{\mathcal{C}}_{\gamma_1,\gamma_2}(u),
\end{equation}
where
\begin{align*}
\widetilde{\mathcal{C}}_{\gamma_1,\gamma_2}(u)&=\prod_{P\in\mathcal{P}}\bigg(1-\frac{1}{|P|}\bigg)\bigg(1-\frac{1}{|P|^{1-\gamma_1+\gamma_2}}\bigg)\\
&\qquad\qquad\times\bigg(1+\frac{1}{|P|}+\frac{1}{|P|^{1-\gamma_1+\gamma_2}}+\frac{1}{|P|^{3-2\gamma_1+2\gamma_2}}-\frac{1}{|P|^2u^{d(P)}}-\frac{u^{d(P)}}{|P|^{2-2\gamma_1+2\gamma_2}}\bigg)\\
&=\widetilde{\mathcal{C}}_{-\gamma_2,-\gamma_1}\Big(\frac{q^{-2(\gamma_1-\gamma_2)}}{u}\Big).
\end{align*}
Similarly, note that
\begin{align*}
\mathcal{D}_{\gamma_1,\gamma_2}(h;u)&:=\mathcal{D}_{\gamma_1,\gamma_2}(h;u,q^{\gamma_1})\\
&=\prod_{P|h}\bigg(1+\frac{1}{|P|}+\frac{1}{|P|^{1-\gamma_1+\gamma_2}}+\frac{1}{|P|^{3-2\gamma_1+2\gamma_2}}-\frac{1}{|P|^2u^{d(P)}}-\frac{u^{d(P)}}{|P|^{2-2\gamma_1+2\gamma_2}}\bigg)^{-1}\\
&\qquad\qquad\prod_{P|h_1}\bigg(1+\frac{u^{d(P)}}{|P|^{\gamma_2-\gamma_1}}\bigg)\prod_{\substack{P\nmid h_1\\P|h_2}}\bigg(1+\frac{1}{|P|^{1-\gamma_1+\gamma_2}}\bigg)\\
&=u^{d(h_1)}|h_1|^{\gamma_1-\gamma_2}\mathcal{D}_{-\gamma_2,-\gamma_1}\Big(h;\frac{q^{-2(\gamma_1-\gamma_2)}}{u}\Big).
\end{align*}
Hence
\begin{align*}
S_{\alpha_1,\alpha_2}^{\textrm{e};1}(h;2g)&=-\frac{q^{-2g\alpha_1+2\alpha_1\left[\frac{d(h_1)+1}{2}\right]+1}\zeta_q(1-\alpha_1+\alpha_2)}{(q-1)|h_1|^{1/2+\alpha_1}} \\
&\qquad\qquad\qquad\frac{1}{2 \pi i} \oint_{|u|=q^{-1+\varepsilon}} \frac{\widetilde{\mathcal{C}}_{\alpha_1,\alpha_2}(u)\mathcal{D}_{\alpha_1,\alpha_2}(h;u) du}{u^{\left[\frac{d(h_1)}{2}\right]}(1-u)(1-q^{2\alpha_1}u)(1-q^{2(\alpha_1-\alpha_2)}u) }\\
&=\frac{q^{-2g\alpha_1-2\alpha_2-2\alpha_2\left[\frac{d(h_1)}{2}\right]+1}\zeta_q(1-\alpha_1+\alpha_2)}{(q-1)|h_1|^{1/2-\alpha_2}} \\
&\qquad\qquad\qquad\frac{1}{2 \pi i} \oint_{|u|=r_1'} \frac{\widetilde{\mathcal{C}}_{-\alpha_2,-\alpha_1}(u)\mathcal{D}_{-\alpha_2,-\alpha_1}(h;u) du}{u^{\left[\frac{d(h_1)-1}{2}\right]}(1-u)(1-q^{-2\alpha_2}u)(1-q^{2(\alpha_1-\alpha_2)}u) },
\end{align*}
where $r_1'=q^{1-2\Re(\alpha_1-\alpha_2)-\varepsilon}$, by changing the variables $u\rightarrow q^{-2(\alpha_1-\alpha_2)}/u$.

Furthermore, combining \eqref{1expression} and \eqref{2expression} we obtain
\begin{align*}
&q^{-2g(\alpha_1+\alpha_2)}S_{-\alpha_2,-\alpha_1}^{\textrm{e};2}(h;2g-1)=-\frac{q^{-2g\alpha_1-2\alpha_2-2\alpha_2\left[\frac{d(h_1)}{2}\right]+1}\zeta_q(1-\alpha_1+\alpha_2)}{(q-1)|h_1|^{1/2-\alpha_2}} \\
&\qquad\qquad\times\frac{1}{2 \pi i} \oint_{|u|=q^{-1+\varepsilon}} \frac{\widetilde{\mathcal{C}}_{-\alpha_2,-\alpha_1}(u)\mathcal{D}_{-\alpha_2,-\alpha_1}(h;u) du}{u^{\left[\frac{d(h_1)-1}{2}\right]}(1-u)(1-q^{-2\alpha_2}u)(1-q^{2(\alpha_1-\alpha_2)}u) }.
\end{align*}
So
\begin{align*}
&S_{\alpha_1,\alpha_2}^{\textrm{e};1}(h;2g)+q^{-2g(\alpha_1+\alpha_2)}S_{-\alpha_2,-\alpha_1}^{\textrm{e};2}(h;2g-1)=\frac{q^{-2g\alpha_1-2\alpha_2-2\alpha_2\left[\frac{d(h_1)}{2}\right]+1}\zeta_q(1-\alpha_1+\alpha_2)}{(q-1)|h_1|^{1/2-\alpha_2}} \\
&\qquad\qquad\times\bigg(\frac{1}{2 \pi i} \oint_{|u|=r_1'}-\frac{1}{2 \pi i} \oint_{|u|=q^{-1+\varepsilon}}\bigg) \frac{\widetilde{\mathcal{C}}_{-\alpha_2,-\alpha_1}(u)\mathcal{D}_{-\alpha_2,-\alpha_1}(h;u) du}{u^{\left[\frac{d(h_1)-1}{2}\right]}(1-u)(1-q^{-2\alpha_2}u)(1-q^{2(\alpha_1-\alpha_2)}u) }\\
&\qquad=\text{Res}(u=1)+\text{Res}(u=q^{2\alpha_2})+\text{Res}\big(u=q^{2(\alpha_2-\alpha_1)}\big).
\end{align*}


It is straightforward to verify that
\[
\widetilde{\mathcal{C}}_{-\alpha_2,-\alpha_1}(q^{2\alpha_2})=\frac{\mathcal{A}_{-\alpha_1,\alpha_2}(1)}{\zeta_q(2)}\qquad\text{and}\qquad \mathcal{D}_{-\alpha_2,-\alpha_1}(h;q^{2\alpha_2})=|h_1|^{\alpha_2}\mathcal{B}_{-\alpha_1,\alpha_2}(h;1).
\]
So like in \eqref{relation1} and \eqref{relation2} it follows that
\[
\frac{q^{-2g\alpha_1}}{\sqrt{|h_1|}}\widetilde{\mathcal{S}}_{-\alpha_1,\alpha_2}(h)=\text{Res}(u=q^{2\alpha_2}).
\]
Similarly we have
\[
M_{\alpha_1,\alpha_2}^{1,1}(h;2g)+\text{Res}\big(u=q^{2(\alpha_2-\alpha_1)}\big)=0
\]
and
\[
q^{-2g(\alpha_1+\alpha_2)}M_{-\alpha_2,-\alpha_1}^{1,1}(h;2g-1)+\text{Res}(u=1)=0.
\]
Thus \eqref{toverify} holds, and hence Theorem \ref{theorem2} holds for $k=2$.

 \subsubsection{The case $k=3$}
 
An exercise on the Euler product  \eqref{theEulerproduct} shows that
 \begin{align*}
&\mathcal{N}_{\gamma_1,\gamma_2,\gamma_3}(h;u,w) = \frac{|h| u^{d(h_2)}}{\sqrt{| h_1|}}\mathcal{C}_{\gamma_1,\gamma_2,\gamma_3}(u,w)\mathcal{D}_{\gamma_1,\gamma_2,\gamma_3}(h;u,w)\\
&\qquad\times\mathcal{Z} \Big(  \frac{1}{q^2u} \Big)\mathcal{Z}(u) \mathcal{Z}\Big(\frac {w}{q^{1+\gamma_1}}\Big)\mathcal{Z}\Big(\frac {w}{q^{1+\gamma_2}}\Big)\mathcal{Z}\Big(\frac {w}{q^{1+\gamma_3}}\Big)   \mathcal{Z}\Big(\frac{uw}{q^{1+\gamma_1}}\Big)^{-1}\mathcal{Z}\Big(\frac{uw}{q^{1+\gamma_2}}\Big)^{-1}\mathcal{Z}\Big(\frac{uw}{q^{1+\gamma_3}}\Big)^{-1}\\
&\qquad\times \mathcal{Z}\Big(\frac{uw^2}{q^{1+2\gamma_1}}\Big)  \mathcal{Z}\Big(\frac{uw^2}{q^{1+2\gamma_2}}\Big)  \mathcal{Z}\Big(\frac{uw^2}{q^{1+2\gamma_3}}\Big)  \mathcal{Z}\Big(\frac{uw^2}{q^{1+\gamma_1+\gamma_2}}\Big)  \mathcal{Z}\Big(\frac{uw^2}{q^{1+\gamma_2+\gamma_3}}\Big) \mathcal{Z}\Big(\frac{uw^2}{q^{1+\gamma_3+\gamma_1}}\Big)   ,
\end{align*}
 where
\begin{align*}
&\mathcal{C}_{\gamma_1,\gamma_2,\gamma_3}(u,w) =\prod_P\bigg (1-\frac{w^{d(P)}}{|P|^{1+\gamma_1}}\bigg)\bigg (1-\frac{w^{d(P)}}{|P|^{1+\gamma_2}}\bigg)\bigg (1-\frac{w^{d(P)}}{|P|^{1+\gamma_3}}\bigg)\\
& \times  \bigg(1-\frac{(uw)^{d(P)}}{|P|^{1+\gamma_1}}\bigg)^{-1}\bigg(1-\frac{(uw)^{d(P)}}{|P|^{1+\gamma_2}}\bigg)^{-1} \bigg(1-\frac{(uw)^{d(P)}}{|P|^{1+\gamma_3}}\bigg)^{-1}\\
&\times\bigg (1-\frac{(uw^2)^{d(P)}}{|P|^{1+\gamma_1+\gamma_2}}\bigg)\bigg (1-\frac{(uw^2)^{d(P)}}{|P|^{1+\gamma_2+\gamma_3}}\bigg)\bigg (1-\frac{(uw^2)^{d(P)}}{|P|^{1+\gamma_3+\gamma_1}}\bigg)\\
&\times\bigg(1+\frac{w^{d(P)}(1-u^{d(P)})}{|P|^{1+\gamma_1}}+\frac{w^{d(P)}(1-u^{d(P)})}{|P|^{1+\gamma_2}}+\frac{w^{d(P)}(1-u^{d(P)})}{|P|^{1+\gamma_3}}\\
&\qquad+\frac{(uw^2)^{d(P)}}{|P|^{1+\gamma_1+\gamma_2}}+\frac{(uw^2)^{d(P)}}{|P|^{1+\gamma_2+\gamma_3}}+\frac{(uw^2)^{d(P)}}{|P|^{1+\gamma_3+\gamma_1}}-\frac{(uw^2)^{d(P)}}{|P|^{2+2\gamma_1}}-\frac{(uw^2)^{d(P)}}{|P|^{2+2\gamma_2}}-\frac{(uw^2)^{d(P)}}{|P|^{2+2\gamma_3}}\\
&\qquad-\frac{(uw^2)^{d(P)}}{|P|^{2+\gamma_1+\gamma_2}}-\frac{(uw^2)^{d(P)}}{|P|^{2+\gamma_2+\gamma_3}}-\frac{(uw^2)^{d(P)}}{|P|^{2+\gamma_3+\gamma_1}}+\frac{(uw^3)^{d(P)}\big(1-u^{d(P)}\big)}{|P|^{2+\gamma_1+\gamma_2+\gamma_3}}-\frac{1}{|P|^2u^{d(P)}}\\
&\qquad+\frac{w^{2d(P)}}{|P|^{3+2\gamma_1}}+\frac{w^{2d(P)}}{|P|^{3+2\gamma_2}}+\frac{w^{2d(P)}}{|P|^{3+2\gamma_3}}+\frac{(uw^2)^{2d(P)}}{|P|^{3+2\gamma_1+2\gamma_2}}+\frac{(uw^2)^{2d(P)}}{|P|^{3+2\gamma_2+2\gamma_3}}+\frac{(uw^2)^{2d(P)}}{|P|^{3+2\gamma_3+2\gamma_1}}\\
&\qquad-\frac{(uw^4)^{d(P)}}{|P|^{4+2\gamma_1+2\gamma_2}}-\frac{(uw^4)^{d(P)}}{|P|^{4+2\gamma_2+2\gamma_3}}-\frac{(uw^4)^{d(P)}}{|P|^{4+2\gamma_3+2\gamma_1}}-\frac{(uw^2)^{3d(P)}}{|P|^{4+2\gamma_1+2\gamma_2+2\gamma_3}}+\frac{(uw^3)^{2d(P)}}{|P|^{5+2\gamma_1+2\gamma_2+2\gamma_3}}\bigg),
\end{align*}
and 
\begin{align*}
&\mathcal{D}_{\gamma_1,\gamma_2,\gamma_3}(h;u,w)=\prod_{P|h}\bigg(1+\frac{w^{d(P)}(1-u^{d(P)})}{|P|^{1+\gamma_1}}+\frac{w^{d(P)}(1-u^{d(P)})}{|P|^{1+\gamma_2}}+\frac{w^{d(P)}(1-u^{d(P)})}{|P|^{1+\gamma_3}}\\
&\qquad+\frac{(uw^2)^{d(P)}}{|P|^{1+\gamma_1+\gamma_2}}+\frac{(uw^2)^{d(P)}}{|P|^{1+\gamma_2+\gamma_3}}+\frac{(uw^2)^{d(P)}}{|P|^{1+\gamma_3+\gamma_1}}-\frac{(uw^2)^{d(P)}}{|P|^{2+2\gamma_1}}-\frac{(uw^2)^{d(P)}}{|P|^{2+2\gamma_2}}-\frac{(uw^2)^{d(P)}}{|P|^{2+2\gamma_3}}\\
&\qquad-\frac{(uw^2)^{d(P)}}{|P|^{2+\gamma_1+\gamma_2}}-\frac{(uw^2)^{d(P)}}{|P|^{2+\gamma_2+\gamma_3}}-\frac{(uw^2)^{d(P)}}{|P|^{2+\gamma_3+\gamma_1}}+\frac{(uw^3)^{d(P)}\big(1-u^{d(P)}\big)}{|P|^{2+\gamma_1+\gamma_2+\gamma_3}}-\frac{1}{|P|^2u^{d(P)}}\\
&\qquad+\frac{w^{2d(P)}}{|P|^{3+2\gamma_1}}+\frac{w^{2d(P)}}{|P|^{3+2\gamma_2}}+\frac{w^{2d(P)}}{|P|^{3+2\gamma_3}}+\frac{(uw^2)^{2d(P)}}{|P|^{3+2\gamma_1+2\gamma_2}}+\frac{(uw^2)^{2d(P)}}{|P|^{3+2\gamma_2+2\gamma_3}}+\frac{(uw^2)^{2d(P)}}{|P|^{3+2\gamma_3+2\gamma_1}}\\
&\qquad-\frac{(uw^4)^{d(P)}}{|P|^{4+2\gamma_1+2\gamma_2}}-\frac{(uw^4)^{d(P)}}{|P|^{4+2\gamma_2+2\gamma_3}}-\frac{(uw^4)^{d(P)}}{|P|^{4+2\gamma_3+2\gamma_1}}-\frac{(uw^2)^{3d(P)}}{|P|^{4+2\gamma_1+2\gamma_2+2\gamma_3}}+\frac{(uw^3)^{2d(P)}}{|P|^{5+2\gamma_1+2\gamma_2+2\gamma_3}}\bigg)^{-1}\\
&\ \times \prod_{P|h_1}\bigg(1-u^{d(P)}+\frac{(uw)^{d(P)}}{|P|^{\gamma_1}}+\frac{(uw)^{d(P)}}{|P|^{\gamma_2}}+\frac{(uw)^{d(P)}}{|P|^{\gamma_3}}-\frac{(uw)^{d(P)}}{|P|^{1+\gamma_1}}-\frac{(uw)^{d(P)}}{|P|^{1+\gamma_2}}-\frac{(uw)^{d(P)}}{|P|^{1+\gamma_3}}\\
&\qquad+\frac{(uw^2)^{d(P)}(1-u^{d(P)})}{|P|^{1+\gamma_1+\gamma_2}}+\frac{(uw^2)^{d(P)}(1-u^{d(P)})}{|P|^{1+\gamma_2+\gamma_3}}+\frac{(uw^2)^{d(P)}(1-u^{d(P)})}{|P|^{1+\gamma_3+\gamma_1}}\\
&\qquad+\frac{(u^2w^3)^{d(P)}}{|P|^{1+\gamma_1+\gamma_2+\gamma_3}}-\frac{(u^2w^3)^{d(P)}}{|P|^{2+\gamma_1+\gamma_2+\gamma_3}}\bigg)\\
&\ \times\prod_{\substack{P\nmid h_1\\P|h_2}} \bigg(1-\frac{1}{|P|}+\frac{w^{d(P)}(1-u^{d(P)})}{|P|^{1+\gamma_1}}+\frac{w^{d(P)}(1-u^{d(P)})}{|P|^{1+\gamma_2}}+\frac{w^{d(P)}(1-u^{d(P)})}{|P|^{1+\gamma_3}}\\
&\qquad+\frac{(uw^2)^{d(P)}}{|P|^{1+\gamma_1+\gamma_2}}+\frac{(uw^2)^{d(P)}}{|P|^{1+\gamma_2+\gamma_3}}+\frac{(uw^2)^{d(P)}}{|P|^{1+\gamma_3+\gamma_1}}-\frac{(uw^2)^{d(P)}}{|P|^{2+\gamma_1+\gamma_2}}-\frac{(uw^2)^{d(P)}}{|P|^{2+\gamma_2+\gamma_3}}-\frac{(uw^2)^{d(P)}}{|P|^{2+\gamma_3+\gamma_1}}\\
&\qquad+\frac{(uw^3)^{d(P)}(1-u^{d(P)})}{|P|^{2+\gamma_1+\gamma_2+\gamma_3}}\bigg).
\end{align*}
Note that $\mathcal{C}_{\gamma_1,\gamma_2,\gamma_3}(u,w)$ is absolutely convergent for $|u| > 1/q$, $|w| < q^{1/2+\min\{\Re \gamma_j\}}$, $|uw| <q^{1/2+\min\{\Re \gamma_j\}}$ and $|uw^2| <q^{1/2+\min_{i\ne j}\{\Re(\gamma_i+\gamma_j)\}}$
. 
We hence obtain
\begin{align*}
&S_{\gamma_1,\gamma_2,\gamma_3}^{\textrm{e}}(h;N;V=\square) =- \frac{q}{(q-1)|h_1|^{1/2}} \frac{1}{(2 \pi i)^2} \oint_{|u|=q^{-1+\varepsilon}}  \oint_{|w|=r_2} \\
&\qquad\frac{(1-q^{-\gamma_1}uw)(1-q^{-\gamma_2}uw)(1-q^{-\gamma_3}uw)\mathcal{C}_{\gamma_1,\gamma_2,\gamma_3}(u,w)\mathcal{D}_{\gamma_1,\gamma_2,\gamma_3}(h;u,w)}{(1-u)(1-q^{-\gamma_1}w)(1-q^{-\gamma_2}w)(1-q^{-\gamma_3}w)(1-uw^2)}\\
&\qquad\qquad\times\frac{1}{(1-q^{-2\gamma_1}uw^2)(1-q^{-2\gamma_2}uw^2)(1-q^{-2\gamma_3}uw^2)(1-q^{-(\gamma_1+\gamma_2)}uw^2)}\\
&\qquad\qquad\times\frac{1}{(1-q^{-(\gamma_2+\gamma_3)}uw^2)(1-q^{-(\gamma_3+\gamma_1)}uw^2)} \frac{ dwdu }{u^{\left[\frac{N+d(h_1)}{2}\right]-g}w^{2\left[\frac{N-d(h_1)}{2}\right]+d(h_1)+1}}\\
&\qquad\qquad\qquad +O_\varepsilon\big(q^{-g/2-3g\min\{\Re \gamma_j\}+\varepsilon g}\big).\nonumber
\end{align*}

We move the $w$-contour to $|w|=q^{1/2+\min\{0,\Re \gamma_j\}-\varepsilon}$, crossing three simple poles at $w=q^{\gamma_j}$, $1\leq j\leq 3$. 
The new integral is bounded trivially by $O_\varepsilon(q^{-g-3g\min\{0,\Re \gamma_j\}+\varepsilon g})$. So we get
\begin{align}\label{6000}
&S_{\gamma_1,\gamma_2,\gamma_3}^{\textrm{e}}(h;N;V=\square)\nonumber \\
&\qquad=\text{Res}(w=q^{\gamma_1})+\text{Res}(w=q^{\gamma_2})+\text{Res}(w=q^{\gamma_3})\nonumber\\
&\qquad\qquad\qquad+O_\varepsilon\big(q^{-g/2-3g\min\{\Re \gamma_j\}+\varepsilon g}\big)+O_\varepsilon\big(q^{-g-3g\min\{0,\Re \gamma_j\}+\varepsilon g}\big)\nonumber\\
&\qquad=S_{\gamma_1,\gamma_2,\gamma_3}^{\textrm{e}}(h;N)+S_{\gamma_2,\gamma_3,\gamma_1}^{\textrm{e}}(h;N)+S_{\gamma_3,\gamma_1,\gamma_2}^{\textrm{e}}(h;N)\nonumber\\
&\qquad\qquad\qquad+O_\varepsilon\big(q^{-g/2-3g\min\{\Re \gamma_j\}+\varepsilon g}\big)+O_\varepsilon\big(q^{-g-3g\min\{0,\Re \gamma_j\}+\varepsilon g}\big),
\end{align}
say, where
\begin{align*}
&S_{\gamma_1,\gamma_2,\gamma_3}^{\textrm{e}}(h;N) =- \frac{q^{-2\gamma_1\left[\frac{N-d(h_1)}{2}\right]+1}\zeta_q(1-\gamma_1+\gamma_2)\zeta_q(1-\gamma_1+\gamma_3)}{(q-1)|h_1|^{1/2+\gamma_1}}\frac{1}{2 \pi i }\oint_{|u|=q^{-1+\varepsilon}} \\
&\qquad\qquad \frac{\mathcal{C}_{\gamma_1,\gamma_2,\gamma_3}(u,q^{\gamma_1})\mathcal{D}_{\gamma_1,\gamma_2,\gamma_3}(h;u,q^{\gamma_1})du}{u^{\left[\frac{N+d(h_1)}{2}\right]-g}(1-u)(1-q^{2\gamma_1}u)(1-q^{2(\gamma_1-\gamma_2)}u)(1-q^{2(\gamma_1-\gamma_3)}u)(1-q^{2\gamma_1-\gamma_2-\gamma_3}u)}.
\end{align*}

Notice that $\mathcal{C}_{\gamma_1,\gamma_2,\gamma_3}(u,q^{\gamma_1})$ converges absolutely for $q^{-1+\varepsilon}<|u|<q^{1/2+2\min\{\Re \gamma_j\}-2\Re \gamma_1-\varepsilon}$. We move the $u$-contour to $|u|=q^{1/2+2\min\{\Re \gamma_j\}-2\Re \gamma_1-2\varepsilon}$ 
encountering five simple poles and the new integral is bounded trivially by 
\begin{align*}
&\ll_\varepsilon|h_1|^{-3/4-\min\{\Re \gamma_j\}}q^{-g/4-g\min\{\Re \gamma_j\}+\varepsilon g}\Big(\frac{q^{2g}}{|h_1|}\Big)^{-\Re\gamma_1}\\
&\ll_\varepsilon|h_1|^{-3/4-\min\{\Re \gamma_j\}}q^{-g/4-g\min\{\Re \gamma_j\}+\varepsilon g}\Big(\frac{q^{2g}}{|h_1|}\Big)^{-\min\{\Re \gamma_j\}}=|h_1|^{-3/4}q^{-g/4-3g\min\{\Re \gamma_j\}+\varepsilon g},
\end{align*} 
provided that $|h_1|<q^{2g}$. Hence
\begin{align}\label{6001}
&S_{\gamma_1,\gamma_2,\gamma_3}^{\textrm{e}}(h;N)\nonumber\\
&\qquad=\text{Res}(u=1)+\text{Res}(u=q^{-2\gamma_1})+\text{Res}(u=q^{2(\gamma_2-\gamma_1)})+\text{Res}(u=q^{2(\gamma_3-\gamma_1)})\nonumber\\
&\qquad\qquad+\text{Res}(u=q^{\gamma_2+\gamma_3-2\gamma_1})+O_\varepsilon\big(|h_1|^{-3/4}q^{-g/4-3g\min\{\Re \gamma_j\}+\varepsilon g}\big)\nonumber\\
&\qquad=\sum_{j=1}^{5}S_{\gamma_1,\gamma_2,\gamma_3}^{\textrm{e};j}(h;N)+O_\varepsilon\big(|h_1|^{-3/4}q^{-g/4-3g\min\{\Re \gamma_j\}+\varepsilon g}\big),
\end{align}
say, where
\begin{align*}
&S_{\gamma_1,\gamma_2,\gamma_3}^{\textrm{e};1}(h;N)=- \frac{q\mathcal{C}_{\gamma_1,\gamma_2,\gamma_3}(1,q^{\gamma_1})\mathcal{D}_{\gamma_1,\gamma_2,\gamma_3}(h;1,q^{\gamma_1})}{(q-1)|h_1|^{1/2+\gamma_1}}q^{-2\gamma_1\left[\frac{N-d(h_1)}{2}\right]}\zeta_q(1-2\gamma_1)\\
&\ \times\zeta_q(1-2\gamma_1+2\gamma_2)\zeta_q(1-2\gamma_1+2\gamma_3)\zeta_q(1-2\gamma_1+\gamma_2+\gamma_3)\zeta_q(1-\gamma_1+\gamma_2)\zeta_q(1-\gamma_1+\gamma_3),\\
&S_{\gamma_1,\gamma_2,\gamma_3}^{\textrm{e};2}(h;N)=- \frac{q\mathcal{C}_{\gamma_1,\gamma_2,\gamma_3}(q^{-2\gamma_1},q^{\gamma_1})\mathcal{D}_{\gamma_1,\gamma_2,\gamma_3}(h;q^{-2\gamma_1},q^{\gamma_1})}{(q-1)|h_1|^{1/2-\gamma_1}}q^{-2g\gamma_1-2\gamma_1}\\
&\qquad \times\zeta_q(1+2\gamma_1)\zeta_q(1+2\gamma_2)\zeta_q(1+2\gamma_3)\zeta_q(1+\gamma_2+\gamma_3)\zeta_q(1-\gamma_1+\gamma_2)\zeta_q(1-\gamma_1+\gamma_3),\\
&S_{\gamma_1,\gamma_2,\gamma_3}^{\textrm{e};3}(h;N)=- \frac{q\mathcal{C}_{\gamma_1,\gamma_2,\gamma_3}(q^{2(\gamma_2-\gamma_1)},q^{\gamma_1})\mathcal{D}_{\gamma_1,\gamma_2,\gamma_3}(h;q^{2(\gamma_2-\gamma_1)},q^{\gamma_1})}{(q-1)|h_1|^{1/2-\gamma_1}}\\
&\qquad\qquad\times q^{-2g(\gamma_1-\gamma_2)-2(\gamma_1-\gamma_2)-2\gamma_2\left[\frac{N+d(h_1)}{2}\right]}\zeta_q(1-2\gamma_2)\zeta_q(1+2\gamma_1-2\gamma_2)\\
&\qquad\qquad \times\zeta_q(1-2\gamma_2+2\gamma_3)\zeta_q(1-\gamma_2+\gamma_3)\zeta_q(1-\gamma_1+\gamma_2)\zeta_q(1-\gamma_1+\gamma_3),\\
&S_{\gamma_1,\gamma_2,\gamma_3}^{\textrm{e};4}(h;N)=S_{\gamma_1,\gamma_3,\gamma_2}^{\textrm{e};3}(h;N)
\end{align*}
and
\begin{align*}
&S_{\gamma_1,\gamma_2,\gamma_3}^{\textrm{e};5}(h;N)=- \frac{q\mathcal{C}_{\gamma_1,\gamma_2,\gamma_3}(q^{\gamma_2+\gamma_3-2\gamma_1})\mathcal{D}_{\gamma_1,\gamma_2,\gamma_3}(h;q^{\gamma_2+\gamma_3-2\gamma_1},q^{\gamma_1})}{(q-1)|h_1|^{1/2-\gamma_1}}\\
&\qquad\qquad\times q^{-g(2\gamma_1-\gamma_2-\gamma_3)-(2\gamma_1-\gamma_2-\gamma_3)-(\gamma_2+\gamma_3)\left[\frac{N+d(h_1)}{2}\right]}\zeta_q(1-\gamma_2-\gamma_3)\zeta_q(1+2\gamma_1-\gamma_2-\gamma_3)\\
&\qquad\qquad \times\zeta_q(1+\gamma_2-\gamma_3)\zeta_q(1-\gamma_2+\gamma_3)\zeta_q(1-\gamma_1+\gamma_2)\zeta_q(1-\gamma_1+\gamma_3).
\end{align*}
 
It is straightforward to verify that
\begin{equation}\label{verify1}
\mathcal{C}_{\gamma_1,\gamma_2,\gamma_3}(1,q^{\gamma_1})=\frac{\mathcal{A}_{\gamma_1,\gamma_2,\gamma_3}(q^{\gamma_1})}{\zeta_q(2)}\quad\text{and}\quad \mathcal{D}_{\gamma_1,\gamma_2,\gamma_3}(h;1,q^{\gamma_1})=|h_1|^{\alpha_1}\mathcal{B}_{\gamma_1,\gamma_2,\gamma_3}(h;q^{\gamma_1}).
\end{equation}
Hence
\[
M_{\gamma_1,\gamma_2,\gamma_3}^{1,1}(h;N)+S_{\gamma_1,\gamma_2,\gamma_3}^{\textrm{e};1}(h;N)=0,
\]
where $M_{\gamma_1,\gamma_2,\gamma_3}^{i,j}(h;N)$ is defined in \eqref{Mịformula}, and also
\begin{align*}
&M_{\gamma_1,\gamma_2,\gamma_3}^{2,2}(h;N)+S_{\gamma_2,\gamma_3,\gamma_1}^{\textrm{e};1}(h;N)=0,\\
&M_{\gamma_1,\gamma_2,\gamma_3}^{3,3}(h;N)+S_{\gamma_3,\gamma_1,\gamma_2}^{\textrm{e};1}(h;N)=0.
\end{align*}
So combining with \eqref{boundfornonsquare}, \eqref{MCformula}, \eqref{6000} and \eqref{6001} we obtain
\begin{align}\label{combine1}
&S_{\gamma_1,\gamma_2,\gamma_3}(h;N)=\frac{\widetilde{\mathcal{S}}_{\gamma_1,\gamma_2,\gamma_3}(h)}{\sqrt{|h_1|}}+\sum_{1\leq i< j\leq 3}M_{\gamma_1,\gamma_2,\gamma_3}^{i,j}(h;N)+\sum_{j=2}^{5}S_{\gamma_1,\gamma_2,\gamma_3}^{\textrm{e};j}(h;N)\nonumber\\
&\qquad\qquad+\sum_{j=2}^{5}S_{\gamma_2,\gamma_3,\gamma_1}^{\textrm{e};j}(h;N)+\sum_{j=2}^{5}S_{\gamma_3,\gamma_1,\gamma_2}^{\textrm{e};j}(h;N)+O_\varepsilon\big(|h|^{1/2} q^{-g/2-3g \min\{\Re\gamma_j\}+\varepsilon g}\big)\nonumber\\
&\qquad\qquad+O_\varepsilon\big(q^{-g-3g\min\{0,\Re \gamma_j\}+\varepsilon g}\big)+O_\varepsilon\big(|h_1|^{-3/4}q^{-g/4-3g\min\{\Re \gamma_j\}+\varepsilon g}\big).
\end{align}

From \eqref{switchingalpha} we see that the total error is
\begin{align*}
&\ll_\varepsilon q^{-2g\sum\mathfrak{a}_j\Re\alpha_j+\varepsilon g}\Big(|h|^{1/2}q^{-g/2-3g\min\{|\Re\alpha_j|\}}+q^{-g}+|h_1|^{-3/4}q^{-g/4-3g\min\{|\Re\alpha_j|\}}\\
&\qquad\qquad\qquad\qquad\quad+q^{-2g\sum|\Re\alpha_j|}\big(|h|^{1/2}q^{-g/2+3g\max\{|\Re\alpha_j|\}}+|h_1|^{-3/4}q^{-g/4+3g\max\{|\Re\alpha_j|\}}\big)\Big)\\
&\ll_\varepsilon |h|^{1/2}q^{-g/2+4g\max\{|\Re\alpha_j|\}-g\min\{|\Re\alpha_j|\}+\varepsilon g}+q^{-g+6g\max\{|\Re\alpha_j|\}+\varepsilon g}\\
&\qquad\qquad\qquad\qquad\quad+|h_1|^{-3/4}q^{-g/4+4g\max\{|\Re\alpha_j|\}-g\min\{|\Re\alpha_j|\}+\varepsilon g},
\end{align*}
and we obtain the expression for $\widetilde{E}_3$ in \eqref{formulaE2}.

For the main term, like in \eqref{verify1} we have
\begin{equation*}
\mathcal{C}_{\gamma_1,\gamma_2,\gamma_3}(q^{-2\gamma_1},q^{\gamma_1})=\frac{\mathcal{A}_{-\gamma_1,\gamma_2,\gamma_3}(1)}{\zeta_q(2)}\quad\text{and}\quad \mathcal{D}_{\gamma_1,\gamma_2,\gamma_3}(h;q^{-2\gamma_1},q^{\gamma_1})=|h_1|^{-\alpha_1}\mathcal{B}_{-\gamma_1,\gamma_2,\gamma_3}(h;1).
\end{equation*}
Together with the identity
\[
-q^{-2\gamma_1}\zeta_q(1+2\gamma_1)=\zeta_q(1-2\gamma_1)
\]
it follows that
\begin{equation}\label{combine2}
S_{\gamma_1,\gamma_2,\gamma_3}^{\textrm{e};2}(h;N)=\frac{q^{-2g\gamma_1}}{\sqrt{|h_1|}}\widetilde{\mathcal{S}}_{-\gamma_1,\gamma_2,\gamma_3}(h),
\end{equation}
and hence also
\begin{equation}\label{combine3}
q^{-2g(\gamma_1+\gamma_2+\gamma_3)}S_{-\gamma_1,-\gamma_2,-\gamma_3}^{\textrm{e};2}(h;N)=\frac{q^{-2g(\gamma_2+\gamma_3)}}{\sqrt{|h_1|}}\widetilde{\mathcal{S}}_{\gamma_1,-\gamma_2,-\gamma_3}(h).
\end{equation}

Using the same arguments, by comparing the coefficients of $q^{-g\alpha_j}$ for $1\leq j\leq 3$ we get
\begin{align}\label{combine4}
&S_{\alpha_1,\alpha_2,\alpha_3}^{\textrm{e};3}(h;3g)+q^{-2g(\alpha_1+\alpha_2+\alpha_3)}S_{-\alpha_3,-\alpha_1,-\alpha_2}^{\textrm{e};4}(h;3g-1)=0,\nonumber\\
&S_{\alpha_1,\alpha_2,\alpha_3}^{\textrm{e};4}(h;3g)+q^{-2g(\alpha_1+\alpha_2+\gamma_3)}S_{-\alpha_2,-\alpha_3,-\alpha_1}^{\textrm{e};3}(h;3g-1)=0,\nonumber\\
&S_{\alpha_1,\alpha_2,\alpha_3}^{\textrm{e};5}(h;3g)+q^{-2g(\alpha_1+\alpha_2+\gamma_3)}\mathcal{M}_{-\alpha_1,-\alpha_2,-\alpha_3}^{2,3}(h;3g-1)=0,\\
&\mathcal{M}_{\alpha_1,\alpha_2,\alpha_3}^{1,2}(h;3g)+q^{-2g(\alpha_1+\alpha_2+\gamma_3)}S_{-\alpha_3,-\alpha_1,-\alpha_2}^{\textrm{e};5}(h;3g-1)=0.\nonumber
\end{align}
In view of \eqref{combine1}--\eqref{combine4} we have
\begin{align*}
&S_{\alpha_1,\alpha_2,\alpha_3}(h;3g)+q^{-2g(\alpha_1+\alpha_2+\gamma_3)}S_{-\alpha_1,-\alpha_2,-\alpha_3}(h;3g-1)\\
&\qquad\qquad=\frac{1}{\sqrt{|h_1|}}\sum_{{\bf R}\subset{\bf A}=\{\alpha_1,\alpha_2,\alpha_3\}}q^{-2g\bf R}\widetilde{\mathcal{S}}_{\bf (A\backslash\bf R)\cup\bf R^-}(h)+O_\varepsilon\big(|h|^{1/2} q^{-g/2+\varepsilon g}\big),
\end{align*}
as required.

\subsection{Evaluate $S_{{\bf C}}(h;N;V \neq \square)$}\label{sectionnonsquare}

Recall from \eqref{Sknonsquare} that 
\begin{align*}
S_{{\bf C}}(h;N;V \neq \square) &=\big(S_{{\bf C};1}^{\textrm{o}}(h;N)-qS_{{\bf C};2}^{\textrm{o}}(h;N)\big)\nonumber\\
&\qquad\qquad+\big(S_{{\bf C};1}^{\textrm{e}}(h;N;V\ne\square)-qS_{{\bf C};2}^{\textrm{e}}(h;N;V\ne\square)\big) 
\end{align*} and $S_{{\bf C};1}^{\textrm{o}}(h;N)$ is given by equation \eqref{s1odd},
\[
S_{{\bf C};1}^{\textrm{o}}(h;N)=\frac{q^{3/2}}{(q-1)|h|}\sum_{\substack{d(f)\leq N\\d(fh)\ \textrm{odd}}}\frac{ \tau_{\bf C}(f)}{|f|^{3/2}}\sum_{\substack{C|(fh)^\infty\\d(C)\leq g}}\frac{1}{|C|^2}\sum_{V\in\mathcal{M}_{d(fh)-2g-2+2d(C)}}G(V,\chi_{fh}).
\] We will focus on bounding $S_{{\bf C};1}^{\textrm{o}}(h;N)$, since bounding the other ones follow similarly.

Using the fact that for $r_1<1$,
$$\sum_{\substack{C \in \mathcal{M}_j \\ C| (f h)^{\infty}}} \frac{1}{|C|^2} = \frac{1}{2 \pi i} \oint_{|u|=r_1} q^{-2j}   \prod_{P | f h} \big(1-u^{d(P)}\big)^{-1}\, \frac{du}{u^{j+1}},$$ and writing $V=V_1V_2^2$ with $V_1$ a square-free polynomial, we have
\begin{align*}
S_{{\bf C};1}^{\textrm{o}}(h;N) &= \frac{q^{3/2}}{(q-1)|h|} \frac{1}{2 \pi i} \oint_{|u|=q^{-\varepsilon}
} \sum_{\substack{n\leq N \\ n +d(h) \text{ odd}}}\sum_{j=0}^g  \sum_{\substack{r\leq n+d(h)-2g+2j-2 \\ r \text{ odd}}}q^{-2j} \\
& \ \times \sum_{V_1 \in \mathcal{H}_r} \sum_{V_2 \in \mathcal{M}_{(n+d(h)-r)/2-g+j-1}} \sum_{f \in \mathcal{M}_n} \frac{\tau_{\bf C}(f) G(V_1V_2^2, \chi_{f h})}{|f|^{3/2} }\prod_{P | f h} \big(1-u^{d(P)}\big)^{-1} \, \frac{du}{u^{j+1}}.
\end{align*}
Now
\begin{align*}
\sum_{f \in \mathcal{M}}   \frac{\tau_{\bf C}(f) G(V_1V_2^2, \chi_{f h})}{|f|^{3/2}}  \prod_{P | f h} \big(1-u^{d(P)}\big)^{-1}w^{d(f)}=  \mathcal{H}(V_1 ; u,w) \mathcal{K}(V,h;u,w), 
\end{align*} where

$$ \mathcal{H}(V_1;u,w) =  \prod_{P\nmid V_1} \bigg( 1+\sum_{\gamma\in\bf C} \frac{\chi_{V_1}(P) w^{d(P)}}{|P|^{1+\gamma}}\big(1-u^{d(P)}\big)^{-1} \bigg) $$ 
and
\begin{align}
\mathcal{K}(V,h;u,w) &=  \prod_{P | h} \bigg( \sum_{j=0}^{\infty} \frac{\tau_{\bf C}(P^j) G(V,\chi_{P^{j+\text{ord}_P(h)}}) w^{j d(P)}}{|P|^{3j/2}}   \bigg) \big(1-u^{d(P)}\big)^{-1}\nonumber \\
&\qquad\qquad\times \prod_{\substack{P \nmid h\\P |V  }} \bigg( 1+ \sum_{j=1}^{\infty} \frac{\tau_{\bf C}(P^j) G(V, \chi_{P^j}) w^{j d(P)}}{|P|^{3j/2}} \big(1-u^{d(P)}\big)^{-1}  \bigg)  \nonumber  \\
& \qquad\qquad\times \prod_{\substack{P \nmid V_1\\P | h V_2 }}\bigg( 1+\sum_{\gamma\in\bf C} \frac{\chi_{V_1}(P) w^{d(P)}}{|P|^{1+\gamma}} \big(1-u^{d(P)}\big)^{-1}\bigg)^{-1}. \label{k_euler}
\end{align}
Note that $\mathcal{H}(V_1;u,w)$ is convergent for $|u|<1$ and $|w|<q^{1/2+\min\{\Re \gamma_j\}-\varepsilon}$. Using the Perron formula for the sum over $f$ we obtain
\begin{align}
S_{{\bf C};1}^{\textrm{o}}(h;N)&= \frac{q^{3/2}}{(q-1)|h|} \frac{1}{(2 \pi i)^2} \oint_{|u|=q^{-\varepsilon}} \oint_{|w| = r_2} \sum_{\substack{n\leq N \\ n +d(h) \text{ odd}}}  \sum_{j=0}^g  \sum_{\substack{r\leq n+d(h)-2g+2j-2 \\ r \text{ odd}}} q^{-2j} \label{error_explicit}\\
& \qquad\qquad\times  \sum_{V_2 \in \mathcal{M}_{(n+d(h)-r)/2-g+j-1}}\sum_{V_1 \in \mathcal{H}_r} \mathcal{H}(V_1;u,w)  \mathcal{K}(V,h;u,w) \,  \frac{du}{u^{j+1}} \, \frac{dw}{w^{n+1}}, \nonumber 
\end{align}
with $r_2=q^{1/2+\min\{\Re \gamma_j\}-\varepsilon}$.

For the application to Theorem \ref{theorem1}, we will keep the terms of the form $S_{{\bf C};1}^{\textrm{o}}(h;N)$ explicit as in the formula above. 
For the purpose of Theorem \ref{theorem2}, we will proceed to bound the term $S_{{\bf C};1}^{\textrm{o}}(h;N)$.

 Let $i_0$ be minimal such that $|u^{i_0}w| < q^{\max\{\Re \gamma_j\}}$. Then we write
\begin{equation*}
\mathcal{H}(V_1;u,w) =  \prod_{\gamma\in\bf C}\mathcal{L}\Big(\frac {w}{q^{1+\gamma}},\chi_{V_1}\Big) \mathcal{L}\Big(\frac{uw}{q^{1+\gamma}},\chi_{V_1}\Big)  \ldots \mathcal{L}\Big(\frac{u^{i_0-1}w}{q^{1+\gamma}}, \chi_{V_1}\Big) \mathcal{T}(V_1;u,w), 
\end{equation*} where $\mathcal{T}(V_1;u,w)$ is absolutely convergent in the selected region.
Using Theorem 3.3 in \cite{AT} and the remarks
in the proof of Lemma 7.1 in \cite{F1}, it follows that 
\begin{equation}
\mathcal{L}\Big(\frac{u^{i}w}{q^{1+\gamma}}, \chi_{V_1}\Big)\ll \exp\Big(\frac{r}{\log_q(r/2)}+2\sqrt{2qr}\Big) \label{lindelof}
\end{equation}
for any $1\leq i\leq i_0-1$. We also have $$ \mathcal{K}(V,h;u,w)  \ll_\varepsilon  |h|^{1/2+\varepsilon} \big| (h, V_2^2)\big|^{1/2}|V|^{\varepsilon}.$$ Trivially bounding the rest of the expression, we obtain that
$$ S_{{\bf C};1}^{\textrm{o}}(h;N) \ll_\varepsilon |h|^{1/2} q^{N/2-N\min\{\Re \gamma_j\} -2g+ \varepsilon g}.$$

\section{Proof of Theorem \ref{theorem1}}
\label{mainthm}
We write 
\begin{align}\label{start}
&\frac{1}{|\mathcal{H}_{2g+1}|}  \sum_{D\in \mathcal{H}_{2g+1}}  \frac{\prod_{j=1}^kL(1/2+\alpha_j,\chi_D)}{\prod_{j=1}^kL(1/2+\beta_j,\chi_D)}  \\
&\qquad = \sum_{h_1, \ldots, h_k \in \mathcal{M}} \frac{\prod_{j=1}^k \mu(h_j)}{ \prod_{j=1}^k |h_j|^{1/2+\beta_j}}\frac{1}{|\mathcal{H}_{2g+1}|}  \sum_{D \in \mathcal{H}_{2g+1}} \bigg(\prod_{j=1}^kL(\tfrac12+\alpha_j,\chi_D) \bigg)  \chi_D\bigg(\prod_{j=1}^k h_j\bigg) . \nonumber
\end{align}
For some parameter $X$ to be chosen later, let 
\begin{equation}
S_{k,\leq X} = \sum_{d(h_1), \ldots, d(h_k) \leq X} \frac{\prod_{j=1}^k \mu(h_j)}{ \prod_{j=1}^k |h_j|^{1/2+\beta_j}}  \frac{1}{|\mathcal{H}_{2g+1}|} \sum_{D \in \mathcal{H}_{2g+1}} \bigg(\prod_{j=1}^kL(\tfrac12+\alpha_j,\chi_D)\bigg)    \chi_D\bigg(\prod_{j=1}^k h_j\bigg),
\label{sfirst}
\end{equation}
and $S_{k,>X}$ denote the term in \eqref{start} where at least one polynomial $h_j$ has degree bigger than $X$. Let $S_{k,>X,1}$ denote the term in \eqref{start} where $d(h_1)>X$. 

We will now bound the contribution from $S_{k,>X,1}$. Using Perron's formula for the sum over $h_1$, we rewrite 
\begin{align}
S_{k,>X,1}  = \frac{1}{|\mathcal{H}_{2g+1}|}  & \frac{1}{ 2 \pi i} \oint  \sum_{D \in \mathcal{H}_{2g+1}} \frac{ \prod_{j=1}^kL(1/2+\alpha_j,\chi_D)}  {     \mathcal{L} \big( \frac{z_1}{q^{1/2+\beta_1}} ,\chi_D\big)  \prod_{j=2}^k L(1/2+\beta_j,\chi_D)} \frac{dz_1}{z_1^{X+1}(z_1-1)} ,
\label{multiple_integral}
\end{align}
where we are integrating over a circle $|z_1|>1$. We pick $|z_1| = q^{(1-\varepsilon)\Re \beta_1 }$. Using H\"{o}lder's inequality we have 
\begin{align*}
\sum_{D \in \mathcal{H}_{2g+1}} & \bigg|  \frac{ \prod_{j=1}^kL(1/2+\alpha_j,\chi_D) } {  \mathcal{L} \big( \frac{z_1}{q^{1/2+\beta_1}},\chi_D\big) \prod_{j=2}^k L(1/2+\beta_j,\chi_D)} \bigg| \leq \bigg( \sum_{D \in \mathcal{H}_{2g+1}}  \prod_{j=1}^k  \big| L(\tfrac12+\alpha_j,\chi_D)  \big| ^{\frac{1+\varepsilon}{\varepsilon}}   \bigg)^{\frac{\varepsilon}{1+\varepsilon}} \\
& \times \bigg(  \sum_{D \in \mathcal{H}_{2g+1}}  \frac{1}{  \big|\mathcal{L} \big( \frac{z_1}{q^{1/2+\beta_1}},\chi_D\big ) \prod_{j=2}^k L(1/2+\beta_j,\chi_D) \big|^{1+\varepsilon} } \bigg)^{\frac{1}{1+\varepsilon}}.
\end{align*}

For the first term above, we use Corollary $2.8$ in  \cite{F4} and get that 
\begin{align}
 \bigg( \sum_{D \in \mathcal{H}_{2g+1}}  \prod_{j=1}^k  \big| L(\tfrac12+\alpha_j,\chi_D)  \big| ^{\frac{1+\varepsilon}{\varepsilon}}   \bigg)^{\frac{\varepsilon}{1+\varepsilon}} \ll q^{\frac{2g\varepsilon}{1+\varepsilon}} g^{ \frac{k}{2} \big( \frac{k(1+\varepsilon)}{\varepsilon}+1 \big)}.
 \label{sum1}
\end{align}
Now using Theorem \ref{theorem3} we have 
\begin{align}
\bigg(  \sum_{D \in \mathcal{H}_{2g+1}}  \frac{1}{  \big|\mathcal{L} \big( \frac{z_1}{q^{1/2+\beta_1}},\chi_D\big ) \prod_{j=2}^k L(1/2+\beta_j,\chi_D) \big|^{1+\varepsilon} } \bigg)^{\frac{1}{1+\varepsilon}} \ll q^{\frac{2g}{1+\varepsilon}} \Big( \frac{ \log g }{\beta'} \Big)^{\frac{k}{2} (k(1+\varepsilon)+1 )}
\label{sum2}
\end{align}
for $\Re \beta_j \gg g^{-\frac{1}{2k}+\varepsilon}$, where $\beta'= \min\{  \varepsilon \Re \beta_1, \Re \beta_2, \ldots, \Re \beta_k\}.$ Combining equations \eqref{multiple_integral}, \eqref{sum1} and \eqref{sum2}, we obtain that
\begin{equation*}
S_{k,>X,1} \ll \frac{q^{ -(1-\varepsilon)X \Re  \beta_1 }}{\Re \beta_1}g^{ \frac{k}{2} \big( \frac{k(1+\varepsilon)}{\varepsilon}+1 \big)} \Big( \frac{ \log g }{\beta'} \Big)^{\frac{k}{2} (k(1+\varepsilon)+1 )} .
\label{bigx}
\end{equation*}
Bounding the rest of the terms in $S_{k,>X}$ is similar to bounding $S_{k,>X,1}$ and we obtain the bound
\begin{equation}
S_{k,>X} \ll_\varepsilon q^{- (1-\varepsilon)X \beta} ,\label{sk}
\end{equation}
where $\beta= \min \{ \Re \beta_1, \ldots, \Re \beta_k\}$.

Next we use Theorem \ref{theorem2} to evaluate the term \eqref{sfirst}. Once getting the main terms from Theorem \ref{theorem2}, we use the argument above to re-extend the sums over $h_j$ to all $h_j\in\mathcal{M}$ for $1\leq j\leq k$ at the cost of a negligible error term. A standard exercise with the Euler product then gives the main terms of Conjecture \ref{ratios}.

Now we will focus on bounding the contributions coming from the error terms $\widetilde{E}_k$. 

For the cases $k=2$ and $k=3$, we will simply use the bounds from Theorem \ref{theorem2} to bound the error terms. For the case $k=1$, we will keep the error terms in the proof of Theorem \ref{theorem2} explicit, and will exploit the cancellation provided by the M\"{o}bius function. 

We first bound the error terms in the cases $k=2,3$, which is more straightforward. By interchanging the sums over $D$ and $h_j$ and trivially bounding the sums over $h_j$, the error terms $\widetilde{E}_k$ in Theorem \ref{theorem2} will overall contribute  error terms of size
\begin{equation}
\label{error23}
\begin{cases}
\ll_\varepsilon q^{2X-(1-2\alpha) g+\varepsilon g}+q^{X-(1-4\alpha) g+\varepsilon g},& \mbox{ if } k=2 \\
\ll_\varepsilon q^{3X-(1/2-4\alpha) g+\varepsilon g}+q^{3X/2-(1-6\alpha)g+\varepsilon g}+ q^{-(1/4-4\alpha) g+\varepsilon g},  & \mbox{ if } k=3.
\end{cases}
\end{equation}
For $k=2$, if $1\leq 2 \alpha (\beta+3)$, then we pick $X = \frac{g (1-4 \alpha-\varepsilon)}{1+\beta-\beta \varepsilon}$ and we obtain Theorem \ref{theorem1} with an error term of size $q^{- (1-\varepsilon)g\beta  \frac{1-4 \alpha-\varepsilon}{1+\beta-\beta \varepsilon}}$. If $1>2 \alpha(\beta+3)$, we choose $X = \frac{g (1-2\alpha-\varepsilon)}{2+\beta-\beta \varepsilon}$, and we obtain Theorem \ref{theorem2} with an error term of size $q^{- (1-\varepsilon)g\beta \frac{1-2 \alpha-\varepsilon}{2+\beta-\beta \varepsilon}}$. Combining the two bounds gives the error term $E_2$ in Theorem \ref{theorem1}. 

For $k=3$, if $3(1-16 \alpha) \geq \beta$, we pick $X = \frac{g (1/2-4\alpha-\varepsilon)}{3+\beta-\beta \varepsilon}$ and we obtain Theorem \ref{theorem1} with an error term of size $q^{-(1-\varepsilon)g\beta \frac{1/2-4\alpha-\varepsilon}{3+\beta-\beta \varepsilon}}$. If $3(1-16 \alpha) < \beta$, we pick $X = \frac{g (1/4-4\alpha-\varepsilon)}{\beta-\beta \varepsilon}$ and we obtain Theorem \ref{theorem1} with an error term of size $q^{-g (1/4-4\alpha - \varepsilon)}.$ Combining the two bounds gives the error term $E_3$ in Theorem \ref{theorem1}.

\kommentar{For $k=2$, if $\alpha\leq 1/6$ then we pick $X= \frac{g(1-2\alpha)}{2+\beta-\beta \varepsilon}$ and we obtain Theorem \ref{theorem1} with an error term of size $q^{-g \beta ( \frac{1-2\alpha-\varepsilon}{2}-\varepsilon )(1-\varepsilon)}$. If $1/6 < \alpha <1/4$, then we pick $X= \frac{g (1-4\alpha-\varepsilon)}{1+\beta-\beta \varepsilon}$, and then we obtain Theorem \ref{theorem1} with an error term of size $q^{-g \beta (1-4 \alpha-\varepsilon)(1-\varepsilon)}$. Combining the two bounds gives the error term in Theorem \ref{theorem1}.

For $k=3$, we pick $X= \frac{g(1/2-4\alpha-\varepsilon)}{3+\beta-\beta \varepsilon}$, and then we obtain Theorem \ref{theorem1} with the given error term. }

\subsection{The case $k=1$}
Here, we follow along the proof of Theorem \ref{theorem2}, and in many places, instead of bounding the various error terms as in Theorem \ref{theorem2}, we keep them explicit and exploit the extra cancellation provided by the M\"{o}bius function, in order to obtain a better error term in Theorem \ref{theorem1}.

By trivially bounding the sum over $h$ in the expression for $S_{1, \leq X}$, the error term in equation \eqref{mc} will contribute a total error term of size $q^{X/2-3g/2- g \gamma +\varepsilon g}$, where we denote $\gamma=\Re\gamma_1$. 

Now in equation \eqref{2001}, note that
\[
\mathcal{A}_{\gamma_1}(u)=\prod_{P}\bigg(1-\frac{u^{2d(P)}}{|P|^{1+2\gamma_1}(|P|+1)}\bigg)
\]
has analytic continuation  for $|u|< q^{1+\gamma}$. We move the contour of integration to $|u|=q^{1+\gamma-\varepsilon}$, crossing three simple poles at $u=1$ and $u=\pm q^{\gamma_1}$. We keep the new integral as it is rather than bounding it trivially on the new contour. Let $E_{\leq X, \gamma_1}(N;V=0)$ be the error term obtained after introducing the sum over $h$. We write it as
\begin{align*}
E_{\leq X, \gamma_1}(N;V=0) = \sum_{h \in \mathcal{M}_{\leq X}} \frac{\mu(h)}{|h|^{1+\beta_1}} \frac{1}{2 \pi i} \oint \frac{ \mathcal{A}_{\gamma_1} (u) \mathcal{B}_{\gamma_1}(h;u) du}{u^{N-d(h)+1} (1-u)(1-q^{-2 \gamma_1} u^2)}, 
\end{align*}
where $|u| = q^{1+\gamma-\varepsilon}$. 
By a standard argument, it follows that
\begin{align*}
E_{\leq X,\gamma_1}(N;V=0) = \frac{1}{ (2 \pi i)^2} \oint \oint \frac{ \mathcal{A}_{\gamma_1}(u) \mathcal{F}_{\gamma_1}(u,y)  du \, dy}{u^{N+1} y^{X+1}(1-u)(1-y)(1-q^{-2 \gamma_1 } u^2)} ,
\end{align*}
where $\mathcal{F}_{\gamma_1}(u,y)$ is the generating series of the sum over $h$ and has the following Euler product:
$$\mathcal{F}_{\gamma_1} (u,y) = \prod_P \bigg( 1- \frac{(uy)^{d(P)}}{|P|^{1+\beta_1+\gamma_1}} \bigg( 1+ \frac{1}{|P|} - \frac{u^{2d(P)}}{|P|^{2+2 \gamma_1}} \bigg)^{-1} \bigg).$$  In the integral above, we are integrating along a circle $|y|<1$. Note that $\mathcal{F}_{\gamma_1}(u,y)$ has an analytic continuation for $|uy|<q^{1+\beta+\gamma}$. 
In the double integral above, we shift the contour over $y$ to $|y|=q^{\beta}$ and encounter a pole at $y=1$. The residue at $y=1$ is $$\ll_\varepsilon q^{-g - g \gamma+\varepsilon g}.$$ For the new integral 
we bound it trivially by $q^{-X\beta-g-g \gamma+\varepsilon g}$. Hence 
\begin{equation}
E_{\leq X,\gamma_1}(N;V= 0) \ll_\varepsilon q^{-g-g \gamma+\varepsilon g}.
\label{error_0}
\end{equation}  

For the error term in \eqref{second_error}, similarly as before, by trivially bounding the sum over $h$, we will get that overall that error term will be $\ll_\varepsilon q^{X/2-3g/2+g \alpha +\varepsilon g}$.
Now in equation \eqref{second_explicit}, similarly as before, after shifting the contour over $w$ to $|w|= q^{3/4+\min \{ 0, \gamma\}-2 \varepsilon}$, we keep the integral as it is and introduce the sum over $h$. Let $E_{\leq X, \gamma_1}(N;V=\square)$ denote this error term, which we rewrite as follows.
\begin{align*}
&E_{\leq X, \gamma_1} ( N; V=\square)  = \frac{1}{q-1} \sum_{h \in \mathcal{M}_{\leq X}} \frac{ \mu(h)}{|h|^{1+\beta_1}}  \nonumber  \\
&\qquad\qquad \times \frac{1}{(2 \pi i)^2}\oint  \oint  \frac{\mathcal{C}_{\gamma_1}(u,w)\mathcal{D}_{\gamma_1}(h;u,w) (1-qu)(1-q^{-(1+2\gamma_1)}w^2)  }{(1-u) (1-q^{-\gamma_1}w)(1-uw^2)(1-q^{-2\gamma_1}uw^2)(u-q^{-(2+\gamma_1)}w)} \\
&\qquad\qquad\qquad\qquad\qquad\qquad\times \frac{dwdu}{u^{\left[\frac{N+d(h)}{2}\right]-g}w^{2\left[\frac{N-d(h)}{2}\right]+d(h)+1}},
\end{align*}
where we are integrating over circles $|u|=q^{-3/2+\varepsilon}$ and $|w|=q^{3/4+\min\{0,\gamma\}-2 \varepsilon}$. We look at the generating series of the sum over $h$, and we have
\begin{align*}
\mathcal{M}(u,w,y) &:= \sum_{h \in \mathcal{M} } \frac{\mu(h) y^{d(h)} }{|h|^{1+\beta_1}u^{d(h)/2}} \mathcal{D}_{\gamma_1}(h;u,w) \\
&= \prod_P \bigg(1 - \frac{y^{d(P)}}{|P|^{1+\beta_1} u^{d(P)/2}}  \bigg(1+\frac{w^{d(P)}(1-u^{d(P)})}{|P|^{1+\gamma_1}}  -\frac{1}{|P|^2u^{d(P)}}\\ 
&\qquad\qquad-\frac{(uw^2)^{d(P)}}{|P|^{2+2\gamma_1}}+\frac{w^{2d(P)}}{|P|^{3+2\gamma_1}}\bigg)^{-1}\bigg(1-u^{d(P)}+\frac{(uw)^{d(P)}}{|P|^{\gamma_1}}-\frac{(uw)^{d(P)}}{|P|^{1+\gamma_1}}\bigg) \bigg) .
\end{align*}
Note that the generating series above has an analytic continuation for $|y|<q^{1+\beta} |u|^{1/2}$, $|u^{1/2}y|<q^{\beta}, |u^{1/2}wy| <q^{\beta+\gamma}, |wy|<q^{1+\beta+\gamma} |u|^{1/2}, |y|<q^{2+\beta} |u|^{3/2},|u^{1/2}w^2y|<q^{2+\beta+2 \gamma}$. We use Perron's formula for the sum over $h$, obtaining a triple integral for $E_{\leq X, \gamma_1}(N;V = \square)$, with the integral over $y$ being over a circle of radius $q^{-1/2+\beta}$. We then get that
\begin{equation}
E_{\leq X, \gamma_1}(N;V=\square) \ll_\varepsilon q^{X(1/2-\beta)-3g/2 - g \min \{0 ,\gamma\}+\varepsilon g}.
\label{error_sq}
\end{equation}
Next, the error term in bounding $S_{\gamma_1}^{\textrm{e};2}(h;N)$ after shifting the contour to $|u|=q^{-1-\varepsilon}$ will be $\ll_\varepsilon q^{X/2-3g/2-g\gamma+\varepsilon g}$ after trivially bounding the sum over $h$. Now let $E_{\leq X, \gamma_1}^{\textrm{e};2}(N)$ denote the term obtained after introducing the sum over $h$ in equation \eqref{s2}, namely
\begin{align*}
 E_{\leq X, \gamma_1}^{\textrm{e};2}(N)&   = \sum_{h \in \mathcal{M}_{\leq X}} \frac{\mu(h)}{|h|^{1/2+\beta_1}}  \frac{q^{-2(2+\gamma_1)\left[\frac{N-d(h_1)}{2}\right]}}{(q-1)|h|^{5/2+\gamma_1}} \\
&\qquad\qquad \times   \frac{1}{2 \pi i}\oint_{|u|=r_1} \frac{\mathcal{C}_{\gamma_1}(u,q^{2+\gamma_1}u)\mathcal{D}_{\gamma_1}(h;u,q^{2+\gamma_1}u) (1-qu)(1-q^{3}u^2)  }{(1-u) (1-q^{2}u)(1-q^{4+2\gamma_1}u^3)(1-q^{4}u^3)} \\
&\qquad\qquad\qquad\qquad\qquad\qquad\times \frac{du}{u^{\left[\frac{N+d(h_1)}{2}\right]+2\left[\frac{N-d(h_1)}{2}\right]-g+d(h_1)+1}},
\end{align*} 
where $|u|=q^{-3/2+\varepsilon}$.
We use Perron's formula for the sum over $h$, obtaining a double integral over $|y|=q^{-1/2}$ and $|u|=q^{-3/2+\varepsilon}$. As in the proof of Theorem \ref{theorem2}, we shift the contour of integration to $|u|=q^{-1-\varepsilon}$, encountering poles when $u^3=q^{-4}$ and $u^3 = q^{-4-2 \gamma_1}$. When $u^3=q^{-4}$, note that $\mathcal{M}(q^{-4/3}, q^{2/3+\gamma_1},y)$ converges absolutely for $|y|<q^{\beta-1/3}$, so shifting the contour to $|y|=q^{(1-\varepsilon)\beta-1/3}$, we get that the contribution to the double integral of the pole at $u^3=q^{-4}$ will be $\ll_\varepsilon q^{-X  \beta+X/3-4g/3 - g\gamma+\varepsilon g}$. When $u^3=q^{-4-2 \gamma_1}$, note that $\mathcal{M}(q^{-(4+2\gamma_1)/3},q^{(2+\gamma_1)/3},y)$ converges absolutely for $|y|<q^{\beta+\gamma/3-1/3}$, so we shift the contour to $|y|=q^{(1-\varepsilon)\beta+\gamma/3-1/3}$, and we get that the contribution to the double integral of the pole at $u^3=q^{-4-2\gamma_1}$ will be bounded by $q^{-X  \beta -X \gamma/3+X/3-4g/3-2g \gamma/3 +\varepsilon g}$. Combining these bounds, it follows that
\begin{equation}
 E_{\leq X, \gamma_1}^{\textrm{e};2}(N) \ll_\varepsilon q^{-X \beta+X/3-4g/3 -g \gamma+\varepsilon g}.
 \label{e2}
\end{equation}

Now we bound the error term coming from terms like $S_{{\bf C};1}^{\textrm{o}}(h;N)$ in Theorem \ref{theorem2} (see equation \eqref{error_explicit}. Let $E_{\leq X}(N;V \neq \square)$ denote the term corresponding to $V \neq \square$ in the proof of Theorem \ref{theorem2} after introducing the sum over $h$. Let $E_{\leq X,1}(N;V \neq \square)$ denote the term corresponding to $S_{{\bf C};1}^{\textrm{o}}(h;N)$. We will only bound $E_{\leq X,1}(N; V \neq \square)$, as bounding $E_{\leq X}(N; V \neq \square)$ is similar. We rewrite $E_{\leq X,1}(N; V \neq \square)$ as
\begin{align*}
& E_{\leq X,1}(N; V \neq \square) = \frac{q^{3/2}}{q-1}  \sum_{h \in \mathcal{M}_{\leq X}} \frac{ \mu(h)}{|h|^{3/2+\beta_1}}\frac{1}{(2 \pi i)^2} \oint_{|u|=q^{-\varepsilon}} \oint_{|w| = r_2} \sum_{\substack{n\leq N \\ n +d(h) \text{ odd}}}  \sum_{j=0}^g  \nonumber \\
& \quad\times  \sum_{\substack{r\leq n+d(h)-2g+2j-2 \\ r \text{ odd}}} q^{-2j}  \sum_{V_2 \in \mathcal{M}_{(n+d(h)-r)/2-g+j-1}}\sum_{V_1 \in \mathcal{H}_r} \mathcal{H}(V_1;u,w)  \mathcal{K}(V,h;u,w) \,  \frac{dwdu}{u^{j+1}w^{n+1}}.
\end{align*}

Now we look at the generating series of the sum over $h$. In the definition of $\mathcal{K}(V,h;u,w)$, let $A_P(V;u,w)$ denote the first Euler factor over primes $P |h$, let $B_P(V;u,w)$ denote the second Euler factor over primes $P|V, P \nmid h$ and let $C_P(V_1;u,w)^{-1}$ denote the Euler factor over primes $P \nmid V_1, P|hV_2$. Then we have
\begin{align*}
 \mathcal{H}(V_1&;u,w)   \sum_{h \in \mathcal{M}}  \frac{ \mu(h)}{|h|^{3/2+\beta_1}} \mathcal{K}(V,h;u,w) y^{d(h)}\\
 &= \mathcal{H}(V_1;u,w)  \prod_{P|V} B_P(V;u,w) \prod_{\substack{P \nmid V_1 \\ P|V_2}} C_P(V_1;u,w)^{-1}\\
&\qquad\qquad \times \prod_{P \nmid V} \bigg( 1-  \frac{\chi_{V_1}(P) y^{d(P)}}{|P|^{1+\beta_1}} C_P(V_1;u,w)^{-1} (1-u^{d(P)})^{-1} \bigg) \\
&\qquad\qquad \times   \prod_{P|V} \bigg(1- \frac{y^{d(P)}}{|P|^{3/2+\beta_1}} A_P(V;u,w) B_P(V;u,w)^{-1} \bigg) \\
&=  \prod_{P \nmid V_1} \bigg( 1+\frac{\chi_{V_1}(P) w^{d(P)}}{|P|^{1+\gamma_1}} (1-u^{d(P)})^{-1}-  \frac{\chi_{V_1}(P) y^{d(P)}}{|P|^{1+\beta_1}} (1-u^{d(P)})^{-1}  \bigg)\\
&\qquad\qquad \times  \prod_{P|V} B_P(V;u,w)\bigg(1-  \frac{y^{d(P)}}{|P|^{3/2+\beta_1}} A_P(V;u,w) B_P(V;u,w)^{-1} \bigg)\\
&\qquad\qquad \times  \prod_{\substack{P \nmid V_1 \\ P|V_2}} C_P(V_1;u,w)^{-1} \bigg( 1-  \frac{\chi_{V_1}(P) y^{d(P)}}{|P|^{1+\beta_1}} C_P(V_1;u,w)^{-1} (1-u^{d(P)})^{-1} \bigg)^{-1}.
\end{align*}
We then rewrite
\begin{align*}
& \mathcal{H}(V_1;u,w) \sum_{h \in \mathcal{M}}  \frac{ \mu(h)}{|h|^{3/2+\beta_1}} \mathcal{K}(V,h;u,w) y^{d(h)} = \mathcal{J}(V_1;u,w) \prod_{P|V_1} E_P(V;u,w) \prod_{\substack{P \nmid V_1 \\ P|V_2}} F_P(V;u,w),
\end{align*}
where
$$ \mathcal{J}(V_1;u,w) = \prod_{P \nmid V_1} \bigg( 1+ \frac{\chi_{V_1}(P) w^{d(P)}}{|P|^{1+\gamma_1}} (1-u^{d(P)})^{-1}- \frac{\chi_{V_1}(P) y^{d(P)}}{|P|^{1+\beta_1}} (1-u^{d(P)})^{-1}  \bigg),$$
$$ E_P(V;u,w) = B_P(V;u,w)\bigg(1-  \frac{y^{d(P)}}{|P|^{3/2+\beta_1}} A_P(V;u,w) B_P(V;u,w)^{-1} \bigg)$$
and
\begin{align*}
 F_P(V;u,w) &= B_P(V;u,w)\bigg(1-  \frac{y^{d(P)}}{|P|^{3/2+\beta_1}} A_P(V;u,w) B_P(V;u,w)^{-1} \bigg) \\
 &\qquad\quad \times C_P(V_1;u,w)^{-1} \bigg( 1-\frac{\chi_{V_1}(P) y^{d(P)}}{|P|^{1+\beta_1}} C_P(V_1;u,w)^{-1} (1-u^{d(P)})^{-1} \bigg)^{-1}.
\end{align*}
Using Perron's formula for the sum over $h$, we then get that
\begin{align*}
& E_{\leq X,1}(N;V \neq \square) \\
&\ = \frac{q^{3/2}}{(q-1)} \frac{1}{(2 \pi i)^3} \oint_{|u|=q^{-\varepsilon}} \oint_{|w| = r_2} \oint_{|y|=r_3} \sum_{m \leq X} \sum_{\substack{n\leq N \\ n +m \text{ odd}}}  \sum_{j=0}^g  \sum_{\substack{r\leq n+m-2g+2j-2 \\ r \text{ odd}}} q^{-2j} \nonumber \\
&  \quad\times  \sum_{V_2 \in \mathcal{M}_{(n+m-r)/2-g+j-1}}\sum_{V_1 \in \mathcal{H}_r} \mathcal{J}(V_1;u,w) \prod_{P|V_1} E_P(V;u,w) \prod_{\substack{P \nmid V_1 \\ P|V_2}} F_P(V;u,w)  \frac{dydwdu}{u^{j+1}w^{n+1}y^{m+1}},
\end{align*} 
where recall that $|w| = r_2=q^{1/2+\gamma- \varepsilon}$, and we pick $|y| = q^{1/2+\beta-\varepsilon}$. Similarly as in Section $3$, let $i$ be minimal such that $|u^i w|<q^{\gamma}$ and such that $|u^i y|<q^{\beta}$. We then write
$$ \mathcal{J}(V_1;u,w) = \frac{  \mathcal{L} \big(  \frac{w}{q^{1+\gamma_1}},\chi_{V_1} \big) \mathcal{L} \big(  \frac{uw}{q^{1+\gamma_1}},\chi_{V_1} \big) \cdot \ldots \cdot \mathcal{L} \big( \frac{u^{i-1} w}{q^{1+\gamma_1}}, \chi_{V_1} \big)}{   \mathcal{L} \big( \frac{y}{q^{1+\beta_1}}, \chi_{V_1} \big) \mathcal{L} \big( \frac{uy}{q^{1+\beta_1}} ,  \chi_{V_1}\big) \cdot \ldots \cdot \mathcal{L} \big(  \frac{u^{i-1} y}{q^{1+\beta_1}},\chi_{V_1} \big)} \mathcal{U}(V_1;u,w),$$ where $\mathcal{U}(V_1;u,w)$ is analytic in the selected region. Now we trivially bound the factors $E_P(V;u,w)$ and $F_P(V;u,w)$, and use the Cauchy-Schwarz inequality for the sum over $V_1$. For the $L$--functions in the numerator, we use the bound \eqref{lindelof}, while for the $L$--functions in the denominator we use Theorem \ref{theorem3}. We then get that
$$ \sum_{V_1 \in \mathcal{H}_r} \mathcal{J}(V_1;u,w) \prod_{P|V_1} E_P(V;u,w) \prod_{\substack{P \nmid V_1 \\ P|V_2}} F_P(V;u,w) \ll_\varepsilon q^{r+\varepsilon r}.$$ Trivially bounding all the other sums, we get that
$$E_{\leq X, 1} (N ;V \neq \square) \ll_\varepsilon q^{- X\beta+X/2 - 3g/2  - g \gamma+\varepsilon g}$$
and hence
\begin{equation} \label{enonsquare}
 E_{\leq X}(N;V \neq \square) \ll_\varepsilon q^{- X\beta+X/2 - 3g/2  - g \gamma+\varepsilon g.}
\end{equation} 

If $\Re \alpha_1<0$, combining equations \eqref{switchingalpha}, \eqref{error_0}, \eqref{error_sq}, \eqref{e2}, \eqref{enonsquare}, it follows that the error term in evaluating $S_{1,\leq X}$ will be
\begin{align}
\ll_\varepsilon q^{X/2-X\beta -3g/2+2g \alpha + \varepsilon g}+ q^{-X \beta+X/3-4g/3+g \alpha+\varepsilon g} + q^{-g+ g  \alpha+\varepsilon g}.
\nonumber
\end{align}
Now using the bound above and \eqref{sk}, we choose $X = \frac{ g(3-4 \alpha -2 \varepsilon)}{1 - 2 \beta \varepsilon}$ and the second bound for $E_1$ in \eqref{first_part} follows.

When $\Re \alpha_1 \geq 0$, the error term in evaluating $S_{1,\leq X}$ will be
$$ \ll_\varepsilon q^{X/2-X\beta-3g/2 - g \alpha+\varepsilon g}.$$ We choose $X= \frac{g (3+2\alpha-2 \varepsilon)}{1-2\beta \varepsilon}$ and obtain an error term of size $q^{-g  \beta( 3+2 \alpha-\varepsilon)}$ for $E_1$.

\section{Proof of Theorem \ref{theorem3}}\label{sectiontheorem3}

To prove Theorem \ref{theorem3}, we need the following lower bound.
\begin{lemma}
Let $\beta > 0$ and let $N$ be a positive integer. Then
\begin{align*}
\log | L( \tfrac12+\beta+ it, \chi_D) | &\geq \frac{2g}{N+1} \log \Big( \frac{ 1- q^{- (N+1)\beta}}{1- q^{-2(N+1)}} \Big) +  \Re \bigg( \sum_{d(f ) \leq N} \frac{ b_{\beta} (d(f)) \chi_D(f) \Lambda(f)}{|f|^{1/2+it}} \bigg) \\
&\qquad\qquad+O(1),
\end{align*}
where $b_{\beta}(n)$ is given in \eqref{balpha}.
 \label{lb}
\end{lemma}
\begin{proof}
The proof is similar to the proof of Lemma $8.1$ in \cite{F4} and we will only sketch it. Let
$$f(x) = \log \Big(  \frac{a^2+ \sin^2 x}{b^2 + \sin^2 x} \Big)$$ with
$$ a = \frac{q^2-1}{2q}\qquad\text{and}\qquad  b= \frac{q^{\beta}-1}{2q^{\beta/2}}.$$
Equation $8.2$ in \cite{F4} gives
$$ \log | L( \tfrac12+\beta+it, \chi_D) | = -\frac{1}{2} \sum_{j=1}^{2g} f_2(\theta_j)+O(1),$$ where
$$  f_1(x) = f( \pi x) - ( 2 - \beta ) \log q\qquad\text{and} \qquad f_2(x) = f_1 \Big( x - \frac{t \log q}{2 \pi }\Big).$$
We want to find an appropriate majorant for $f_1$ and then use the explicit formula. We will prove the following.

\begin{lemma}
Let $\beta > 0$ and let $N$ be a positive integer. If $r$ is a real valued trigonometric polynomial of degree at most $N$ such that $r(x) \geq f_1(x)$ for all $x\in\mathbb{R}/ \mathbb{Z}$, then
$$ \int_{\mathbb{R} / \mathbb{Z}} r(x) \, dx \geq -\frac{2}{N+1} \log \Big( \frac{1 - q^{-(N+1)  \beta}}{1-q^{-2(N+1)}}\Big), $$ with equality if and only if $r(x) = \sum_{|n| \leq N} \widehat{r}_\mu(N,n) e(nx),$ where $\widehat{r}_\mu(N,n)$ is given in \eqref{rmun1} and \eqref{rmun2}.
\label{majorant}
\end{lemma}
\begin{proof}
The proof is similar to the proof of Lemma $8.3$ in \cite{F4}. Here, we will use ideas from \cite{CLV}. Following the notation in \cite{CLV}, let $G_{\lambda}(x) = e^{-\pi \lambda x^2}$ with $\lambda>0$ and let
$$ M_{\lambda}(x) = \Big( \frac{\sin \pi x}{\pi} \Big)^2\bigg( \sum_{n=-\infty}^{\infty} \frac{ G_{\lambda}(n)}{(x-n)^2} + \sum_{n=-\infty}^{\infty} \frac{ G_{\lambda}'(n)}{x-n}\bigg).$$
For $x \in \mathbb{R} / \mathbb{Z}$, let 
$$ m(\lambda, N,x) = \frac{\lambda^{1/2}}{N+1} \sum_{|n| \leq N} \widehat{M} \Big(  \frac{\lambda}{(N+1)^2}, \frac{n}{N+1} \Big) e(nx).$$
For $\tau$ a complex number with $\Im \tau>0$, let
$$ \theta_3(v,\tau) = \sum_{n=-\infty}^{\infty} e^{ \pi i \tau n^2} e(nv).$$ Similarly as in the proof of Lemma $8.3$ in \cite{F4}, let $p: (0 , \infty) \times \mathbb{R}/ \mathbb{Z}\rightarrow\mathbb{R}$ be defined by
$$p (\lambda, x) = -\lambda^{-1/2} + \lambda^{- 1/2} \sum_{n=-\infty}^{\infty}  e^{- \pi\lambda^{-1} n^2}e(nx) = - \lambda^{-1/2}+ \lambda^{-1/2} \theta_3(x, i \lambda^{-1}).$$
Theorem $6$ in \cite{CLV} shows that if $r(x)$ is a real trigonometric polynomial of degree at most $N$ such that $r(x)\geq p (\lambda, x) $ for all $x \in \mathbb{R}/ \mathbb{Z}$, then 
$$\int_{\mathbb{R}/\mathbb{Z}} r(x) \, dx\geq - \lambda^{-1/2}+ \lambda^{-1/2} \theta_3\big(0, i \lambda^{-1} (N+1)^2\big),$$ with equality if and only if $r(x) = -\lambda^{-1/2}+ \lambda^{-1/2} m(\lambda, N,x)$.

Now let $\mu$ be the finite non-negative Borel measure on $(0,\infty)$ defined by
$$ d \mu(\lambda) = \frac{e^{- \pi \lambda c^2} - e^{- \pi \lambda d^2}}{\lambda} \, d \lambda,$$ where $0<c<d$, and let $h_{\mu}(x) = \int_0^{\infty} p(\lambda, x) d \mu(\lambda).$ Define
$$ r_{\mu}(x) = \sum_{|n| \leq N} \widehat{r}_{\mu}(N,n) e(nx),$$ where
\begin{equation}\label{rmun1}
\widehat{r}_{\mu}(N,n) =\frac{1}{N+1} \int_0^{\infty} \widehat{M} \Big(  \frac{\lambda}{(N+1)^2}, \frac{n}{N+1} \Big) d \mu(\lambda)
\end{equation} for $n \neq 0$ and
\begin{equation}\label{rmun2} 
\widehat{r}_{\mu}(N,0) = \int_0^{\infty} \Big( -\lambda^{-1/2} + \lambda^{-1/2} \theta_3\big(0, i \lambda^{-1} (N+1)^2\big) \Big) d \mu(\lambda).
\end{equation}
Using Theorem $6$  and Corollary $17$ in \cite{CLV}, it follows that $r_{\mu}(x)$ is the optimal majorant for $h_{\mu}(x)$. That is to say if $r(x) \geq h_{\mu}(x)$ for all $x \in \mathbb{R}/ \mathbb{Z}$, then $$\int_{\mathbb{R}/\mathbb{Z}} r(x) \, dx \geq \int_{\mathbb{R}/ \mathbb{Z}} r_{\mu}(x) \, dx.$$

We pick $$c = \frac{\beta\log q}{2 \pi }\qquad\text{and}\qquad d = \frac{ \log q}{ \pi }.$$ From the proof of Lemma $8.3$ in \cite{F4}, we have that $h_{\mu}=f_1.$ Note that
\begin{align*}
\widehat{r}_{\mu}(N,0) &= \int_0^{\infty} \Big( -\lambda^{-1/2}+ \lambda^{-1/2} \theta_3\big( 0, i \lambda^{-1} (N+1)^2\big) \Big) d \mu(\lambda) \\
&= -\frac{2}{N+1} \log \Big( \frac{1-q^{-(N+1)\beta}}{1-q^{-2(N+1)}}\Big),
\end{align*} 
which finishes the proof of Lemma \ref{majorant}.
\end{proof}

Now we go back to the proof of Lemma \ref{lb}. By the explicit formula in Lemma \ref{explicit} and Lemma \ref{majorant} it follows that 
\begin{align*}
\log | L( \tfrac12+\beta+it,\chi_D) | &\geq \frac{2g}{N+1} \log \Big(  \frac{1- q^{-(N+1)\beta}}{1- q^{-2(N+1)}} \Big) \\
&\qquad\qquad+ \Re \bigg( \sum_{d(f)\leq N} \frac{ \widehat{r}_\mu(N,d(f)) \chi_D(f) \Lambda(f)}{|f|^{1/2+it}} \bigg)+O(1),
\end{align*} where $r_\mu$ is the optimal majorant found in Lemma \ref{majorant}. Let $b_{\beta}(n) = \widehat{r}_\mu(N,n)$. From the proof of Lemma \ref{majorant}, recall that for $n \neq 0,$
$$ b_{\beta}(n) = \frac{1}{N+1}\int_0^{\infty} \widehat{M} \Big(  \frac{\lambda}{(N+1)^2}, \frac{n}{N+1} \Big) d \mu(\lambda),$$ and from Theorem $4$ in \cite{CLV} we have
\begin{align*}
\widehat{M} \Big(  \frac{\lambda}{(N+1)^2}, \frac{n}{N+1} \Big) &= \Big( 1- \frac{ |n|}{N+1} \Big) \theta_3 \Big(  \frac{n}{N+1}, \frac{ i \lambda}{(N+1)^2} \Big)\\
&\qquad\qquad -\frac{ \lambda\,\text{sgn} (n)}{2 \pi(N+1)^2}  \frac{ \partial \theta_3}{\partial t} \Big(  \frac{n}{N+1}, \frac{i \lambda}{(N+1)^2} \Big).
\end{align*}
Similarly as in the proof of Lemma $8.3$ in \cite{F4}, for $|n| \leq N$, we have that
\begin{align}
b_{\beta}(n)& =  \sum_{j=0}^{\infty}  (j+1) \bigg( \frac{1}{|n|+j(N+1)} \Big( q^{-(|n|+j(N+1))\beta}-q^{-2(|n|+j(N+1))} \Big) \nonumber \\
&\qquad\quad-\frac{1}{(j+2)(N+1) - |n|} \Big( q^{(|n|-(j+2)(N+1))\beta} - q^{2(|n|-(j+2)(N+1))} \Big) \bigg)\nonumber\\
&=\sum_{j=0}^{\infty}  (j+1) \bigg( \frac{q^{-(|n|+j(N+1))\beta}}{|n|+j(N+1)} -\frac{q^{(|n|-(j+2)(N+1))\beta}}{(j+2)(N+1) - |n|}\bigg)+O\Big(\frac{q^{-2|n|}}{|n|}\Big). \label{balpha}
\end{align}
\end{proof}

Using Lemma \ref{lb}, we can prove the following. Similar bounds for the Riemann zeta-function were obtained in \cite{CC}.
\begin{lemma}
\label{initial}
If $0<\beta\ll 1/\log g$, then we have
\begin{align} \frac{1}{ |L(1/2+\beta+it,\chi_D)|} &\leq  \exp \bigg( -\Big(1+O\Big(\frac{1}{\log g}\Big)\Big) \frac{g}{\log_q g} \log(1-g^{-2\beta}) \bigg). \label{bd1}
\end{align}
If $0<\beta<1/2$ such that $(1/2-\beta) \log g = O(1)$, then
\begin{equation}
  \frac{1}{ |L(1/2+\beta+it,\chi_D)|} \ll \log g. \label{bd2}
  \end{equation}
If $0<\beta<1/2$ and neither of the above two conditions on $\beta$ are satisfied, then
\begin{align}
 \frac{1}{ |L(1/2+\beta+it,\chi_D)|} \leq \exp \Bigg( \frac{g^{1-2 \beta}}{\log_q g} \bigg(2+ \frac{q^{1/2+\beta}-q^{1/2-\beta}}{(q^{1/2-\beta}-1)(q^{1/2+\beta}-1)} \bigg) \bigg( 1+ O \bigg( \frac{1}{\log g} \bigg) \bigg)\Bigg).\label{bd3}
\end{align}
\end{lemma}
\begin{proof}
Let $N+1=2 \log_q g$. Using Lemma \ref{lb} and the expression \eqref{balpha} we get that
\begin{align}\label{4001}
&\log  |L(\tfrac12+\beta+it,\chi_D)|  \geq \frac{2g}{N+1} \log \Big(1-q^{-(N+1)\beta} \Big) \\
&\qquad- \sum_{d(f) \leq N} \frac{\Lambda(f)}{|f|^{1/2+\beta}} \sum_{j=0}^{\infty} \bigg( \frac{(j+1)q^{-j(N+1)\beta}}{d(f)+j(N+1)}-\frac{(j+1)q^{-(j+2)(N+1)\beta} |f|^{2\beta}}{  (j+2)(N+1)-d(f)} \bigg) +O(1).\nonumber
\end{align}
For $d(f)\leq N$ we have
\begin{align}\label{4002}
&\sum_{j=1}^{\infty} \bigg( \frac{(j+1)q^{-j(N+1)\beta}}{d(f)+j(N+1) }-\frac{(j+1)q^{-(j+2)(N+1)\beta} |f|^{2\beta}}{  (j+2)(N+1)-d(f)} \bigg)\nonumber \\
&\qquad=2\big (N+1-d(f)\big)\sum_{j=1}^{\infty} \frac{(j+1)q^{-j(N+1)\beta}}{(d(f)+j(N+1))((j+2)(N+1)-d(f))}\nonumber \\
&\qquad\qquad+ \sum_{j=1}^{\infty} \frac{(j+1) q^{-j(N+1)\beta}(1-q^{-2(N+1)\beta}|f|^{2\beta})}{(j+2)(N+1)-d(f)}\\
&\qquad < \frac{2\big(N+1-d(f)\big)}{(N+1)^2} \sum_{j=1}^{\infty} \frac{q^{-j(N+1)\beta}}{j }+O\Big(\frac{\big(N+1-d(f)\big)\beta}{N} \sum_{j=1}^{\infty} q^{-j(N+1)\beta}\Big)\nonumber\\
&\qquad= - \frac{2\big(N+1-d(f)\big)}{(N+1)^2}\log \Big(1-q^{-(N+1)\beta} \Big) +O\Big(\frac{N+1-d(f)}{N^2} \Big).\nonumber
\end{align}
Putting this into \eqref{4001} we obtain that
\begin{align}
&\log  |L(\tfrac12+\beta+it,\chi_D)|  > \frac{2g}{N+1} \log \Big(1-q^{-(N+1)\beta} \Big)\\
&\qquad-\sum_{d(f) \leq N} \frac{\Lambda(f)}{|f|^{1/2+\beta}} \bigg(\frac{1}{d(f)}-\frac{q^{-2(N+1)\beta} |f|^{2\beta}}{ 2(N+1)-d(f)}- \frac{2\big(N+1-d(f)\big)}{(N+1)^2}\log \Big(1-q^{-(N+1)\beta} \Big)\nonumber  \\
&\qquad\qquad\qquad\qquad\qquad\qquad+O\Big(\frac{N+1-d(f)}{N^2} \Big) \bigg). \label{inl}
\end{align}

Using the Prime Polynomial Theorem in the form
$$\sum_{f \in \mathcal{M}_n} \Lambda(f) = q^n,$$ 
we have that the sum over $f$ becomes
\begin{align}
\sum_{n=1}^Nq^{n(1/2-\beta)}\bigg(\frac{1}{n}-\frac{q^{-2(N+1-n)\beta} }{ 2(N+1)-n}- \frac{2\big(N+1-n\big)}{(N+1)^2}\log \Big(1-q^{-(N+1)\beta} \Big)+O\Big(\frac{N+1-n}{N^2} \Big) \bigg).
\label{sumf}
\end{align}
If $(1/2-\beta) \log g = O(1)$, then we write $q^{n(1/2-\beta)}= 1+O(n/\log g)$, and then the sum over $f$ becomes 
\begin{align*}
\sum_{n=1}^N \frac{1}{n} + O(1) = \log N+ O(1) = \log \log g+O(1),
\end{align*}
where we used the fact that $N+1=2 \log_q g$. This gives \eqref{bd2}.

Now we assume that $(1/2-\beta) \log g \neq O(1)$. 
Using partial summation, we then see that the sum over $f$ is 
\begin{align}
& \frac{q^{(N+1)(1/2-\beta)}}{(q^{1/2-\beta}-1)N}-\frac{q^{(N+1)(1/2-\beta)}}{(q^{1/2+\beta}-1)(N+2)}- O\Big(\frac{q^{(N+1)(1/2-\beta)}}{N^2}\Big)\log \Big(1-q^{-(N+1)\beta} \Big) + O \Big( \frac{q^{(N+1)(1/2-\beta)}}{N^2} \Big)\nonumber \\
&= \frac{q^{(N+1)(1/2-\beta)}(q^{1/2+\beta}-q^{1/2-\beta})}{(q^{1/2-\beta}-1)(q^{1/2+\beta}-1)N}-  O\Big(\frac{q^{(N+1)(1/2-\beta)}}{N^2}\Big)\log \Big(1-q^{-(N+1)\beta} \Big) + O \Big( \frac{q^{(N+1)(1/2-\beta)}}{N^2} \Big) .\label{partial_sum}
\end{align}
Now if $\beta \ll 1/\log g$, then $q^{1/2+\beta}-q^{1/2-\beta}= O(\beta)$. As $N+1= 2 \log_q g$ the first bound \eqref{bd1} now follows.

If $(1/2-\beta)\log g \neq O(1)$ and $\beta \neq O(1\log g)$, then combining \eqref{inl} and \eqref{partial_sum},  \eqref{bd3} follows.

\end{proof}

\begin{remark}
\emph{We will use Lemma \ref{initial} in the following form. For $0<\beta \ll 1/\log g$, we have 
\begin{equation}
\frac{1}{ |L(1/2+\beta+it,\chi_D)|} \leq \Big( \frac{1}{1-g^{-2\beta}} \Big)^{  \frac{ (1+\varepsilon)g }{\log_q g}} ,
\label{ub_first}
\end{equation}
and for $\beta \neq O(1\log g)$ and $(1/2-\beta) \log g \neq O(1)$, we have
\begin{equation}
\frac{1}{ |L(1/2+\beta+it,\chi_D)|} \leq \exp \Bigg( \frac{g^{1-2\beta}}{\log_q g} \Big( 2+ \frac{q^{1/2+\beta}-q^{1/2-\beta}}{(q^{1/2-\beta}-1)(q^{1/2+\beta}-1)} +\varepsilon  \Big) \Bigg).\label{ub_second}
\end{equation}
}
\end{remark}


\begin{proof}[Proof of Theorem \ref{theorem3}]

Let 
$$\beta=\min \{\beta_1,\ldots,\beta_k\}.$$
We assume that $\beta = O(1/\log g)$, which is the more difficult case. We will only sketch the proof when $\beta \neq O(1\log g)$. 

Let $a<2$, $1/2<d \leq 1-\varepsilon$, $r>1$ be constants to be chosen later. 
Let
\begin{equation} 
N_0 =\bigg[- \frac{ (d-1/2)(\log g)^2}{(1+2\epsilon) (\log q) km\log (1-g^{-2\beta} )}  \bigg] , \label{n0} 
\end{equation}
\begin{equation} s_0 =  2\bigg[ \frac{-(1+2\varepsilon) (\log q) km g \log (1-g^{-2\beta} )}{2(d-1/2) (\log g)^2}  \bigg] \quad \ell_0=s_0^d.
\end{equation} 
For $1\leq j \leq K$, let 
\begin{equation*}
N_j =[ r( N_{j-1}+1)], \quad s_j = 2 \bigg[ \frac{ag}{2N_j}  \bigg] \quad \text{and}\quad \ell_j = 2 \bigg[ \frac{s_j^d}{2} \bigg], \label{param}
\end{equation*}
where we choose $K$ so that 
\begin{equation}
N_K = \Big[\frac{ \log_q ( g\beta)}{\beta} \Big].
\label{rk}
\end{equation}
Note that from our choice of parameters, we have $s_0 N_0 \leq g$, 
\begin{equation*}
\sum_{j=0}^K \ell_j N_j \leq 2g, \label{cond1}
\end{equation*}
and for $0 \leq j <K$, we have
\begin{equation*}
\sum_{r=0}^j \ell_r N_r +s_{j+1} N_{j+1} \leq 2g.
\label{cond2}
\end{equation*}

Let $I_0 = (0, N_{0}], I_1=(N_0,N_1], \ldots, I_K=(N_{K-1}, N_K]$.  Let 
\begin{align} \label{ap}
a_{\alpha}(P;N) & = -\cos\big( t d(P) \log q\big)  \\
&\qquad\qquad\times \sum_{j=0}^{\infty}\bigg(\frac{ ( j+1) d(P)q^{- j(N+1)\alpha}}{ d(P) + j(N+1) }-\frac{ (j+1) d(P) q^{ -(j+2)(N+1)\alpha}|P|^{2 \alpha}}{  (j+2)(N+1) - d(P)}\bigg)\nonumber
\end{align}
and 
\begin{equation}
c(P;N) = \sum_{j=1}^k \frac{ a_{\beta_j}(P;N)}{|P|^{\beta_j-\beta}}.
\label{cp}
\end{equation}
We extend $c(P;N)$ to a completely multiplicative function in the first variable. 

For $d(P)\leq N$ we have
$$\big| a_{\beta_j}(P;N) \big| \leq 1+ \frac{1}{q^{(N+1)\beta}-1},$$
if $N\beta_j \gg 1$, and as in \eqref{4002}
\begin{align*}
\big|a_{\beta_j}(P;N)\big|&<  1-\frac{d(P)q^{-2(N+1)\beta_j} |P|^{2\beta_j}}{  2(N+1)-d(P)} - \frac{2\big(N+1-d(P)\big)d(P)}{(N+1)^2}\log \Big(1-q^{-(N+1)\beta_j} \Big) \\
&\qquad\qquad+O\Big(\frac{(N+1-d(P))d(P)}{N^2} \Big)\\
&<-\frac{1}{2}\log \Big(1-q^{-(N+1)\beta_j} \Big)+O(1)\ll \log\frac{1}{\beta}
\end{align*}
if $N\beta_j \ll 1$. It follows from \eqref{cp} that for $d(P) \leq N$ ,
\begin{equation}
c(P;N) \leq k+ \frac{k}{q^{(N+1)\beta}-1}.
\label{bign}
\end{equation} 
if $N\beta \gg 1$, and there exists some constant $A>0$ such that
\begin{equation}
c(P;N) \leq Ak \log \frac{1}{\beta} 
\label{c_bound}
\end{equation}
if $N\beta \ll 1$.

For $0 \leq r \leq K$, let
$$P_{I_r}(D;N_j) = \sum_{d(P) \in I_r} \frac{c(P;N_j)\chi_D(P) }{|P|^{1/2+\beta}}.$$
We will first prove the following lemma.
\begin{lemma}
\label{cases}
Let $k$ be positive. We either have
$$ \max_{0 \leq u \leq K}  \Big|  P_{I_0} (D; N_u) \Big| > \frac{\ell_0}{me^2},$$ or 
\begin{align}\label{ineq2}
& \frac{1}{ \prod_{j=1}^k | L(1/2+\beta_j+it_j,\chi_D)|^m}  \nonumber\\
& \ll \Big(1-q^{-(N_K+1)\beta} \Big)^{-\frac{2gkm}{N_K+1}} (\log g)^{\frac{mk}{2}} \prod_{r=1}^k   \min \Big\{ \frac{1}{\beta_r}, \min \Big\{ N_K, \frac{1}{\overline{t_r}}\Big\}\Big\}^{-m/2} \\
& \times  \prod_{r=0}^K (1+e^{-\ell_r/2}) E_{\ell_r} \Big(m P_{I_r}(D;N_K)\Big)   + \sum_{0 \leq j \leq K-1} \sum_{j < u \leq K} \Big(1-q^{-(N_j+1)\beta} \Big)^{-\frac{2gkm}{N_j+1}} (\log g)^{\frac{mk}{2}} \nonumber \\
& \times  \prod_{r=1}^k   \min \Big\{ \frac{1}{\beta_r}, \min \Big\{ N_j, \frac{1}{\overline{t_r}}\Big\}\Big\}^{-m/2}  \prod_{r=0}^j (1+e^{-\ell_r/2}) E_{\ell_r} \Big( m P_{I_r} (D;N_j)\Big)  \Big(\frac{me^2}{\ell_{j+1}} P_{I_{j+1}} (D;N_u) \Big)^{s_{j+1}}, \nonumber
\end{align}
where $\overline{t} = \min \{ t \bmod {2 \pi}, 2 \pi - (t \bmod {2 \pi})\} .$
\end{lemma}

\begin{proof}
For $r \leq K$, let
$$\mathcal{T}_r = \Big\{ D \in \mathcal{H}_{2g+1} \, | \,   \max_{r \leq u \leq K} \big| P_{I_r} ( D; N_u) \big|  \leq \frac{\ell_r}{me^2} \Big\}.$$
We have the following possibilities:
\begin{enumerate}
\item $D \not \in \mathcal{T}_0$;
\item $D \in \mathcal{T}_r$ for all $r \leq K$;
\item There exists $0 \leq j \leq K-1$ such that $D \in \mathcal{T}_r$ for all $r \leq j$, and $D \not \in \mathcal{T}_{j+1}$.
\end{enumerate}
The first condition corresponds to the first statement of the lemma.

 If the second condition is satisfied, then we use Lemma \ref{lb} and we pick $N= N_K$. We use the expression \eqref{balpha} for $b_{\beta_j}(m)$, evaluate the contribution from $f=P^2$  and bound the contribution from $f=P^i$ with $i \geq 3$ by $O(1)$ in Lemma \ref{lb}. Also note that using \eqref{balpha}, the second and the fourth terms will be bounded by $O(1)$ when summing over the primes, so we get that  
\begin{align*}
&\sum_{j=1}^k  m \log |L(\tfrac12+\beta_j+it_j,\chi_D)|  \geq \frac{2gkm}{N_K+1} \log \Big(1-q^{-(N_K+1)\beta} \Big) \\
&\qquad\qquad- m\sum_{d(P) \leq N_K} \frac{  c(P;N_K) \chi_D(P)}{|P|^{1/2+\beta}} + \frac{m}{2} \sum_{r=1}^k \sum_{\substack{d(P) \leq N_K/2 \\ P \nmid D}} \frac{ \cos(2t_r d(P) \log q)}{|P|^{1+2\beta_j}} +O(1) .
\end{align*} 
Now using the fact that 
$$\sum_{P|D} \frac{1}{|P|} \leq \log\log g+O(1),$$ and using Lemma \ref{sumscos}, it follows that
\begin{align*}
\sum_{j=1}^k  & m \log |L(\tfrac12+\beta_j+it_j,\chi_D)|  \geq \frac{2gkm}{N_K+1} \log \Big(1-q^{-(N_K+1)\beta} \Big)- m\sum_{d(P) \leq N_K} \frac{  c(P;N_K) \chi_D(P)}{|P|^{1/2+\beta}} \\
&+ \frac{m}{2} \sum_{r=1}^k \log \min \Big\{ \frac{1}{\beta_r}, \min \Big\{ N_K, \frac{1}{\overline{t_r}}\Big\}\Big\}- \frac{mk}{2} \log \log g+O(1).
\end{align*}

 We exponentiate the expression above and use inequality \eqref{taylor}. Since $D \in \mathcal{T}_r$ for all $r \leq K$, we obtain the first term in \eqref{ineq2}.

If the third condition is satisfied, then we pick $N=N_j$ in Lemma \ref{lb}. Since $D \not \in \mathcal{T}_{j+1}$, it follows that there exists some $u \geq j+1$ such that $| P_{I_{j+1}} (D; N_u) | > \frac{\ell_{j+1}}{me^2}$, and since $s_{j+1}$ is even, we have
$$1<  \Big(\frac{me^2}{\ell_{j+1}} P_{I_{j+1}} (D; N_u) \Big)^{s_{j+1}},$$ and we proceed as in the previous case.
\end{proof}

We now return to the proof of Theorem \ref{theorem3}. We use Lemma \ref{cases}. If $D \notin \mathcal{T}_0$, then there exists some $0 \leq u \leq K$ such that 
$$ 1< \Big(  \frac{me^2}{\ell_0} P_{I_0} (D;N_u) \Big)^{s_0},$$ since $s_0$ is even. Then using the Cauchy-Schwarz inequality, we have 
\begin{align}\label{t0}
\sum_{D \notin \mathcal{T}_0} & \prod_{j=1}^k \frac{1}{|L(1/2+\beta_j+it_j,\chi_D)|^m}\nonumber\\
& \leq \sum_{D \in \mathcal{H}_{2g+1}}  \prod_{j=1}^k \frac{1}{|L(1/2+\beta_j+it_j,\chi_D)|^m}  \Big( \frac{me^2}{\ell_0} P_{I_0}(D;N_u) \Big)^{s_0}   \\
& \leq \Big( \frac{me^2}{\ell_0} \Big)^{s_0} \bigg( \sum_{D \in \mathcal{H}_{2g+1}}  \prod_{j=1}^k \frac{1}{|L(1/2+\beta_j+it_j,\chi_D)|^{2m}} \bigg)^{1/2}  \bigg( \sum_{D \in \mathcal{H}_{2g+1}}  P_{I_0} (D; N_u) ^{2 s_0} \bigg)^{1/2}.\nonumber 
\end{align}
For the first term in the inequality above, we use the pointwise bound \eqref{ub_first} for each of the $L$--functions. For the second term, since each of the summands is positive, we bound the sum over $\mathcal{H}_{2g+1}$ by the sum over all $D \in \mathcal{M}_{2g+1}$ and using Lemma \ref{power}, we have the following
\begin{align*}
 \sum_{D \in \mathcal{H}_{2g+1}}  P_{I_0}(D;N_u) ^{2 s_0}& \leq \sum_{D \in \mathcal{M}_{2g+1}}  \bigg( \sum_{d(P) \in I_0} \frac{c(P;N_u)\chi_D(P)}{|P|^{1/2+\beta}} \bigg)^{2 s_0}\\
 & = \sum_{D \in \mathcal{M}_{2g+1}} (2s_0)! \sum_{\substack{P | f \Rightarrow d(P) \in I_0 \\ \Omega(f) = 2 s_0}} \frac{c(f;N_u) \nu(f)\chi_D(f)  }{|f|^{1/2+\beta}}.
\end{align*}
We interchange the sum over $D$ and the sum over $f$ and note that $d(f) \leq 2 s_0 N_0$. If $f \neq \square$, then the sum over $D$ vanishes, since $d(D) = 2g+1 > 2 s_0 N_0$ from our choice of parameters \eqref{n0}. It then follows that
\begin{align*}
 \sum_{D \in \mathcal{H}_{2g+1}} &  P_{I_0}(D;N_u) ^{2 s_0} \leq q^{2g+1} (2 s_0)! \sum_{\substack{P | f \Rightarrow d(P) \in I_0 \\ \Omega(f)=s_0}} \frac{c(f;N_u)^2 \nu(f^2)}{|f|^{1+2 \beta}} .
 \end{align*}
 Now since $N_0 \leq N_u$, using \eqref{c_bound} it follows that for $f$ as above
 \begin{equation*}
 c(f;N_u) \leq A^{\Omega(f)} k^{\Omega(f)} \Big(\log  \frac{1}{\beta} \Big)^{\Omega(f)},
 \end{equation*}
  and hence
 \begin{align}
 \sum_{D \in \mathcal{H}_{2g+1}} &  P_{I_0}(D;N_u)^{2 s_0}  \ll q^{2g+1} (2 s_0)!  \sum_{\substack{P | f \Rightarrow d(P) \in I_0 \\ \Omega(f)=s_0}} \frac{ A^{2 \Omega(f)} k^{2 \Omega(f)} (\log  \frac{1}{\beta} )^{2 \Omega(f)}  \nu(f)}{|f|^{1+2 \beta}} \nonumber\\
 &  = q^{2g+1} \frac{(2 s_0)!}{s_0!} A^{2 s_0} k^{2 s_0} \Big(\log  \frac{1}{\beta} \Big)^{2 s_0} \bigg( \sum_{d(P) \in I_0} \frac{1}{|P|^{1+2 \beta}} \bigg)^{s_0}  \nonumber \\
 &   \ll q^{2g+1} \frac{(2 s_0)!}{s_0!} A^{2 s_0} k^{2 s_0} \Big( \log \frac{1}{\beta} \Big)^{2 s_0} (\log N_0)^{s_0}. \label{sump}
\end{align}
Combining equations \eqref{t0} and \eqref{sump}, Stirling's formula and keeping in mind the choice of parameters \eqref{n0}, \eqref{s0}, it follows that
\begin{align*}
&\sum_{D \notin \mathcal{T}_0}   \prod_{j=1}^k  \frac{1}{|L(1/2+\beta_j+it_j,\chi_D)|^m} \\
&\quad \ll q^{2g} \Big(\frac{Akm e^2 \log \frac{1}{\beta}}{\ell_0} \Big)^{s_0} \Big( \frac{1}{1-g^{-2\beta}} \Big)^{\frac{ (1+\varepsilon) kmg}{\log_q g}} \sqrt{ \frac{ (2 s_0)!}{s_0!}}    (\log N_0)^{s_0/2}\\
& \quad  \ll q^{2g}  \Big( \frac{1}{1-g^{-2\beta}} \Big)^{\frac{ (1+\varepsilon) kmg}{\log_q g}} \exp \Big(-\frac{s_0 \log s_0}{2}  \Big)\exp \Big(s_0 \log \Big( 2e^{3/2}  A  km\Big( \log  \frac{1}{\beta} \Big) \sqrt{ \log N_0}  \Big) \Big).
 \end{align*}
Since $\log\frac{1}{\beta} \ll \log g$, we have
\begin{align}
\sum_{D \notin \mathcal{T}_0}  & \prod_{j=1}^k  \frac{1}{|L(1/2+\beta_j+it_j,\chi_D)|^m}  = o(q^{2g}). \label{t0'}
\end{align}

Now assume that $D \in \mathcal{T}_0$ and that \eqref{ineq2} holds. Incorporating the terms of the form $\prod_{r=0}^j (1+e^{-\ell_r/2})$ into the error term, we then have that
\begin{align}\label{eq3}
& \sum_{D \in \mathcal{T}_0} \prod_{j=1}^k \frac{1}{ |L(1/2+\beta_j+it_j,\chi_D)|^m}\nonumber\\
&\qquad \ll  \Big(1-q^{-(N_K+1)\beta} \Big)^{-\frac{2gkm}{N_K+1}} (\log g)^{\frac{mk}{2}} \prod_{r=1}^k  \min \Big\{ \frac{1}{\beta_r}, \min \Big\{ N_K, \frac{1}{\overline{t_r}}\Big\}\Big\}^{-m/2} \\
& \times \sum_{D \in \mathcal{M}_{2g+1}} \prod_{r=0}^K  E_{\ell_r} \Big(m P_{I_r} (D;N_K) \Big) + \sum_{0 \leq j \leq K-1} (\log g)^{\frac{mk}{2}} \prod_{r=1}^k  \min \Big\{ \frac{1}{\beta_r}, \min \Big\{ N_j, \frac{1}{\overline{t_r}}\Big\}\Big\}^{-m/2} \nonumber \\
& \times  \sum_{j < u \leq K}  \Big(1-q^{-(N_j+1)\beta} \Big)^{-\frac{2gkm}{N_j+1}}  \sum_{D \in \mathcal{M}_{2g+1}}  \prod_{r=0}^j  E_{\ell_r} \Big( m P_{I_r} (D;N_j) \Big)  \Big(\frac{me^2}{\ell_{j+1}} P_{I_{j+1}} (D;N_u) \Big)^{s_{j+1}}. \nonumber
\end{align} 

Now we focus on the second term in \eqref{eq3}. We have
\begin{align} \label{tb}
&\sum_{D \in \mathcal{M}_{2g+1}}  \prod_{r=0}^j  E_{\ell_r} \Big( m P_{I_r} (D;N_j) \Big)  \Big(\frac{me^2}{\ell_{j+1}} P_{I_{j+1}} (D;N_u) \Big)^{s_{j+1}}  \nonumber \\
&\qquad\qquad= \Big( \frac{me^2}{\ell_{j+1}}\Big)^{s_{j+1}} (s_{j+1})! \sum_{D \in \mathcal{M}_{2g+1}}\prod_{r=0}^j \bigg( \sum_{\substack{P | f_r \Rightarrow d(P) \in I_r \\ \Omega(f_r) \leq \ell_r}} \frac{m^{\Omega(f_r)}  c(f_r;N_j) \nu(f_r)\chi_D(f_r)}{|f_r|^{1/2+\beta}} \bigg)\nonumber\\
&\qquad\qquad\qquad\qquad\times \bigg( \sum_{\substack{P|f_{j+1} \Rightarrow d(P) \in I_{j+1} \\ \Omega(f_{j+1})=s_{j+1}}} \frac{ c(f_{j+1};N_u) \nu(f_{j+1})\chi_D(f_{j+1})}{|f_{j+1}|^{1/2+\beta}} \bigg).
\end{align}
Interchanging the sum over $D$ with the sums over the $f_i$, note that if $f_0 \cdot \ldots \cdot f_{j+1} \neq \square$, then the sum over $D$ vanishes since $2g+1 > \sum_{r=0}^{j} \ell_r N_r +s_{j+1} N_{j+1} \geq d(f_0 \cdot \ldots \cdot f_{j+1})$. We are then only left with the diagonal terms corresponding to $f_0 \cdot \ldots \cdot f_{j+1} = \square$ in the equation above. Since the $f_r$ are pairwise coprime, it follows that their product is a square if and only if each $f_r = \square$ for $r \leq j+1$. Note that for $r \leq j$ and $f_r$ as above, from \eqref{c_bound}, we have 
\begin{equation*}
c(f_r;N_j ) \leq A^{\Omega(f_r)} k^{\Omega(f_r)}  \Big( \log\frac{1}{\beta} \Big)^{\Omega(f_r)},
\end{equation*}
 and for $f_{j+1}$ we similarly have 
 $$c(f_{j+1};N_u) \leq A^{\Omega(f_{j+1})} k^{\Omega(f_{j+1}})  \Big(\log \frac{1}{\beta} \Big)^{\Omega(f_{j+1})}.$$ Bounding  $\nu(f_r^2) \leq \nu(f_r) / 2^{\Omega(f_r)} \leq 1/2^{\Omega(f_r)}$ and $\nu(f_{j+1}^2) \leq \nu(f_{j+1})$, we then get that
\begin{align*}
\eqref{tb} & \ll q^{2g} \Big( \frac{me^2}{\ell_{j+1}} \Big)^{s_{j+1}} (s_{j+1})!   \prod_{r=0}^j \Big( \sum_{\substack{P | f_r \Rightarrow d(P) \in I_r \\ \Omega(f_r) \leq \ell_r/2}} \frac{A^{2 \Omega(f_r)} k^{2 \Omega(f_r)}m^{2 \Omega(f_r)} ( \log \frac{1}{\beta})^{2 \Omega(f_r)}}{2^{\Omega(f_r)} |f_r|^{1+2\beta}} \Big) \\
&\qquad\qquad \times \bigg( \sum_{\substack{P | f_{j+1} \Rightarrow d(P) \in I_{j+1} \\ \Omega(f_{j+1}) = s_{j+1}/2}} \frac{ A^{2 \Omega(f_{j+1})} k^{2 \Omega(f_{j+1})} (\log  \frac{1}{\beta})^{2 \Omega(f_{j+1})}  \nu(f_{j+1})}{|f_{j+1}|^{1+2\beta}} \bigg) \\
& \ll q^{2g}  \Big( \frac{me^2}{\ell_{j+1}} \Big)^{s_{j+1}} (s_{j+1})!  \prod_{r=0}^j  \bigg( \prod_{d(P) \in I_r} \bigg(1- \frac{  A^2k^2m^2 (\log\frac{1}{\beta} )^2}{2 |P|^{1+2 \beta}} \bigg)^{-1} \bigg) \\
&\qquad\qquad \times  \frac{1}{ (s_{j+1}/2)!} A^{s_{j+1}} k^{s_{j+1}} \Big( \log \frac{1}{\beta} \Big)^{s_{j+1}}\bigg( \sum_{d(P) \in I_{j+1}} \frac{1}{|P|^{1+2 \beta}} \bigg)^{s_{j+1}/2} \\
& \ll  q^{2g}  \Big( \frac{Akme^2 \log  \frac{1}{\beta} }{\ell_{j+1}} \Big)^{s_{j+1}}  \frac{ (s_{j+1})!}{ (s_{j+1}/2)!}   N_j^{A^2k^2m^2(\log  \frac{1}{\beta})^2/2},
\end{align*}
where in the last line we used Lemma $3.6$ in \cite{BF} and the Prime Polynomial Theorem. 

We proceed similarly for the first term in \eqref{eq3}, but use equation \eqref{bign} instead of \eqref{c_bound} since $N_K\beta \gg 1$, and putting things together, we get that
\begin{align*}
 \sum_{D \in \mathcal{T}_0} \prod_{j=1}^k & \frac{1}{| L(1/2+\beta_j+it_j,\chi_D)|^m}  \ll  q^{2g} \Big(1-q^{-(N_K+1)\beta} \Big)^{-\frac{2gkm}{N_K+1}}  N_K^{k^2m^2/2} (\log g)^{\frac{km}{2}} \\
& \times  \prod_{r=1}^k \min \Big\{ \frac{1}{\beta_r}, \min \Big\{ N_K, \frac{1}{\overline{t_r}}\Big\}\Big\}^{-m/2} \exp \Big( \frac{B \log g}{g^{1-\frac{1}{2km}}} \Big) \\
 &+ q^{2g} (\log g)^{\frac{km}{2}}  \sum_{0 \leq j \leq K-1} \sum_{j<u \leq K} \Big(1-q^{-(N_j+1)\beta} \Big)^{-\frac{2gkm}{N_j+1}}  \Big( \frac{Akme^2 \log \frac{1}{\beta} }{\ell_{j+1}} \Big)^{s_{j+1}}  \frac{ (s_{j+1})!}{ (s_{j+1}/2)!}\\
 & \times  N_j^{A^2k^2m^2(\log \frac{1}{\beta} )^2/2} \prod_{r=1}^k \min \Big\{ \frac{1}{\beta_r}, \min \Big\{ N_j, \frac{1}{\overline{t_r}}\Big\}\Big\}^{-m/2},
\end{align*}
for some constant $B>0$, where we used the fact that $\log N_k \ll \log g$ and the fact that $\beta >g^{-\frac{1}{2km}}$.

Let $S_1$ denote the first term above and $S_2$ the second. We first focus on $S_2$. With the choice of our parameters, using Stirling's formula we get that
\begin{align*}
& S_2 \ll q^{2g}  (\log g)^{\frac{km}{2}}  \sum_{0 \leq j \leq K-1} (K-j) \Big( 1-q^{-(N_j+1)\beta} \Big)^{-\frac{2gkm}{N_j+1}}  \\
& \times  \exp \Big(s_{j+1} \log \Big( \frac{2^{1/2} e^{3/2} A km \sqrt{s_{j+1}} \log \frac{1}{\beta}} {\ell_{j+1}} \Big) \Big) N_j^{A^2k^2m^2(\log \frac{1}{\beta} )^2/2 } \prod_{r=1}^k \min \Big\{ \frac{1}{\beta_r}, \min \Big\{ N_j, \frac{1}{\overline{t_r}}\Big\}\Big\}^{-m/2} \\
&\ll q^{2g}  (\log g)^{\frac{km}{2}} \sum_{0 \leq j \leq K-1} (K-j) \exp \Big( - \frac{2gkm}{N_j+1} \log \Big(1-q^{-(N_j+1)\beta} \Big) \Big) \exp \Big( \Big(\frac{1}{2}-d\Big)  s_{j+1} \log s_{j+1}   \Big)\\
& \times  \exp \Big( s_{j+1} \log \Big(2^{1/2} e^{3/2}A k m\log \frac{1}{\beta} \Big) \Big)  N_j^{A^2k^2m^2(\log \frac{1}{\beta} )^2/2} \prod_{r=1}^k \min \Big\{ \frac{1}{\beta_r}, \min \Big\{ N_j, \frac{1}{\overline{t_r}}\Big\}\Big\}^{-m/2}.
\end{align*}

If $(N_j+1)\beta \log q>\varepsilon$, then from our choice \eqref{rk}, it follows that
 $$ \frac{a(d-1/2)}{r} \Big(\log a+ \log g  - \log r- \log N_K  \Big)> 2km \log  \frac{1}{1-e^{-\varepsilon}} >2km \log  \frac{1}{1-q^{-(N_j+1)\beta}} .$$ Hence the sum over those $j$ with $(N_j+1)\beta \log q>\varepsilon$ will be $O(1)$.
 
If $(N_j+1)\beta \log q <\varepsilon$, then we have
$$ \frac{1}{1-q^{-(N_j+1)\beta}}< \frac{2}{(N_j+1)\beta\log q}.$$ So if 
\begin{equation}
 \frac{a (d-1/2)}{r} \Big( \log a+\log g -\log r -\log (N_j+1)\Big)>2km \Big(\log 2-\log (N_j+1)+ \log\frac{1}{\beta}  - \log \log q \Big),
 \label{cond3}
 \end{equation} then the sum over those $j$ with $(N_j+1)\beta \log q<\varepsilon$ will also be $O(1)$.

 Note that condition \eqref{cond3} follows from
$$ \frac{a(d-1/2)}{r} \log g >2 km \log  \frac{1}{\beta} .$$
Recall that $a<2$, $1/2<d<1$ and $1<r$ so we need
\begin{equation}\log \frac{1}{\beta} < \frac{\log g}{2 km}. \label{condition}
\end{equation}

We now choose the parameters so that 
$$ \frac{a(d-1/2) \log g}{2kmr \log \frac{1}{\beta}}=1+\varepsilon.$$
We let $\varepsilon'>0$ be such that
$$ d = \frac{1}{2} + \frac{ (1+\varepsilon)km \log  \frac{1}{\beta} }{\log g}+\varepsilon'<1,$$
and then
$$ r = 1+ \frac{\varepsilon ' \log g}{ 2(1+\varepsilon)km\log \frac{1}{\beta} }\qquad
a= \frac{ 2(1+\varepsilon)km  \log \frac{1}{\beta} +\varepsilon' \log g}{(1+\varepsilon)km  \log \frac{1}{\beta} +\varepsilon' \log g}.$$
With these choices of parameters, we get that
\begin{equation}
S_2 = o(q^{2g}). 
\label{s2bd}
\end{equation}

Now we focus on bounding $S_1$. Recall that $N_K= [\log_q (g\beta)/\beta]$.
Then 
$$\exp \Big(- \frac{2gkm}{N_K+1} \log \Big(1-q^{-(N_K+1)\beta} \Big) \Big)=\exp(o(1)).$$ We hence get that
\begin{align*}
S_1  & \ll q^{2g}  (\log g)^{\frac{km}{2}} \Big( \frac{\log_q (g\beta)}{\beta} \Big)^{k^2m^2/2} \prod_{r=1}^k \min \Big\{ \frac{1}{\beta_r}, \min \Big\{ \frac{ \log_q(g \beta)}{\beta}, \frac{1}{\overline{t_r}}\Big\}\Big\}^{-m/2}\\
& \ll q^{2g}  (\log g)^{\frac{km(km+1)}{2}} \Big( \frac{1}{\beta} \Big)^{k^2m^2/2} \prod_{r=1}^k \min \Big\{ \frac{1}{\beta_r}, \frac{1}{\overline{t_r}}\Big\}^{-m/2},
\end{align*} and from \eqref{condition}, the bound above holds for 
$$\beta \gg g^{-\frac{1}{2km}+\varepsilon}.$$ 

Now assume that $0<\beta<1/2$ such that $\beta \neq O(1\log g)$ and $(1/2-\beta) \log g \neq O(1)$. Then we choose
\begin{equation*}
N_0 = [\sqrt{\log g}] ,\, \,  s_0 = \ell_0 = 2 \Big[ \frac{2 g^{1-2\beta}}{(1-2 \beta) \log_q g} \Big(2+ \frac{q^{1/2+\beta}-q^{1/2-\beta}}{(q^{1/2-\beta}-1)(q^{1/2+\beta}-1)} \Big)\Big]. 
\end{equation*}
For $1 \leq j \leq K$, let
$$N_j = [e( N_{j-1}+1)] , \, \, s_j = 2 \Big[ \frac{g}{2N_j} \Big], \, \, \ell_j = 2 \Big[ \frac{s_j^{3/4}}{2} \Big],$$
where again we choose $K$ such that
$$N_K = \Big[\frac{ \log_q (g \beta)}{\beta} \Big].$$
Then with this choice of parameters, the same proof as in the case $\beta \ll 1/\log g$ goes through, where in equation \eqref{t0} we use the pointwise bound \eqref{ub_second} for the $L$--function instead of \eqref{ub_first}.  

If $(1/2-\beta) \log g =O(1)$, then we choose
$$ N_0 = [\sqrt{\log g}] ,\, \,  s_0 = \ell_0 = 2 [(\log g)^{km}].$$
For $1 \leq j \leq K$, we let
$$ N_j = [e( N_{j-1}+1)] , \, \, s_j = 2 \Big[ \frac{g}{2N_j} \Big], \, \, \ell_j = 2 \Big[ \frac{s_j^{3/4}}{2} \Big],$$
and again we choose $K$ such that
$$N_K = \Big[\frac{ \log_q (g \beta)}{\beta}\Big].$$
Then the same proof as before goes through, where in equation \eqref{t0} we use the pointwise bound  \eqref{bd2}  for the $L$--function instead of \eqref{ub_first}.  

\end{proof}

\section{Computing the $1$-level density of zeros}
\label{density_comp}
Here, we will compute the $1$-level density of zeros, defined in \eqref{density}.

 For $u=q^{-s}$, let
$$f(u) = \Phi\Big( - i(s-\tfrac{1}{2}) \frac{ g \log q}{\pi} \Big).$$
Using the expression \eqref{zeros} for the $L$--function and Cauchy's residue theorem, we have that
\begin{align}
\Sigma(\Phi,g) & = \frac{1}{2 \pi i} \oint_{|u|=q^{\alpha-1/2}} \frac{1}{|\mathcal{H}_{2g+1}|} \sum_{D \in \mathcal{H}_{2g+1}}\frac{\mathcal{L}'(u,\chi_D)}{\mathcal{L}(u,\chi_D)}  f(u) \, du  \nonumber \\
&- \frac{1}{2 \pi i} \oint_{|u|=q^{-\alpha-1/2}}  \frac{1}{|\mathcal{H}_{2g+1}|} \sum_{D \in \mathcal{H}_{2g+1}} \frac{\mathcal{L}'}{\mathcal{L}} (u,\chi_D) f(u) \, du,
\label{cauchy}
\end{align}
for $\alpha>0$. We pick $g^{-1/2+\varepsilon} \ll \alpha < 1/2$.

Now let 
\begin{equation}
\mathcal{A}(u,v) = \prod_P \Big(1- (uv)^{d(P)} \Big)^{-1} \Big(1- \frac{u^{2d(P)}}{|P|+1} - \frac{|P| (uv)^{d(P)}}{|P|+1} \Big).
\label{auv}
\end{equation}
We rewrite Theorem \ref{theorem1} in the case of one $L$-function over one $L$-function as follows. For $ |v| \leq q^{-1/2-c_1/g^{1/2-\varepsilon}}$ for some constant $c_1$ and $|u|<q^{-1}$, we have
\begin{align*}
\frac{1}{|\mathcal{H}_{2g+1}|} \sum_{D \in \mathcal{H}_{2g+1}} \frac{\mathcal{L}(u,\chi_D)}{\mathcal{L}(v,\chi_D)} &= \frac{\mathcal{Z}(u^2)}{\mathcal{Z}(uv)} \mathcal{A}(u,v) + (u \sqrt{q})^{2g} \frac{ \mathcal{Z} \big( \frac{1}{q^2u^2} \big)}{ \mathcal{Z} \big(  \frac{v}{qu}\big)} \mathcal{A} \Big( \frac{1}{qu},v \Big) \\
&+ O \Big( \big( |v| \sqrt{q} \big)^{g(5+4 \log_q |u|-\varepsilon)}  \Big).
\end{align*}
Differentiating the above with respect to $u$ and setting $v=u$, we get the following corollary.
\begin{corollary}
\label{cor}
For $|u| \leq q^{-1/2-c_1/g^{1/2-\varepsilon}}$ for some constant $c_1$, we have
\begin{align*}
\frac{1}{|\mathcal{H}_{2g+1}|} \sum_{D \in \mathcal{H}_{2g+1}} \frac{ \mathcal{L}'(u,\chi_D)}{\mathcal{L}(u,\chi_D)} & = u \frac{ \mathcal{Z}'(u^2)}{\mathcal{Z}(u^2)}+ \frac{d}{du} \mathcal{A}(u,v)_{v=u} +\frac{1}{u} (u \sqrt{q})^{2g} \mathcal{Z} \Big(  \frac{1}{q^2u^2}\Big) \mathcal{A} \Big( \frac{1}{qu},u \Big)\\
&+ O \Big( \big( |u| \sqrt{q} \big)^{g(5+4 \log_q |u|-\varepsilon)}  \Big).
\end{align*}
\end{corollary}
Now in equation \eqref{cauchy}, for the first integral, we make the change of variables $u \mapsto \frac{1}{qu}$, use the fact that $f(\tfrac{1}{qu})=f(u)$ and get that
\begin{equation*}
\frac{1}{2 \pi i} \oint_{|u|=q^{\alpha-1/2}} \frac{\mathcal{L}'(u,\chi_D)}{\mathcal{L}(u,\chi_D)}  f(u) \, du = \frac{1}{2 \pi i} \oint_{|u|=q^{-1/2-\alpha}} \frac{\mathcal{L}'( \frac{1}{qu},\chi_D)}{\mathcal{L}(\frac{1}{qu},\chi_D)}  \frac{du}{qu^2}.
\end{equation*}
Now from the functional equation of the $L$--functions, we have that
\begin{equation}
\frac{1}{qu^2} \frac{\mathcal{L}'(\frac{1}{qu},\chi_D)}{\mathcal{L}( \frac{1}{qu},\chi_D)} = - \frac{ \mathcal{L}'(u,\chi_D)}{\mathcal{L}(u,\chi_D)}+ \frac{2g}{u},
\nonumber
\end{equation}
so combining the two equations above, it follows that
\begin{align*}
\frac{1}{2 \pi i} \oint_{|u|=q^{\alpha-1/2}}  \frac{\mathcal{L}'(u,\chi_D)}{\mathcal{L}(u,\chi_D)}  f(u) \, du &= -  \frac{1}{2 \pi i} \oint_{|u|=q^{-1/2-\alpha}}  \frac{\mathcal{L}'(u,\chi_D)}{\mathcal{L}(u,\chi_D)}  f(u) \, du\\
&+ \frac{2g}{2 \pi i} \oint_{|u|=q^{-1/2-\alpha}}  \frac{ f(u)}{u}  \, du.
\end{align*}
Combining the equation above with \eqref{density}, it follows that
\begin{align}
\Sigma(\Phi,g) = -\frac{1}{\pi i}  \oint_{|u|=q^{-1/2-\alpha}}  \frac{1}{|\mathcal{H}_{2g+1}|} \sum_{D \in \mathcal{H}_{2g+1}} \frac{\mathcal{L}'(u,\chi_D)}{\mathcal{L}(u,\chi_D)}  f(u) \, du + \frac{2g}{2 \pi i} \oint_{|u|=q^{-1/2-\alpha}}  \frac{ f(u)}{u}.
\label{density2}
\end{align}
Now note that we have
\begin{equation}
f(u) = \frac{1}{2g}  \sum_{|n| \leq N} \widehat{\Phi} \Big( \frac{n}{2g} \Big) \frac{1}{q^{n/2} u^n}.
\label{fu}
\end{equation}
Using \eqref{fu}, we easily see that
\begin{equation}
\frac{2g}{2 \pi i} \oint_{|u|=q^{-1/2-\alpha}}  \frac{ f(u)}{u} =   \widehat{\Phi}(0).\label{phi0}
\end{equation}
Now for the sum over $D$ in \eqref{density2} we use Corollary \ref{cor}. Note that $|f(u)| \ll q^{N \alpha}$ from equation \eqref{fu}. Using \eqref{phi0}, it follows that
\begin{align}
\Sigma(\Phi,g) &= \widehat{\Phi}(0)- \frac{1}{ \pi i} \oint_{|u|=q^{-1/2-\alpha}} \Big(u \frac{ \mathcal{Z}'(u^2)}{\mathcal{Z}(u^2)}+ \frac{d}{du} \mathcal{A}(u,v)_{v=u}  \nonumber \\
& +\frac{1}{u} (u \sqrt{q})^{2g} \mathcal{Z} \Big(  \frac{1}{q^2u^2}\Big) \mathcal{A} \Big( \frac{1}{qu},u \Big) f(u) \, du + O \Big( q^{-\alpha g (3+2 \alpha-\varepsilon)} \Big) \nonumber \\
&= \widehat{\Phi}(0) +A_1+A_2+A_3 + O \Big( q^{N \alpha-\alpha g (3+2 \alpha-\varepsilon)} \Big),\label{density3}
\end{align}
where $A_1,A_2,A_3$ are obtained by computing the three integrals above. Using \eqref{fu}, we have  
\begin{align*}
 \frac{1}{2 \pi i}  \oint_{|u|=q^{-1/2-\alpha}} u \frac{ \mathcal{Z}'(u^2)}{\mathcal{Z}(u^2)} f(u) \, du  &=  \frac{1}{2 \pi i}  \oint_{|u|=q^{-1/2-\alpha}} \frac{qu}{1-qu^2} \sum_{|n| \leq N} \frac{1}{2g} \widehat{\Phi} \Big( \frac{n}{2g} \Big) \frac{1}{q^{n/2} u^n} \, du \\
 &= \sum_{n=1}^{N/2} \frac{1}{2g}   \widehat{\Phi} \Big( \frac{n}{g} \Big),
 \end{align*}
 so 
 \begin{equation}
 A_1 = - \frac{1}{g}  \sum_{n=1}^{N/2}   \widehat{\Phi} \Big( \frac{n}{g} \Big). \label{a1}
 \end{equation}
From the expression \eqref{auv}, it follows that
\begin{equation*}
\frac{d}{du} \mathcal{A} (u,v)_{v=u} = -\sum_P \frac{ d(P) u^{2d(P)-1}}{(|P|+1)(1-u^{2 d(P)})},
\end{equation*}
so
\begin{align*}
\frac{1}{2 \pi i}&  \oint_{|u|=q^{-1/2-\alpha}} \frac{d}{du} \mathcal{A}(u,v)_{v=u}  f(u) \, du = -\frac{1}{2 \pi i}  \oint_{|u|=q^{-1/2-\alpha}} \sum_P \frac{ d(P) u^{2d(P)-1}}{(|P|+1)(1-u^{2 d(P)})}\\
& \times \sum_{|n| \leq N} \frac{1}{2g} \widehat{\Phi} \Big( \frac{n}{2g} \Big) \frac{1}{q^{n/2} u^n} \, du = -\frac{1}{2g} \sum_{n=1}^{N/2} \widehat{\Phi} \Big( \frac{n}{g}  \Big) \frac{1}{q^n} \sum_{d(P)|n} \frac{d(P)}{|P|+1}.
\end{align*}
Then
\begin{equation}
A_2 = \frac{1}{g} \sum_{n=1}^{N/2} \widehat{\Phi} \Big( \frac{n}{g}  \Big) \frac{1}{q^n} \sum_{d(P)|n} \frac{d(P)}{|P|+1}. 
\label{a2}
\end{equation}
Finally, note that
$$ \mathcal{A} \Big( \frac{1}{qu},u  \Big)= \frac{\zeta_q(2)}{\zeta_q( \frac{1}{q^3u^2}} = \frac{1- \frac{1}{q^2u^2}}{1-q^{-1}},$$ so
\begin{align*}
\frac{1}{2 \pi i} \oint_{|u|=q^{-1/2-\alpha}} &\frac{1}{u} (u \sqrt{q})^{2g} \mathcal{Z} \Big(  \frac{1}{q^2u^2}\Big) \mathcal{A} \Big( \frac{1}{qu},u \Big) f(u) \, du = \frac{1}{q-1} \frac{1}{2 \pi i} \oint_{|u|=q^{-1/2-\alpha}}  \frac{1}{u} (u\sqrt{q})^{2g} \\
& \times  \frac{q^2u^2-1}{qu^2-1} \sum_{|n| \leq N} \frac{1}{2g} \widehat{\Phi} \Big( \frac{n}{2g} \Big) \frac{1}{q^{n/2} u^n} \, du.
\end{align*}
A standard computation gives that 
\begin{align*}
\frac{1}{2 \pi i} \oint_{|u|=q^{-1/2-\alpha}} &\frac{1}{u} (u \sqrt{q})^{2g} \mathcal{Z} \Big(  \frac{1}{q^2u^2}\Big) \mathcal{A} \Big( \frac{1}{qu},u \Big) f(u) \, du  = \frac{1}{q-1} \frac{ \widehat{\Phi}(1)}{2g} - \frac{1}{2g} \sum_{n=g+1}^{N/2} \widehat{\Phi} \Big( \frac{n}{g} \Big),
\end{align*}
so 
\begin{equation}\label{a3}
A_3= - \frac{1}{q-1} \frac{ \widehat{\Phi}(1)}{g} + \frac{1}{g} \sum_{n=g+1}^{N/2} \widehat{\Phi} \Big( \frac{n}{g} \Big).
\end{equation}
Combining equations \eqref{density3}, \eqref{a1}, \eqref{a2}, \eqref{a3}, it then follows that
\begin{align}
\Sigma(\Phi,g) &= \widehat{\Phi}(0) - \frac{1}{g} \sum_{n=1}^g \widehat{\Phi} \Big(  \frac{n}{g} \Big)- \frac{\widehat{\Phi}(1)}{g(q-1)} + \frac{1}{g} \sum_{n=1}^{N/2} \widehat{\Phi} \Big( \frac{n}{g}  \Big) \frac{1}{q^n} \sum_{d(P)|n} \frac{d(P)}{|P|+1}  \nonumber \\
&+ O(q^{N \alpha-\alpha g (3+2\alpha-\varepsilon)}), \label{formula_density}
\end{align}
where 
 $g^{-1/2-\varepsilon} \ll \alpha <1/2$. We pick $\alpha=1/2-\varepsilon$, and then the error term above becomes $q^{N/2-2g+\varepsilon g}$. Hence, we obtain an asymptotic formula for the $1$--level density when $N<4g$. 
 
\kommentar{ \section{Applications to non-vanishing}
\label{non-vanishing}
Here, we will prove Corollary \ref{nonvanish_cor}.
Theorem \ref{theorem_density} yields
$$ \frac{1}{| \mathcal{H}_{2g+1}|} \sum_{D \in \mathcal{H}_{2g+1}} \sum_{j=1}^{2g} \Phi(2g \theta_{j,D}) = \int_{-\infty}^{\infty} \widehat{\Phi}(y) \widehat{W}_{Sp(2g)}(y) \, dy + o(1),$$ where
$$\widehat{W}_{Sp(2g)}(y) = \delta_0(y) - \frac{1}{2} \chi_{[-1.1]}(y).$$
Let
$$p_0(g) = \frac{1}{|\mathcal{H}_{2g+1}|} \Big| \Big\{ D \in \mathcal{H}_{2g+1} : L(1/2,\chi_D) \neq 0 \Big \} \Big|.$$
Similarly as in \cite{BF1} (Section 6), we have that 
$$p_0(g) \geq 1- \frac{1}{2} \int_{-\infty}^{\infty} \widehat{\Phi}(y) \widehat{W}_{Sp(2g)}(y) \, dy+o(1),$$
as $g \to \infty$. 
The Fourier pair
$$ \Phi(x) = \Big( \frac{\sin(5 \pi x/2)}{5 \pi x/2} \Big)^2, \, \, \widehat{\Phi}(y) = \frac{2}{5} \Big(1- \frac{2|y|}{5} \Big), \, \text{ if } |y|<5/2$$ gives
$$ \int_{-\infty}^{\infty} \widehat{\Phi}(y) \widehat{W}_{Sp(2g)}(y) \, dy = \widehat{\Phi}(0) - \frac{1}{2} \int_{-1}^{1} \frac{2}{5} \Big(1 - \frac{2|y|}{5} \Big) \, dy = \frac{1}{25},$$ and hence
$$p_0(g) \geq \frac{24}{25}+o(1).$$ It then follows that more than $96 \%$ of $L(1/2,\chi_D)$ do not vanish.
}

\section{Appendix} 
\begin{lemma}
Let $ a >0 $ with $a = O(1)$. For $\theta \in [0, 2 \pi ]$, let $ \bar{ \theta} =  \min \{ \theta, 2 \pi - \theta \}$. Then
$$ \sum_{n=1}^g  \frac{ \cos(n \theta)}{ n q^{an}} = \log \min \Big\{  \frac{1}{a}, \min \Big\{ g, \frac{1}{\bar{\theta}} \Big\}\Big\}+ O(1).$$
\label{sumscos}
\end{lemma}
\begin{proof}
If $ \bar{\theta}\leq 1/g,$ we write $\cos(n \theta) = 1+ O(n^2 \theta^2)$. Then
\begin{equation}
 \sum_{n=1}^g  \frac{ \cos(n \theta)}{ n q^{an}}  = \sum_{n=1}^g \frac{1}{n q^{an}} + O\big( (g \bar{\theta})^2\big)= \sum_{n=1}^g \frac{1}{n q^{an}} + O(1). 
 \label{ineq1}
 \end{equation}
If $a\leq 1/g$, then
$$ \sum_{n=1}^g \frac{1}{n q^{an}} \leq  \sum_{n=1}^g \frac{1}{n} = \log g+ O(1).$$
On the other hand, $q^{-an} \geq 1-an \log q$, so
$$ \sum_{n=1}^g \frac{1}{n q^{an}} \geq \sum_{n=1}^g \frac{1- an \log q}{n}= \log g +O(1).$$
Hence $$  \sum_{n=1}^g \frac{1}{n q^{an}} = \log g + O(1),$$ and combining the above with \eqref{ineq1} finishes the proof in this case. 

Now assume that $a > 1/g$. Then we write
\begin{equation}
 \sum_{n=1}^g \frac{1}{n q^{an}} = \sum_{n=1}^{[1/a]} \frac{1}{n q^{an}} + \sum_{n=[1/a]+1}^g \frac{1}{n q^{an}}.
 \label{ineq22}
 \end{equation} By a previous argument, we have
\begin{equation}
  \sum_{n=1}^{[1/a]+1} \frac{1}{n q^{an}} = \log \frac{1}{a}  +O(1).
  \label{ineq3}
  \end{equation}
Now
\begin{align}
\sum_{n=[1/a]+1}^{g} \frac{1}{n q^{an}}  \leq a \sum_{n=1/a}^g \frac{1}{q^{an}} \leq  \frac{a}{1-q^{-a}} = O(1),
\label{ineq4}
\end{align} where in the last line we used the fact that $e^x - 1 > x$ for $x >0$.
Combining equations \eqref{ineq22}, \eqref{ineq3} and \eqref{ineq4} finishes the proof in this case.

Now assume that $\bar{\theta} > 1/g.$ Write
\begin{equation}
\sum_{n=1}^g \frac{\cos( n \theta)}{n q^{an}} = \sum_{n=1}^{[1/\overline{ \theta}]} \frac{ \cos( n \theta)}{ n q^{an} } + \sum_{n= [1/\overline{ \theta}]+1}^g \frac{\cos( n \theta)}{n q^{an}}= \log  \min\Big \{  \frac{1}{a},\frac{1}{ \overline{ \theta} } \Big\} + O(1) + \sum_{n= [1/\overline{ \theta}]+1}^g \frac{\cos( n \theta)}{n q^{an}},
\label{ineq5}
\end{equation} where the last equality follows from the previous case.
If $ a \geq \bar{\theta}$, then 
\begin{equation}
\sum_{n=[ 1/ \overline{ \theta}]+1}^g \frac{\cos( n \theta)}{n q^{an}} \leq a \sum_{n= 1}^{\infty} \frac{1}{q^{an}} = \frac{a}{ q^a-1} = O(1), 
\label{ineq6}
\end{equation}
where we have used again the fact that $e^x - 1 > x$ for $x >0$.
Combining equations \eqref{ineq5} and \eqref{ineq6} finishes the proof when $a\geq \bar{\theta}.$ 

Finally assume that $a < \bar{\theta}$. Write
$$ \sum_{n=[ 1/ \overline{ \theta}]+1}^g \frac{\cos( n \theta)}{n q^{an}} = \sum_{n=[1/ \overline{\theta}]+1}^{[1/a]} \frac{\cos(n \theta)}{n q^{an}} + \sum_{n=[1/a]+1}^g \frac{ \cos( n \theta)}{n q^{an}},$$ and let $S_1$ denote the first summand above and $S_2$ the second. Note that 
$$ S_2 \leq  a \sum_{n=1}^{\infty} \frac{1}{q^{an}} = O(1).$$ For $S_1$, note that when $n<1/a$, we have $q^{-an}= 1+ O(an)$. Hence
$$ S_1 =  \sum_{n=[1/\overline{\theta}]+1}^{[1/a]} \frac{\cos(n \theta)}{n} + O(1)= O(1),$$ where the last equality follows by partial summation. This finishes the proof of Lemma \ref{sumscos}.
\end{proof}

\section*{Acknowledgments}
The authors would like to thank B. Conrey and K. Soundararajan for useful discussions, as well as Z. Rudnick for helpful comments on the paper. AF gratefully acknowledges support from NSF grant DMS-2101769 and NSF Postdoctoral Fellowship DMS-1703695 while working on this paper. JPK is pleased to acknowledge support from ERC Advanced Grant 740900 (LogCorRM).

\end{document}